\documentclass[11pt]{amsart}

\usepackage{amsthm,amsmath,amssymb}
\usepackage[T1]{fontenc}
\usepackage[utf8]{inputenc}
\usepackage[english]{babel}
\usepackage{indentfirst}
\usepackage[section]{placeins}
\usepackage{graphicx,comment}
\graphicspath{{images/}}
\usepackage[dvipsnames]{xcolor}
\usepackage[pagebackref,breaklinks,colorlinks=true,linkcolor=MidnightBlue,citecolor=MidnightBlue]{hyperref}
\usepackage[shortlabels]{enumitem}
\usepackage{tikz}

\usepackage{mathFL,mathMJ}

\newcommand{\normhone}[1]{\norm{#1}_{\dot{H}^1}}
\newcommand{\normhmone}[1]{\norm{#1}_{\dot{H}^{-1}}}
\newcommand{\normH}[1]{\norm{#1}_{\calH}}

\newcommand{\babs}[1]{\big\lvert #1 \big\rvert}

\title{A fast approach to optimal transport:\\
The back-and-forth method}

\date{\today}

\author{Matt Jacobs}
\address{UCLA, Los Angeles, CA, USA}
\email{majaco@math.ucla.edu}

\author{Flavien Léger}
\address{UCLA, Los Angeles, CA, USA}
\email{flavienleger@nyu.edu}

\thanks{M.J. and F.L. are supported by  AFOSR MURI FA9550-18-1-0502 and ONR N00014-12-1-0838. M.J. is also supported by NSF DMS-1118971, DARPA FA8750-18-2-0066, and ONR N00014-18-1-2527}


\begin{document}
\begin{abstract}
We present an iterative method to efficiently solve the optimal transportation problem for a class of strictly convex costs which includes quadratic and $p$-power costs. Given two probability measures supported on a discrete grid with $n$ points, we compute the optimal map using $O(n)$ storage space and $O(n\log(n))$ operations per iteration, with an approximately exponential convergence rate.  Our approach allows us to solve optimal transportation problems on spatial grids as large as $4096\times 4096$ and $384\times 384\times 384$ in a matter of minutes. 
\end{abstract}
\maketitle

\section{Introduction}\label{sec:intro}

The optimal transportation problem was first introduced by Monge in 1781, to find the most cost-efficient way to transport mass from a set of sources to a set of sinks.  The theory was modernized and revolutionized by Kantorovich in 1942, who found a key link between optimal transport and linear programming.  In recent years, there has been an explosion of interest in optimal transport thanks in part to the discovery of deep connections between the quadratic-cost optimal transport problem and a diverse class of partial differential equations (PDEs) arising in statistical mechanics and fluid mechanics  (\cite{brenier_polar, brenier_euler_assignment, bb00, otto_porous_media, jko98} to name just a few of the most prominent results). In addition, optimal transport has become popular in data science (particularly machine learning and image processing), where it provides a very natural way to compare and interpolate probability distributions~\cite{McCann_interpolation,  HZTA,WassersteinGAN}.

In this work, we are interested in computing the optimal transport problem for large-scale applications that arise in image processing and in numerical methods for solving PDEs (in both two and three dimensions).  For these applications, one needs to accurately compute the optimal transport map on enormous computational domains (millions of grid points or pixels).   Up to now, computing the optimal map has been a notoriously difficult task. To the best of our knowledge, all previously known methods for solving the optimal transport problem either do not scale linearly with respect to the problem size \cite{bertsekas_transportation_auctions, bb00,Prins2015,LindseyRubinstein}, cannot provide an accurate computation of the optimal map \cite{cuturi_sinkhorn}, or are only applicable to a limited class of probability densities \cite{benamou_froese_oberman}.  Semi-discrete optimal transport methods~\cite{Levy3D,Merigot2011} have been applied successfully to a wide range of optimal transport problems~\cite{GallouetMerigot,MerigotMirebeau}, but require expensive geometric processing overhead. 

Our goal in this paper is to remedy the current situation and provide a simple, efficient, and accurate algorithm for computing optimal transport maps. To that end, we introduce a new method for solving optimal transport problems: the back-and-forth method.    Given two probability densities $\mu$ and $\nu$ discretized on a grid with $n$ points and a strictly convex cost function, our approach computes the optimal map using $O(n)$ storage space and $O(n\log(n))$ computation time per iteration.  The number of iterations required for $\epsilon$ accuracy grows like $O\Big(\max(\norm{\mu}_{\infty}, \norm{\nu}_{\infty})\log(\frac{1}{\epsilon})\Big)$ in all tested experiments.  As a result, the method converges extremely rapidly.  Notably, in contrast to many other approaches, the back-and-forth method does not require positive lower bounds on the probability densities.  Thus, we can solve optimal transport problems with densities that vanish on large portions of the domain.  In total, the efficiency and flexibility of the back-and-forth method allows us to solve optimal transport problems on grids of size $4096\times 4096$ and $384\times 384\times 384$ in a matter of minutes on a personal computer.  Furthermore, the method could be accelerated even further by applying parallelization or other standard scientific computing techniques such as multigrid acceleration.

\subsection{Overall approach}

The back-and-forth method is based on solving the Kantorovich dual formulation of optimal transport.  To introduce this formulation, suppose we are given two probability measures $\mu$ and $\nu$ and a cost function $c(x,y)$, which measures the cost to move a unit of mass from location $x$ to location $y$ (for now we shall be deliberately vague about underlying spaces). 
The optimal cost to transport $\mu$ to $\nu$ is given by the value of the Kantorovich dual problem:
\begin{equation*}
\sup_{\phi,\psi} \int \phi\,d\nu + \int \psi\,d\mu,
\end{equation*}
where the Kantorovich potentials $\phi$ and $\psi$ are two scalar functions constrained by the inequality
\begin{equation}\label{eq:c_constraint}
\phi(y) + \psi(x) \le c(x,y) .
\end{equation}
In addition to computing the optimal cost, the dual problem also encodes information about the optimal map.
If an optimal map from $\mu$ to $\nu$ exists (or vice-versa), it can be recovered from the maximizers $\phi_*, \psi_*$ of the dual problem.  Indeed, in this case, the pair $(x,y)$ is in the graph of the map if and only if $\phi_*(y)+\psi_*(x)=c(x,y).$  Thus, to compute optimal maps it is enough to solve the dual problem.  

In what follows, we will consider two equivalent unconstrained formulations of the dual problem.  These formulations are based upon the observation that if $\phi_*$ and $\psi_*$ are the maximizers of the dual problem, then one must have the relations
\[
\phi_*(y) = \psi_*^c(y) := \inf_{x} c(x,y) - \psi_*(x),
\]
and 
\[\psi_*(x) = \phi^c_*(x) := \inf_{y} c(x,y)-\phi_*(y).\]
Here we make use of an operator $\phi\to\phi^c$ (resp. $\psi\to\psi^c$) called the $c$-transform, which plays an important role in optimal transport and in our proposed method.  Note that if one views the dual problem as a linear programming problem, then the above two equations are precisely the complementary slackness condition.

Now we see that there are two equivalent ways to remove the constraint (\ref{eq:c_constraint}), we can either replace $\phi$ by $\psi^c$ or we can replace $\psi$ by $\phi^c$. 
This leads to the twin functionals
\begin{equation*}
J(\phi) = \int \phi\,d\nu + \int \phi^c\,d\mu,
\end{equation*}
and 
\begin{equation*}
I(\psi) = \int\psi^c\,d\nu + \int \psi\,d\mu ,
\end{equation*}
which encode exactly the same problem, just posed in different spaces.  $J$ formulates the problem in ``$\phi$-space'' and $I$ formulates the problem in ``$\psi$-space''.  Here one can draw an analogy to the Fourier transform: indeed, the reader is probably already quite familiar with functionals that can be equivalently expressed in either physical space or Fourier space.  

The back-and-forth method solves the Kantorovich dual problem by hopping back-and-forth between gradient ascent updates on $J$ in $\phi$-space and gradient ascent updates on $I$ in $\psi$-space (hence the name). Gradient are taken with respect to the $\dot{H}^1$ metric (see~\eqref{eq:defH1dot}). In between gradient steps, information in one space is propagated back to the other space by taking a $c$-transform (c.f. Algorithm~\ref{algo:main}).  

The advantage of the back-and-forth approach is that certain features of the optimal solution pair $(\phi_*, \psi_*)$ may be easier to build in one space compared to the other.  For example, the Hessian of $\psi=\phi^c$ is very closely related to the inverse of the Hessian of $\phi$.  Thus, large Hessian eigenvalues in one space correspond to small Hessian eigenvalues in the other space, and smaller features can be built in fewer gradient steps.  By hopping back-and-forth between the two spaces, we get the best of both worlds --- there is always an opportunity to build features in the space where they are smaller.  As a result, the back-and-forth method converges far more rapidly than vanilla gradient ascent methods that operate only on $\phi$-space or only on $\psi$-space.  
Indeed, on certain examples, our method produces a 10,000-fold reduction in error in just 2 to 4 additional iterations (c.f. Tables \ref{table:2d_2_discs}, \ref{table:3d_2_balls}). 

In addition, let us highlight that the choice of gradient ascent steps in the $\dot{H}^1$ metric appears to be crucial for maintaining the stability of the algorithm. Indeed, it seems that it is not possible to guarantee an increase in the value of the dual problem if the gradient ascent steps are taken in any Hilbert space weaker than $\dot{H}^1$ (see Section \ref{ssec:gd} and Proposition \ref{prop:ascent-descent} for more details).

Turning to the computational efficiency of the back-and-forth method, we see that the scheme amounts to performing $\dot{H}^1$ gradient ascent iterations alternatively on $J$ and $I$ (i.e. in $\phi$-space and $\psi$-space), and computing $c$-transforms. Crucially, the derivative of $J$ (and, by symmetry, the derivative of $I$) takes a simple form~\cite{gangbo_habilitation} which can be efficiently computed. 
Furthermore, for a large class of costs, the $c$-transform can be computed extremely efficiently.   On a discrete grid with $n$ points, an exact $c$-transform can be computed in  $O(n\log(n))$ operations \cite{corrias_legendre, lucet_legendre} (see Section \ref{ssec:numeric_c_transform} for more details).  

Finally, let us note that a gradient ascent method on the Kantorovich dual problem was previously considered in \cite{cwbv_mk_gd}. In their paper, the authors took $L^2$ gradient ascent steps on only one form of the dual problem.  Furthermore, the authors tried to avoid computing $c$-transforms, and thus were forced to use an inaccurate estimation of the true gradient.  Unsurprisingly, the authors note that their scheme appears to be unstable.  Indeed, we believe that the ideas behind the back-and-forth method are necessary to obtain a robust, efficient, and accurate gradient ascent type method for solving the Kantorovich dual problem.

\subsection{Future work and paper outline}

While the present work is focused on the optimal transport problem, we anticipate that the back-and-forth method will prove useful for many other problems.  For instance, the back-and-forth method can be readily adapted to compute Wasserstein gradient flows.  This opens the door to large-scale simulations of a wide class of important and interesting PDEs, \cite{otto_porous_media, jko98,brenier_euler_assignment,jacobs_kim_meszaros} to name just a few. In addition, the method may prove useful to solve computational problems arising from the burgeoning area of mean field games \cite{Lasry_Lions_MFG,HMC_MFG}.  We look forward to exploring these applications in future work.      

The rest of the paper is organized as follows.  In Section 2, we recall the necessary background information for optimal transport and gradient-based optimization schemes. In Section 3, we introduce the back-and-forth method and provide arguments for its stability and efficiency.  In Section 4, we discuss the numerical implementation of the algorithm, and conduct various experiments to demonstrate its performance.

\section{Background}
\label{sec:background}

\subsection{Optimal transport and the \texorpdfstring{$c$}{c}-transform}
\label{sec:background:ot}

Let $\Omega$ be a convex and compact subset of $\Rd$. A \emph{cost} on $\Om$ is a continuous function $c\colon\Om\times\Om\to\R$. In the theory of optimal transport fairly general costs can be considered, however in the present work we will focus on the case
\[
c(x,y) = h(y-x),
\]
for a strictly convex and even function $h\colon\Rd\to\R$. Given two probability measures $\mu$ and $\nu$ supported in $\Om$, the Monge formulation of the optimal transport problem is defined by
\begin{equation*}
C(\mu,\nu) = \inf_{T} \int_{\Om} c\big(x, T(x)\big)\,d\mu(x),
\end{equation*}
where the infimum runs over maps $T\colon\Om\to\Om$ which transport $\mu$ to $\nu$, i.e. such that $T_{\#}\mu = \nu$. We recall that the pushforward measure $T_{\#}\mu$ is defined by $T_{\#}\mu(A)=\mu(T^{-1}(A))$ for any measurable subset $A\subset\Om$.   One can also characterize the pushforward by defining the integral of the pushforward measure against continuous test functions $f\colon\Omega\to\R$:
$$  \int_{\Omega} f(y) \,d(T_{\#}\mu)(y)=\int_{\Omega} f(T(x))\,d\mu(x).$$
This second formulation is extremely useful for studying pushforwards in the context of optimal transport. 

Throughout the paper, we will concentrate on the special case where both probability measures $\mu$ and $\nu$ are absolutely continuous with respect to the Lebesgue measure. As a result, we will frequently conflate a measure and its density function. Under this assumption, there exists a unique optimal map $T_*$, which pushes $\mu$ to $\nu$, and its inverse $T^{-1}_*$ is the optimal map that pushes $\nu$ to $\mu$ ~\cite{brenier_polar, gangbo_habilitation}.   Furthermore, one can find the optimal map by solving the Kantorovich dual formulation of the optimal transport problem. 

The dual formulation can be derived by introducing Lagrange multipliers for the pushforward constraint.  Indeed, the pushforward constraint $T_{\#}\mu=\nu$ holds if and only if 
$$\int_{\Omega} \phi\big(T(x)\big)\,d\mu(x)=\int_{\Omega} \phi(y)\,d\nu(y)$$
for every continuous function $\phi$. Therefore,
$$\inf_{T_{\#}\mu=\nu} \int_{\Om} c\big(x, T(x)\big)\,d\mu(x)=\inf_{T}\sup_{\phi} \int_{\Om} c\big(x, T(x)\big)\,d\mu(x) - \phi\big(T(x)\big)\,d\mu(x) + \int_{\Om}\phi(y)\,d\nu(y).$$
Under the assumption that $\mu$ is absolutely continuous, the interchange of infimum and supremum is valid \cite{VillaniBook1, santambrogio_otam}. If we then group the terms involving $T$ we arrive at 
$$ \sup_{\phi} \inf_{T} \int_{\Om} \Big(c\big(x, T(x)\big)-\phi\big(T(x)\big)\Big)\,d\mu(x) +\int_{\Om}\phi(y)\,d\nu(y).$$ 
Now it is clear that the optimal choice of $T$ at some point $x_0$ is completely decoupled from the choice at any other point $x$.  Therefore, one can move the infimum inside of the integral, and it becomes clear that the operation $\inf_{T(x)} c\big(x, T(x)\big)-\vp\big(T(x)\big) $ plays an important role in the dual problem.  This operation is known as the $c$-transform, and it is at the heart of optimal transport. The $c$-transform maps a function $\phi$ to another function $\phi^c$ and it can be seen as a generalization of the Legendre transform (convex conjugation) from convex analysis.
\begin{definition}
Given a continuous function $\phi\colon\Om\to\R$, we define its $c$-transform $\phi^c\colon\Om\to\R$ by
\[
\phi^c(x) = \inf_{y\in\Om} c(x,y) - \phi(y) .
\]
Additionally, we say that $\phi$ is $c$-concave if there exists a continuous function $\psi\colon\Om\to\R$ such that $\phi = \psi^c$ and we say that a pair of functions $(\phi, \psi)$ is $c$-conjugate if $\phi=\psi^c$ and $\psi=\phi^c$. 
\end{definition}

We are now ready to introduce the Kantorovich dual functional 
\begin{equation*}
J(\phi) = \int_{\Om} \phi\,d\nu + \int_{\Om} \phi^c\,d\mu.
\end{equation*}
Our above discussion then shows that
\[
C(\mu,\nu) = \sup_{\phi} J(\phi) ,
\]
where the supremum runs over continuous functions $\phi\colon \Om\to\R$.  Note that in this section we will just focus on one formulation of the dual problem.  It is easy to see that any properties satisfied by $J$ must also be satisfied by the alternative formulation, which we previously denoted in the introduction as $I$.

The following lemma encapsulates the most important properties of the $c$-transform (from the perspective of optimal transport).
\begin{lemma}[\cite{GangboMcCann,gangbo_habilitation}]
\label{lem:c_transform_properties}
Let $\phi :\Omega\to\R$ be a continuous function.  The $c$-transform satisfies the following properties:
\begin{enumerate}[(i)]
\item $\phi(x)\leq \phi^{cc}(x)$ for every $x\in\Omega$, and $\phi^{cc}=\phi$ if and only if $\phi$ is $c$-concave.   As a result, $\phi^{ccc}=\phi^{c}$ for any continuous function $\phi$. 

\item If $\phi$ is $c$-concave, the minimization problem $\inf_{y\in\Omega} c(x,y)-\phi(y)$ has a unique minimizer $T_{\phi}(x)$ for almost every $x$. Furthermore, we have the explicit formula: 
$$T_{\phi}(x)=x-(\nabla h)^{-1}(\nabla \phi^c(x)),$$
where we recall that $c(x,y)=h(y-x)$.
\item \label{lem:c_transform_properties:derivative} If   $\phi$ is $c$-concave  and $u$ is a continuous function on $\Om$, then
$$\lim_{\epsilon\to0} \frac{(\phi+\epsilon u)^c(x)-\phi^c(x)}{\epsilon}=-u(T_{\phi}(x))$$
for almost every $x\in\Omega$.

\end{enumerate}  
\end{lemma}

With these properties in hand, one can establish one of the most fundamental results in optimal transport: 

\begin{theorem}[\cite{brenier_polar, GangboMcCann, gangbo_habilitation}]
\label{thm:dual_problem}
The dual problem 
$$\sup_{\phi} \int_{\Om} \phi\,d\nu + \int_{\Om} \phi^c\,d\mu =: J(\phi)$$
satisfies the following properties:
\begin{enumerate}[(i)]
\item $J$ is concave with respect to $\phi$.
\item  If $\mu$ is absolutely continuous with respect to the Lebesgue measure, then $J$ is maximized by a $c$-concave function $\phi_*$ and $T_{\phi_*}$ is the unique optimal map which pushes $\mu$ to $\nu$.
\item If $\nu$ is also absolutely continuous, then $T_{\phi_*}$ is invertible almost everywhere, and $T_{\phi_*}^{-1}$ is the optimal map which pushes $\nu$ to $\mu$.
\end{enumerate}
\end{theorem}

From Theorem \ref{thm:dual_problem}, we see that the optimal map can be computed by solving a concave maximization problem.  Thus standard techniques of optimization, such as gradient ascent, can be leveraged to solve the optimal transport problem.  However, in order to use gradient ascent effectively, one must choose the correct notion of distance and step size.  For this reason, we review gradient ascent in Hilbert spaces below.

\subsection{Gradient Ascent}
In this section we recall elementary results on constant step-size gradient ascent methods for concave functions. Let $(\calH, \norm{\cdot}_{\calH})$ be a separable Hilbert space and suppose that $F$ is a smooth convex functional
$$F\colon\mathcal{H}\to\R.$$
We first recall the notions of differential map and gradients.
\begin{definition}
Given a point $\phi\in\calH$ we say that a bounded linear map $\delta F_{\phi}\colon\calH\to\R$ is the first variation (Fréchet derivative) of $F$ at $\phi$ if 
$$  \lim_{\norm{h}_{\calH}\to0} \frac{\norm{F(\phi+h)-F(\phi)-\delta F_{\phi}(h)}_{\calH}}{\norm{h}_{\calH}}=0. $$
\end{definition}

\begin{definition}
Let $\bracket{\cdot,\cdot}_{\calH}$ be the inner product associated with the Hilbert space $\calH$. We say that a map $\nabla_{\calH} F\colon\calH\to\calH$ is the $\calH$-gradient of $F$ if
\[
\bracket{\nabla_{\calH} F(\phi),h}_{\calH} = \delta F_{\phi}(h)
\]
for all $(\phi,h)\in \calH\times\calH$. 
\end{definition}

In the back-and-forth method, we will make use of the $\dot{H}^1$-gradient where $\dot{H}^1(\Omega)$ is the Hilbert space
\begin{equation}
\label{eq:defH1dot}
\dot{H}^1(\Omega)=\{\vp\colon\Omega\to \R: \int_{\Omega} \vp(x)\, dx=0\quad\text{and} \quad \int_{\Omega} |\nabla \vp(x)|^2\, dx<\infty\},
\end{equation}
with inner product
\[
\bracket{\vp_1, \vp_2}_{\dot{H}^1}=\int_{\Omega} \nabla \vp_1(x)\cdot \nabla \vp_2(x)\, dx.
\]
Let us also mention the dual space $\dot{H}^{-1}(\Om)$: for two probability densities $\rho_1$ and $\rho_2$ we define
\begin{equation}
\label{eq:defHm1dot}
\norm{\rho_2 - \rho_1}_{\dot{H}^{-1}}^2 = \int_{\Om}\abs{\nabla\phi(x)}^2\,dx,
\end{equation}
where $\phi$ is the unique solution in $\dot{H}^1(\Om)$ to the Laplace equation $-\laplacian\phi = \rho_2-\rho_1$ with zero Neumann boundary conditions (for more details we refer to~\cite[Section~5.5.2]{santambrogio_otam}). 

The following lemma shows that the $\dot{H}^1$-gradient has a particularly simple form.
The proof is a straightforward application of integration by parts.
\begin{lemma}
Suppose that $F\colon \dot{H}^1(\Om)\to\R$ is a Fr\'echet-differentiable function such that for any $\phi\in \dot{H}^1(\Om)$, the first variation $\delta F_{\phi}$ evaluated at any point $h\in \dot{H}^1(\Om)$ can be written as  integration against a  function $f_{\phi}$, i.e.
\[
\delta F_{\phi}(h)=\int_{\Omega} h(x) f_{\phi}(x) \, dx.
\]
Then the $\dot{H}^1$-gradient of $F$ has the form
\[
\nabla_{\!\dot{H}^1} F(\phi)=(-\Delta)^{-1}\bar{f}_{\phi},
\]
where $(-\Delta)^{-1}$ is the inverse negative Laplacian operator taken with zero Neumann boundary conditions and $\bar{f}_\phi=f_{\phi}-\frac{1}{|\Omega|}\int_{\Omega} f_{\phi}$.
\end{lemma}

Now that we have seen an example of a $\calH$-gradient, we are ready to discuss gradient ascent methods on general Hilbert spaces.
Gradient ascent maximizes a Fr\'echet-differentiable functional $F\colon\calH\to\R$ by iterating
\[
\phi_{n+1} = \phi_n +\sigma\nabla_{\calH} F(\phi_n),
\]
where the step-size $\sigma>0$ is a fixed constant. In order to obtain convergence to the supremum, i.e.
\[
F(\phi_n)\to \sup F,
\]
one needs to control the continuity of the gradient mapping $\nabla_{\calH} F$.  The following convergence theorem for gradient ascent is one of the cornerstones of optimization.

\begin{theorem}[See for instance~\cite{Nesterov_book}]
\label{thm:gradient_ascent}
Let $F\colon\calH\to\R$ be a Fréchet-differentiable concave functional and suppose that there exists $\sigma>0$ such that 
\[
F(\phi)\geq F(\hat{\phi}) + \delta F_{\hat{\phi}}(\phi-\hat{\phi})
-\frac{1}{2\sigma}\norm{\phi-\hat{\phi}}^2_{\calH}
\]
for all $\phi,\hat{\phi}\in\calH$. Then the gradient ascent iterations
\[
\phi_{n+1} = \phi_n +\sigma\nabla_{\calH} F(\phi_n)
\]
 satisfy the ascent property
\[
F(\phi_{n+1}) - F(\phi_n) \geq \frac{\sigma}{2}\normH{\nabla_{\calH} F(\phi_n)}^2.
\]
Furthermore,  if $F$ has a unique maximizer $\phi_*$ and if $\sup_n \normH{\phi_n}<\infty$ then the sequence $\{\phi_n\}_{n=0}^{\infty}$ converges weakly to $\phi_*$.
\end{theorem}

For completeness, we give a short proof of Theorem \ref{thm:gradient_ascent} in the appendix. With these tools in hands we are now ready to introduce the back-and-forth method.

\section{The back-and-forth method } \label{sec:main}

Let $\Om\subset\Rd$ be a convex and compact region and consider a cost $c\colon\Om\times\Om\to\R$ of the form
\[
c(x,y) = h(y-x),
\]
for a strictly convex and even function $h\colon\Rd\to\R$. Given two probability densities $\mu$ and $\nu$ supported in $\Om$, we are interested in the optimal transport problem in its dual form
 \begin{equation}
\label{eq:dual}
C(\mu,\nu) = \sup_{\phi} \int_{\Om} \phi(y)\, \nu(y)\,dy + \int_{\Om} \phi^c(x)\,\mu(x)\,dx,
\end{equation}
where the supremum runs over continuous functions $\phi\colon\Om\to\R$, and the $c$-transform is defined by $\phi^c(x) = \inf_{y\in\Om} c(x,y) - \phi(y)$. For more background on the optimal transport problem and on $c$-transforms we refer to Section~\ref{sec:background:ot}. 

Our goal in this section is to develop an efficient algorithm to compute the maximizer of the dual problem.  In what follows, we will consider the dual functional in the two following equivalent forms:
\begin{equation}
\label{eq:defJ}
J(\phi) = \int_{\Om} \phi\,d\nu + \int_{\Om} \phi^c\,d\mu
\end{equation}
and
\begin{equation}
\label{eq:defI}
I(\psi) = \int_{\Om} \psi^c\,d\nu + \int_{\Om} \psi\,d\mu.
\end{equation}
Note that the functionals $J$ and $I$ are essentially identical, the only difference is that the roles of $\mu$ and $\nu$ are flipped.
In order to proceed further, we need formulas for the variations of $J$ and $I$.  
\begin{lemma}[\cite{gangbo_elementary_polar, gangbo_habilitation, GangboMcCann}]
\label{fact:derivative-J}
Consider the functionals $J$ and $I$ defined by~\eqref{eq:defJ} and~\eqref{eq:defI}  over the space of continuous functions $\phi\colon\Om\to\R$ and $\psi\colon\Om\to\R$ respectively.
If $\phi$ is $c$-concave, the first variation of $J$ can be expressed as
\[
\delta J_{\phi} = \nu - T_{\phi\,\#}\mu,
\]
and if $\psi$ is $c$-concave, the first variation of $I$ can be expressed as
\[
\delta I_{\psi} = \mu - T_{\psi\,\#}\nu.
\]
Here we recall that for any $c$-concave function $\vp:\Om\to\R$,
$$T_{\vp}(x)=x-(\nabla h)^{-1}(\nabla\vp^c(x)).$$
\end{lemma}
A short proof of Lemma~\ref{fact:derivative-J} is provided in the appendix for the reader's convenience.

Now we are ready to introduce our approach to solve the optimal transport problem: the back-and-forth method.  The method is outlined in Algorithm~\ref{algo:main}. It is based on two main ideas:
\begin{enumerate}[1.]
\item Gradient ascent steps in the $\dot{H}^1$ metric, where
\begin{equation}
\begin{aligned}
\label{eq:gradients_J_I}
\nabla_{\!\dot{H}^1} J(\phi)=(-\Delta)^{-1}\big(\nu - T_{\phi\,\#}\mu\big), \\ \nabla_{\!\dot{H}^1} I(\psi)=(-\Delta)^{-1}\big(\mu - T_{\psi\,\#}\nu\big).
\end{aligned}    
\end{equation}
\item A back-and-forth update scheme, alternating between gradient ascent steps on $J$ and $I$.

\end{enumerate}
\begin{algorithm2e}[!h]
\caption{The back-and-forth method} \label{algo:main}
Given probability densities $\mu$ and $\nu$, set $\phi_0=0, \psi_0=0$, and iterate:
\vspace{.1 in}
\begin{align*}
\phi_{n+\frac{1}{2}} 
&=\phi_n+\sigma\nabla_{\!\dot{H}^1}J(\phi_n), \\
\psi_{n+\frac{1}{2}} &= (\phi_{n+\frac{1}{2}})^c, \\
~\\
\psi_{n+1} 
&=\psi_{n+\frac{1}{2}}+\sigma\nabla_{\!\dot{H}^1} I(\psi_{n+\frac{1}{2}}), \\
\phi_{n+1} &= (\psi_{n+1})^c .
\end{align*}
\end{algorithm2e}
In the following two subsections, we will motivate the choice of $\dot{H}^1$-gradient steps and the back-and-forth updates.   For information about the step size $\sigma$, see Section \ref{sec:numerics:steps-sizes}.

\subsection{\texorpdfstring{$\dot{H}^1$}{H1}-gradient ascent}\label{ssec:gd}
The main steps of the back-and-forth method are the $\dot{H}^1$-gradient ascent steps 
\[
\phi_{n+\frac{1}{2}} =\phi_n+\sigma\nabla_{\!\dot{H}^1}J(\phi_n),
\]
and
\[
\psi_{n+1} =\psi_{n+\frac{1}{2}}+\sigma\nabla_{\!\dot{H}^1} I(\psi_{n+\frac{1}{2}}).
\]
In order to obtain the convergence of the back-and-forth method, we will need to know if these steps increase the values of the dual functionals $J$ and $I$ respectively. In what follows, we will focus exclusively on the $J$ update step, since any properties that hold for $J$ will also hold for $I$ by symmetry.

Recalling Theorem \ref{thm:gradient_ascent} from the previous section, one can show that for a general Hilbert space $\mathcal{H}$, the $\calH$-gradient ascent steps
\[
\phi_{n+1}=\phi_n+\sigma\nabla_{\calH}J(\phi_n)
\]
satisfy the ascent property 
\[
J(\phi_{n+1})\geq J(\phi_n)+\frac{\sigma}{2}\normH{\nabla_{\calH}J(\phi_n)}^2
\]
if the inequality
\begin{equation}\label{eq:suff_ineq}
J(\phi)\geq J(\hat{\phi})+\delta J_{\hat{\phi}}(\phi-\hat{\phi})-\frac{1}{2\sigma}\normH{\phi-\hat{\phi}}^2.
\end{equation}
holds for some $\sigma>0$ and any $\phi,\hat{\phi}\in\calH$.  
Note that this inequality becomes easier to satisfy when the norm associated to $\calH$ becomes stronger i.e. when right hand side becomes more negative.   

To establish the inequality, we need to bound from below
\[
J(\phi)-J(\hat{\phi})-\delta_{\hat{\phi}}J(\phi-\hat{\phi}),
\]
which can be recognized as the error in approximating $J$ with a first-order Taylor expansion about $\hat{\phi}$.   The easiest way to make progress is to use the concave inequality $J(\hat{\phi})\leq J(\phi)+\delta J_{\phi}(\hat{\phi}-\phi)$ to get the bound
\[
J(\phi)-J(\hat{\phi})-\delta_{\hat{\phi}}J(\phi-\hat{\phi})\geq \big(\delta J_{\phi}-\delta J_{\hat{\phi}}\big)(\phi-\hat{\phi}).
\]
 We can use the explicit formula for the first variation in Lemma \ref{fact:derivative-J}, to write 
 \[
 \big(\delta J_{\phi}-\delta J_{\hat{\phi}}\big)(\phi-\hat{\phi})=\int_{\Omega}(\phi-\hat{\phi})(T_{\hat{\phi}\,\#}\mu-T_{\phi\,\#}\mu).
\]
Now inequality (\ref{eq:suff_ineq}) will follow if we can show that 
\[
\int_{\Omega}(\phi-\hat{\phi})(T_{\hat{\phi}\,\#}\mu-T_{\phi\,\#}\mu)\geq -\frac{1}{2\sigma}\normH{\phi-\hat{\phi}}^2.
\]
Thus, our goal is to find a Hilbert space $\calH$ and a parameter $\sigma$ that make this inequality true. 

For any Hilbert space $\calH$, we have the Cauchy--Schwartz inequality
\[
 \int_{\Omega}(\phi-\hat{\phi})(T_{\hat{\phi}\,\#}\mu-T_{\phi\,\#}\mu)\geq - \normH{\phi-\hat{\phi}}\norm{T_{\hat{\phi}\,\#}\mu-T_{\phi\,\#}\mu}_{\calH^*},
\]
where $\calH^*$ is the dual space to $\calH$ with respect to the $L^2$ inner product (i.e. we dualize $\calH$ with respect to the so-called \emph{pivot space} $L^2(\Om)$). If we can show that for an appropriate choice of Hilbert space $\calH$,
\begin{equation}\label{eq:ineq_reduc}
 \norm{T_{\hat{\phi}\,\#}\mu-T_{\phi\,\#}\mu}_{\calH^*}\leq \frac{1}{2\sigma}\normH{\phi-\hat{\phi}},
\end{equation}
 then combining this with our bound from Cauchy--Schwartz we will get (\ref{eq:suff_ineq}) (note a more careful argument can eliminate the factor of $\frac{1}{2}$ on the right side of (\ref{eq:ineq_reduc}), but we won't worry about this here).   Once again, inequality (\ref{eq:ineq_reduc}) becomes easier to satisfy as the norm associated to $\calH$ becomes stronger.  Indeed, as $\calH$ becomes stronger $\calH^*$ becomes weaker, thus the left hand side gets smaller while the right hand side gets larger.  

We can interpret inequality (\ref{eq:ineq_reduc}) as saying that the function $\phi\mapsto T_{\phi\,\#}\mu$ needs to be Lipschitz continuous as a map from $\calH$ to $\calH^*$. Thus, to get a gradient ascent scheme with the ascent property, we will need to choose a Hilbert space $\mathcal{H}$ so that if $\phi$ and $\hat{\phi}$ are close in $\mathcal{H}$ then $T_{\phi\,\#}\mu$ and $T_{\hat{\phi}\,\#}\mu$ are close in $\calH^*$. This is not an easy task. For example, if $\phi$ is not $c$-concave, the map $\phi\mapsto T_{\phi}$ may not even be well-defined, let alone have certain continuity properties.   


If one tries to deal with the above complication head-on, it is not clear how to proceed.  To make progress, we show that an ascent property holds in the Hilbert space $\dot{H}^1(\Omega)=\{\phi\colon\Omega\to\R: \int_{\Om}\phi=0\text{ and }\norm{\nabla \phi}_{L^2(\Omega)}<\infty\}$ if we restrict our attention to the quadratic cost and certain well-behaved $c$-concave functions. Before we state our result, let us note that the choice of $\dot{H}^1$ probably cannot be weakened.  Indeed, the formula for $T_{\phi}$ depends on $\nabla \phi^c$, thus one must have some control on the gradients of $\phi$ and $\hat{\phi}$ to have any hope of showing either (\ref{eq:suff_ineq}) or (\ref{eq:ineq_reduc}).

\begin{prop}
\label{prop:ascent-descent}
Consider the Kantorovich dual problem $\sup_{\phi} J(\phi)$
with the quadratic cost $h(y-x)=\frac{1}{2}|x-y|^2$.
If we know that there exists $\lambda >0$ such that  the iterates
\[
\phi_{n}, \phi_{n+\frac{1}{2}}\in \mathcal{S}_{\lambda}:=\{\phi: (1-\lambda^{-1})I \le D^2\phi(x) \le (1-\lambda)I\}
\]
for every $n$,
then the gradient ascent step
\[
\phi_{n+\frac{1}{2}}=\phi_n+\sigma\nabla_{\!\dot{H}^1}J(\phi_n)
\]
with step size $\sigma = \normlinf{\mu}^{-1}\lambda^{d+1}$ satisfies the ascent property
\[
J(\phi_{n+\frac{1}{2}}) - J(\phi_n) \ge \frac{1}{2}\normlinf{\mu}^{-1}\lambda^{d+1} \normhone{\nabla_{\!\dot{H}^1}J(\phi_n)}^2 .
\]
Moreover, the mass densities move closer to the target $\nu$, 
\[
\normhmone{\rho_{n+\frac{1}{2}} - \nu}^2 - \normhmone{\rho_n - \nu}^2 \le -\normhmone{\rho_{n+\frac{1}{2}}-\rho_n}^2,
\]
where we denote $\rho_n = T_{\phi_n\#}\mu$ and $\rho_{n+\frac{1}{2}} = T_{\phi_{n+\frac{1}{2}}\#}\mu$ and where the $\dot{H}^{-1}$ distance is defined by~\eqref{eq:defHm1dot}.
\end{prop}
See appendix for proof. 

\begin{remark}
Proposition~\ref{prop:ascent-descent} shows that the descent property holds over $\dot{H}^1(\Omega)$ when all of the iterates $\phi_n$ are assumed to have upper and lower Hessian bounds.  Unfortunately, we cannot expect this property to hold throughout the simple ascent scheme 
$\phi_{n+1}=\phi_n+\sigma\nabla_{\!\dot{H}^1} J(\phi_n)$.
Indeed, the gradient ascent updates may eventually push the iterates outside of the desired region.  In fact, we cannot even guarantee that $c$-concavity is preserved by the gradient update steps. 
\end{remark}
\begin{remark}
One way to avoid the aforementioned problem is to add a step where the iterates are projected back onto the set 
\[
\mathcal{S}_{\lambda}=\{\phi: (1-\lambda^{-1})I \le D^2\phi(x) \le (1-\lambda)I\}.
\]
$\mathcal{S}_{\lambda}$ is convex, thus the projection step does not interfere with the ascent property.  Unfortunately, it is expensive to compute projections onto $\mathcal{S}_{\lambda}$ (one must solve a semi-definite programming problem).  Furthermore, when the densities do not have convex support, the maximizer of the dual problem may not lie in $\mathcal{S}_{\lambda}$ for any $\lambda>0$ (however it must be in $\mathcal{S}_{0}$ by $c$-concavity).  In such a case, one would need to send $\lambda\to 0$ to obtain an arbitrarily accurate solution.
\end{remark}

To circumvent some of the difficulties mentioned in the above remarks, we replace traditional gradient ascent steps with a back-and-forth scheme that alternates between maximizing $J$ and maximizing $I$.     

\subsection{Back-and-forth updates}

Let us recall that the Kantorovich dual problem can be written in the form
\begin{equation}
\label{eq:Kantorovich-symmetric}
C(\mu,\nu) = \sup_{\phi,\psi} \int_{\Om} \phi(y)\nu(y)\,dy + \int_{\Om} \psi(x)\mu(x)\,dx,
\end{equation}
where the supremum runs over continuous functions $\phi$ and $\psi$ satisfying
\[
\phi(y) + \psi(x) \le h(y-x),
\]
for all $x,y\in\Om$. This formulation emphasizes the symmetric role played by the potentials $\phi$ and $\psi$. In what follows, it will be convenient to represent the dual functional in the form
\begin{equation}\label{eq:symmetric_dual}
D(\phi,\psi) = \int_{\Om} \phi\,\nu\,dy + \int_{\Om}\psi\,\mu\,dx - \iota_{\calC}(\phi,\psi),
\end{equation}
where the constraints are encoded by the convex indicator function $\iota_{\calC}$
which takes value $0$ on the convex set
\[\calC = \set{(\phi,\psi) :  \forall x,y\in\Om\quad \phi(y)+\psi(x) \le h(y-x)},\]
and $+\infty$ for all other pairs $(\phi, \psi)$. Then the Kantorovich problem~\eqref{eq:Kantorovich-symmetric} can be simply written $C(\mu,\nu) = \sup_{\phi,\psi} D(\phi,\psi)$. Moreover the functionals $J$ and $I$ can be obtained by either eliminating $\phi$ or $\psi$ from the symmetric representation (\ref{eq:symmetric_dual}). Indeed, one can check that
\[
\phi^c = \argmax_{\psi} D(\phi,\psi), \quad\psi^c = \argmax_{\phi} D(\phi,\psi),
\]
and thus
\[
J(\phi) = D(\phi,\phi^c) = \sup_{\psi} D(\phi,\psi), \quad I(\psi) = D(\psi^c,\psi) = \sup_{\phi} D(\phi,\psi).
\]

A vanilla gradient ascent scheme $\phi_{n+1} = \phi_n + \sigma\nabla_{\!\dot{H}^1}J(\phi_n)$ focuses (arbitrarily) on $\phi$-space, iterating gradient steps on the functional $J$. Alternatively, one could work instead in $\psi$-space and write a gradient ascent scheme on the functional $I$. A perhaps better idea is to alternate between $\phi$-space and $\psi$-space. This is the core idea of the back-and-forth method given in Algorithm~\ref{algo:main}. Recall that a  current iterate $(\phi_n,\psi_n)$ is updated as follows:
\begin{equation*}
\label{eq:back-and-forth}
\begin{aligned}
\phi_{n+\frac{1}{2}} &= \phi_n + \sigma\nabla_{\!\dot{H}^1} J(\phi_n),  \\
\psi_{n+\frac{1}{2}} &= (\phi_{n+\frac{1}{2}})^c, \\
\psi_{n+1} &= \psi_{n+\frac{1}{2}} + \sigma \nabla_{\!\dot{H}^1} I(\psi_{n+\frac{1}{2}}), \\
\phi_{n+1} &= (\psi_{n+1})^c.
\end{aligned}
\end{equation*}

The back-and-forth approach already corrects a difficulty we identified in the previous subsection.  When one considers a pure gradient ascent scheme
$\phi_{n+1}=\phi_n+\sigma\nabla_{\!\dot{H}^1}J(\phi_n)$ on $\phi$ and $J$ only ( or analogously on $\psi$ and $I$ only),
there is no reason that the iterates remain $c$-concave.   If $\phi$ or $\psi$ is not $c$-concave, then the gradients $\nabla_{\!\dot{H}^1} J(\phi)$ or $\nabla_{\!\dot{H}^1} I(\psi)$ may not be well defined.   In the back-and-forth method, whenever one takes a gradient, the function in question is always $c$-concave.  Indeed, $\phi_n=\psi_{n-1}^c$ and $\psi_{n+\frac{1}{2}}=\phi_{n+\frac{1}{2}}^c$ thus, $\nabla_{\!\dot{H}^1} J(\phi_n)$ and $\nabla_{\!\dot{H}^1} I(\psi_{n+\frac{1}{2}})$ are always well-defined. 

In addition, the back-and-forth updates can never perform worse than a pure gradient ascent scheme.  Indeed, the intermediate steps where we take a $c$-transform can only increase the value of the dual problem.  It follows from property (i) in Lemma \ref{lem:c_transform_properties} that
\[
    J(\phi_{n+\frac{1}{2}})\leq J(\phi_{n+\frac{1}{2}}^{cc})=I(\psi_{n+\frac{1}{2}}),
\]
and
\[
    I(\psi_{n+1})\leq I(\psi_{n+1}^{cc})=J(\phi_{n+1}).
\]
Thus, the intermediate $c$-transform steps help increase the value of the dual functional. If we combine this with our previous result from Section \ref{ssec:gd}, we can conclude the following result about the back-and-forth method.

\begin{prop}
\label{prop:ascent-back-and-forth}
Consider the back-and-forth method (Algorithm~\ref{algo:main}) for the quadratic-cost optimal transport problem. Assume that there exists $\lambda > 0$ such that all the iterates
$\phi_n, \phi_{n+\frac{1}{2}}$ and $\psi_n, \psi_{n+\frac{1}{2}}$ lie in the set of functions
\[\mathcal{S}_{\lambda}=\{\vp: (1-\lambda^{-1})I \le D^2\vp(x) \le (1-\lambda)I\}.
\]
 Then we have the chain of inequalities 
\[
D(\phi_{n+1},\psi_{n+1}) \ge D(\phi_{n+\frac{1}{2}},\psi_{n+\frac{1}{2}}) \ge D(\phi_n,\psi_n),
\]
for all integers $n$. Equivalently, the values of $I$ and $J$ alternatively increase: 
\[
I(\psi_{n+1}) \ge I(\psi_{n+\frac{1}{2}}) \ge J(\phi_{n+\frac{1}{2}}) \ge J(\phi_n) \ge I(\psi_n).
\]
\end{prop}

The short proof of Proposition~\ref{prop:ascent-back-and-forth} is given in the appendix.

Note that even with the back-and-forth updates we cannot give a rigorous convergence proof without assuming convexity bounds on the iterates.  Nonetheless, for reasons that we do not fully understand, the back-and-forth method appears to be extremely stable even for relatively large step sizes (a vanilla ascent method on either $J$ or $I$ behaves less favorably).  Furthermore, the back-and-forth method converges extremely rapidly to the maximizer. Heuristically, we believe this is because it is easier to build certain features in either ``$\phi$-space'' or ``$\psi$-space''. Indeed, if $(\phi, \psi)$ is a pair of $c$-conjugate functions for the quadratic cost, one has the relationship
$D^2\phi(x-\nabla\phi(x))=(D^2\psi(x))^{-1}$ i.e. the quadratic $c$-transform inverts bounds on Hessians.  By alternating between spaces, one can build features in the space where they are smaller, and thus converge to the solution more rapidly. 

\section{Numerical Implementation and Results}\label{sec:numerics}

\subsection{The fast \texorpdfstring{$c$}{c}-transform}
\label{ssec:numeric_c_transform}
Computing the $c$-transform on a regular grid requires solving the optimization problem 
$$ \phi^c(x)= \inf_{y} h(y-x)-\phi(y) $$
for each grid point $x$.  In one dimension, the strict convexity of the cost implies that 
the $y$ derivative 
$$ h'(y-x)-\phi'(y) $$
is a decreasing function of $x$.  As a result, the minimizers $y(x)=\argmin_y h(y-x)-\phi(y)$ are monotone increasing with respect to $x$.  In other words, if we have $x_1\leq x_2$ then $y(x_1)\leq y(x_2)$.  This observation can be exploited to design a divide-and-conquer algorithm that computes a 1-dimensional $c$-transform on $n$ points in $O(n\log(n))$ operations \cite{corrias_legendre}. 

The key idea is that if one knows the minimizers $y(x)$ for points the points $x=\frac{k}{n}, \frac{k+2}{n}$ then 
$$ y\big(\frac{k}{n}\big)\leq y\big(\frac{k+1}{n}\big)\leq y\big(\frac{k+2}{n}\big), $$
i.e. the minimizer for the middle point is ``trapped'' between the other two.   Thus, one can separately compute a $c$-transform for the even grid points and a $c$-transform for the odd grid points then use the above interlacing property to reconstruct the full solution. See \cite{corrias_legendre} for the explicit algorithm and more details.

If we restrict our attention to costs $h$ which decompose along each dimension, i.e. 
$$ h(y-x)=\sum_{i=1}^d h_i(y_i-x_i),$$
then we can compute a $d$-dimensional $c$-transform by repeated applications of the 1-dimensional $c$-transform.  For example, in 2 dimensions, we have $$ \inf_y h(y-x)-\phi(y)=\inf_{y_2}  \inf_{y_1}  h_2(y_2-x_2)+h_1(y_1-x_1)-\phi(y_1,y_2) $$
$$=\inf_{y_2} h_2(y_2-x_2)+\phi(\cdot,y_2)^c(x_1). $$
If we write $\tilde{\phi}(x_1,y_2)=-\phi(\cdot,y_2)^c(x_1)$ then the last line is a $c$-transform with respect to $y_2$
$$\inf_{y_2} h_2(y_2-x_2)-\tilde{\phi}(x_1,y_2). $$
Thus, in 2 dimensions, one can compute the full $c$-transform by first computing one-dimensional $c$-transforms along all horizontal lines, and then along all vertical lines.  This idea generalizes to arbitrary dimension, and is the exact same mechanism which is used to compute the multidimensional Fast Fourier Transform (FFT). 

In the (extremely important) special case $h(y-x)=\frac{1}{2}|y-x|^2$, the $c$-transform can be computed even more efficiently.  Splitting the quadratic part into 3 terms, we get 
$$\phi^c(x)=\frac{1}{2}|x|^2+\inf_y\Big[ -x\cdot y+\frac{1}{2}|y|^2-\phi(y)\Big].$$
If we then set $\vp(y)=\frac{1}{2}|y|^2-\phi(y)$,
and let $\vp^*(x)$ denote the Legendre transform of $\vp$:
\begin{equation}\label{eq:legendre}\vp^*(x)=\sup_{y} x\cdot y-\vp(y),\end{equation}
then
$$\phi^c(x)=\frac{1}{2}|x|^2-\vp^*(x).$$
Therefore, one can compute the quadratic $c$-transform by computing a Legendre transform instead. 

A one-dimensional Legendre transform on a set of $n$ points can be computed in $O(n)$ operations \cite{lucet_legendre}. The idea of \cite{lucet_legendre} uses two very important properties of the Legendre transform, namely that $\vp^*=\vp^{***}=(\vp^{**})^*$ and that $\vp^{**}$ is the convex hull of $\vp$.  Therefore, computing (\ref{eq:legendre}) amounts to finding the values $y(x)$ where  
\begin{equation}\label{eq:legendre_search} x-(\vp^{**})'(y)\end{equation}
changes sign from positive to negative.   $\vp^{**}$ is convex, so the slopes $(\vp^{**})'(y)$ are increasing.  Therefore, finding the sign change $y(x)$ that corresponds to each $x$ in (\ref{eq:legendre_search}) can be done in a single sweep through the values of $(\vp^{**})'(y)$. On a one-dimensional grid with $n$ regularly spaced points, the sweep step takes $O(n)$ operations and crucially, the convex hull $\vp^{**}$ can also be computed in $O(n)$ operations \cite{lucet_legendre}.  Finally, a multidimensional Legendre transform on a regular grid can be computed by decomposing the problem along each dimension in the same manner as described above for the $c$-transform.  For the explicit algorithm and more details, we refer to~\cite{lucet_legendre}.

\subsection{Step sizes}
\label{sec:numerics:steps-sizes}

As we saw in Section \ref{sec:main}, it is not easy to determine the Lipschitz constant of the gradient mapping 
at an arbitrary function $\phi$.  As such, it is not clear how to optimally choose $\sigma$ for a fixed step size gradient ascent scheme.  Thus, rather than choose a single fixed value $\sigma$ at the outset of the algorithm, we update $\sigma=\sigma_n$ throughout the algorithm using Armijo--Goldstein type update rules \cite{Armijo}. 
 
After each gradient step, we compare the difference $J(\phi_{n+1})-J(\phi_n)$ to the squared norm  $\normhone{\nabla_{\!\dot{H}^1} J(\phi_n)}^2$ (and correspondingly for $I$ and $\psi_{n+1}$, $\psi_n$).  Given parameters $0<\beta_1<\beta_2<1 $, $\alpha_1>1$ and $\alpha_2<1$ we check whether
 \[
  -\sigma_n\beta_2\normhone{\nabla_{\!\dot{H}^1} J(\phi_n)}^2\leq J(\phi_{n})-J(\phi_{n+1})\leq -\sigma_n\beta_1\normhone{\nabla_{\!\dot{H}^1} J(\phi_n)}^2.
 \]
If the upper inequality fails, we decrease $\sigma_n$ by taking $\sigma_{n+1}=\alpha_2\sigma_n$, and if the lower inequality fails then we increase $\sigma_n$ by taking $\sigma_{n+1}=\alpha_1\sigma_n$.  In all of our experiments we take $\beta_1=\frac{1}{4}, \beta_2=\frac{3}{4}, \alpha_1=\frac{5}{4},$ and $\alpha_{2}=\frac{4}{5}$ and we choose $\sigma=8\min(\norm{\mu}_{L^{\infty}}^{-1},\norm{\nu}_{L^{\infty}}^{-1}) $ as our starting value. To ensure that the step sizes stay bounded away from zero, we do not allow the update rules to drop the step size below a small constant $\sigma_{\textrm{min}}=0.01$. 

Note that we do not perform any backtracking. The new value of $\sigma=\sigma_{n+1}$ is simply used in the next update step. Indeed, evaluating $J$ or $I$ at a function requires computing a $c$-transform, thus it is not worth trying to optimize the step size at a particular iterate.

\subsection{Experiments}
Throughout this section we will assume that $\Omega=[0,1]^d$ is the unit cube in $\R^d$.  When implementing our algorithms, we discretize $\Omega$ using a regular finite $d$-dimensional grid.  Recall that Laplace equations on this domain should be solved with zero Neumann boundary conditions.  This ensures that there is no flux of mass outside of the computational domain.

All numerical algorithms were coded in C and executed on a single 1.6 GHz core with 8GB RAM. Inversion of the Laplace operator was performed using the Fast Fourier Transform (FFT).  All FFTs were calculated using the free FFTW C library.   

In what follows, we will consider three sets of experiments. In the first two sets of experiments, we restrict our attention to optimal transport with the quadratic cost $\frac{1}{2}|x-y|^2$. 
We start by computing optimal transport maps between geometric shapes where explicit answers are available.  This allows us to verify the accuracy and convergence rate of the method. We also verify that the back-and-forth method substantially outperforms a gradient ascent method on the dual problem with no back-and-forth structure.  Next, we apply our method to images, and use the resulting optimal map to compute the displacement interpolation \cite{fry, McCann_interpolation} between the source and target images.  As we shall see, our computed maps are sufficiently accurate to capture the fine-scale details of the images. Finally, in the last set of experiments, we consider the more general class of costs outlined in Section \ref{ssec:numeric_c_transform}, and study how the optimal map varies as the underlying cost changes.

\subsubsection*{Quadratic cost optimal transport}  In this subsection, we focus on computing optimal transport with the quadratic cost
\[h(y-x)=\frac{1}{2}|y-x|^2,\]
which represents the most important special case of optimal transport.  

We begin with a simple example where the two probability densities $\mu$ and $\nu$ differ only by a translation, i.e. $\mu(x)=\nu(x+a)$ for a fixed constant $a\in\Rd$.  In this case, it is known that the optimal map $T_*$ pushing $\mu$ to $\nu$ is itself a translation: $T_*(x)=x+a$, and the transportation cost is $\frac{1}{2}|a|^2$. Thus, we can directly check the accuracy of our method on this class of elementary examples.   

First, in two dimensions, suppose that $\mu$ is uniform and supported on a disc of radius $\frac{1}{8}$ centered at $(\frac{1}{4}, \frac{1}{4})$ and $\nu$ is uniform and supported on a disc of the same radius centered at $(\frac{3}{4}, \frac{3}{4})$ (both densities are normalized to have unit mass).
\begin{center}
\begin{tikzpicture}[scale=1.5]
\fill[black] (0,0) rectangle (1,1);
\fill[white] (1/4,1/4) circle (1/8);
\draw (1/2,-0.2) node {$\mu$};
\tikzset{shift={(1.2,0)}}
\fill[black] (0,0) rectangle (1,1);
\fill[white] (3/4,3/4) circle (1/8);
\draw (1/2,-0.2) node {$\nu$};
\end{tikzpicture}
\end{center}
It then follows that the quadratic transportation cost between these densities is $$\frac{1}{2}\Big(\big(\frac{1}{2}\big)^2+\big(\frac{1}{2}\big)^2\Big)=\frac{1}{4}.$$
Our results for this experiment are presented in Table~\ref{table:2d_2_discs}.  We run our algorithm until the difference between the computed transportation cost and the exact solution is less than the prescribed error tolerance.  We see that to get an error of $10^{-4}$ we only need 3 iterations, and to get an error of $10^{-8}$ we only need 5 iterations.  Let us also highlight that the performance of the algorithm is completely independent of the tested grid size. 

\begin{table}
\center{\caption{Two discs in $2$ dimensions\label{table:2d_2_discs}}}

\begin{tabular}{ c | c | c | c | c |}
\cline{2-5}
 & \multicolumn{2}{|c|}{Error $10^{-4}$} & \multicolumn{2}{|c|}{Error $10^{-8}$} \\ \hline 
\multicolumn{1}{|c|}{Grid size} &  Iterations & Time (s) & Iterations & Time (s) \\ \hline
\multicolumn{1}{|c|}{$512 \times 512$} & 3 & 1.00 & 5 & 1.25 \\ \hline
\multicolumn{1}{|c|}{$1024 \times 1024$} & 3 & 4.23 & 5 & 5.21 \\ \hline
\multicolumn{1}{|c|}{$2048 \times 2048$} & 3 & 16.7 & 5 & 21.75 \\ \hline
\multicolumn{1}{|c|}{$4096 \times 4096$} & 3  & 56.46 & 5 & 75.70 \\ \hline
\end{tabular}
\end{table}

Next we repeat the same experiment in three dimensions. In this case, $\mu$ is uniform and supported on a ball centered at $(\frac{1}{4}, \frac{1}{4}, \frac{1}{4})$ with radius $\frac{1}{8}$ and $\nu$ is uniform and supported on a ball centered at $(\frac{3}{4}, \frac{3}{4}, \frac{3}{4})$ with the same radius.  The cost to transport the densities to one another is $\frac{3}{8}$. Our results for this experiment are presented in Table \ref{table:3d_2_balls}.  Again, we run our algorithm until the difference between the computed transportation cost and the exact solution is less than the prescribed error tolerance.  Once again we see that the algorithm converges extremely rapidly, and the convergence rate is independent of the grid size. 

\begin{table}
\center{\caption{\label{table:3d_2_balls}Two balls in $3$ dimensions}}

\begin{tabular}{ c | c | c | c | c |}
\cline{2-5}
 & \multicolumn{2}{|c|}{Error $10^{-4}$} & \multicolumn{2}{|c|}{Error $10^{-8}$} \\ \hline 
\multicolumn{1}{|c|}{Grid size} &  Iterations & Time (s) & Iterations & Time (s) \\ \hline
\multicolumn{1}{|c|}{$128\times 128\times 128$} & 6 & 17.89 & 10  & 27.00 \\ \hline
\multicolumn{1}{|c|}{$256\times 256\times 256$} & 6 & 156.94 & 9 & 205.5 \\ \hline
\multicolumn{1}{|c|}{$384\times 384\times 384$} & 6 & 472.14 & 9 & 699.21 \\ \hline

\end{tabular}
\end{table}

Now we turn to a more difficult example where the optimal transport map is discontinuous.  In two dimensions, we take $\mu$ to be the renormalized characteristic function of a square with side lengths $\frac{1}{4}$ centered at $(\frac{1}{2}, \frac{1}{2})$ and $\nu$ to be the renormalized characteristic function of the union of four squares each with side length $\frac{1}{8}$ and centers at $(\frac{3}{16}+\frac{5i}{8}, \frac{3}{16}+\frac{5j}{8})$ for $i,j \in \{0,1\}$.
\vspace{0.1em}
\begin{center}
\begin{tikzpicture}[scale=1.5]
\fill[black] (0,0) rectangle (1,1);
\fill[white] (3/8,3/8) rectangle (5/8,5/8);
\draw (1/2,-0.2) node {$\mu$};
\tikzset{shift={(1.2,0)}}
\fill[black] (0,0) rectangle (1,1);
\fill[white] (1/8,1/8) rectangle (2/8,2/8);
\fill[white] (6/8,1/8) rectangle (7/8,2/8);
\fill[white] (1/8,6/8) rectangle (2/8,7/8);
\fill[white] (6/8,6/8) rectangle (7/8,7/8);
\draw (1/2,-0.2) node {$\nu$};
\end{tikzpicture}
\end{center}
In this case, there is an explicit solution for the optimal transport map. One cuts $\mu$ into 4 squares along the lines $y=\frac{1}{2}$ and $x=\frac{1}{2}$ and then each cut square moves to the closest square in $\nu$ by translation.  Thus, we see that the map is discontinuous along the lines $y=\frac{1}{2}$ and $x=\frac{1}{2}$. 
Each cut square is translated by a vector of the form $(\pm \frac{1}{4}, \pm \frac{1}{4})$, therefore, the  transportation cost is $\frac{1}{2}\big((\frac{1}{4})^2+(\frac{1}{4})^2\big)=\frac{1}{16}$.  The results for this experiment are presented in Table \ref{table:2d_4_squares}.  Note that due to the discontinuity in the optimal map, it is not possible to achieve the same level of accuracy as the first example.  Nonetheless, we see that the algorithm still converges to a highly accurate solution in a small number of iterations.  

Next, we consider an analogue of this discontinuous example in 3 dimensions.  In this setting, we take $\mu$ to be the renormalized characteristic function of a cube with side length $\frac{1}{4}$ centered at $(\frac{1}{2}, \frac{1}{2}, \frac{1}{2})$ and $\nu$ to be the renormalized characteristic function of the union of eight cubes each with side length $\frac{1}{8}$ and centers at $(\frac{3}{16}+\frac{5i}{8}, \frac{3}{16}+\frac{5j}{8}, \frac{3}{16}+\frac{5k}{8})$ for $i,j ,k \in \{0,1\}$.  For this problem, the optimal transportation map is obtained by cutting $\mu$ into eight cubes along the lines $x=\frac{1}{2}, y=\frac{1}{2}$, and $z=\frac{1}{2}$ and translating each of the eight cubes to the closest cube in $\nu$.  Each cube is translated by a vector of the form $(\pm\frac{1}{4}, \pm\frac{1}{4}, \pm\frac{1}{4})$, thus, the total transportation cost is $\frac{1}{2}\big((\frac{1}{4})^2+(\frac{1}{4})^2+(\frac{1}{4})^2\big)=\frac{3}{32}$. Once again, due to the discontinuity of the map, the accuracy depends on the grid resolution.  As a result, for grids of size $384\times 384\times 384$ and smaller, the method becomes stationary once the accuracy drops slightly below $10^{-5}$.   Nonetheless, we continue to see that the algorithm converges rapidly to a stationary state.

\begin{table}
\center{\caption{One square to four squares \label{table:2d_4_squares}}}

\begin{tabular}{ c | c | c | c | c | c| c|}
\cline{2-7}
 & \multicolumn{2}{|c|}{Error $10^{-4}$} & \multicolumn{2}{|c|}{Error $10^{-5}$}
 & \multicolumn{2}{|c|}{Error $10^{-6}$}
 \\ \hline 
\multicolumn{1}{|c|}{Grid size} &  Iterations & Time (s) & Iterations & Time (s)  & Iterations & Time (s)\\ \hline
\multicolumn{1}{|c|}{$512 \times 512$} & 3 & 1.08 & 5 & 1.45 & 13 & 2.96\\ \hline
\multicolumn{1}{|c|}{$1024 \times 1024$} & 3 & 4.98 & 5 & 5.82 & 14 & 12.2  \\ \hline
\multicolumn{1}{|c|}{$2048 \times 2048$} & 3 & 18.9 & 5 & 25.2 & 14  & 53.9 \\ \hline
\multicolumn{1}{|c|}{$4096 \times 4096$} & 3  & 68.3 & 5 & 93.6 & 13 & 217 \\ \hline
\end{tabular}
\end{table}

\begin{table}
\center{\caption{One cube to eight cubes\label{table:3d_8_cubes}}}

\begin{tabular}{ c | c | c | c | c |}
\cline{2-5}
 & \multicolumn{2}{|c|}{Error $10^{-3}$} & \multicolumn{2}{|c|}{Error $10^{-5}$} \\ \hline 
\multicolumn{1}{|c|}{Grid size} &  Iterations & Time (s) & Iterations & Time (s) \\ \hline
\multicolumn{1}{|c|}{$128\times 128\times 128$} & 3 & 13.36 & 6  & 21.63 \\ \hline
\multicolumn{1}{|c|}{$256\times 256\times 256$} & 3 & 117.44 & 8 & 255.17 \\ \hline
\multicolumn{1}{|c|}{$384\times 384\times 384$} & 3 & 377.05 & 13 & 1577.23 \\ \hline

\end{tabular}
\end{table}

Finally, we compare the performance of the back-and-forth method (BFM) to a simple $\dot{H}^1$-gradient ascent method on the dual problem (\ref{eq:dual}).   The gradient ascent method alternates the following two steps
\begin{equation*}
\begin{aligned}
\phi_{n+\frac{1}{2}}&=\phi_n+\sigma\nabla_{H^1} J(\phi_n),\\
\phi_{n+1}&=\big(\phi_{n+\frac{1}{2}}\big)^{cc},
\end{aligned}
\end{equation*}
hence there is no back-and-forth structure, and the method only considers one of the dual problems.  
We consider two test cases in 2-dimensions, the disc to disc and the square to four squares examples described above.  In each experiment, we choose a grid of size $1024\times 1024$ and run each algorithm until a certain error tolerance is met.  In both experiments, the back-and-forth method substantially outperforms the pure gradient ascent method (c.f. Tables~\ref{table:vs_disc_1024} and~\ref{table:vs_square_1024}).  Note that each gradient ascent iteration only requires computing two Legendre transforms and a single FFT, while each back-and-forth iteration must compute 4 Legendre transforms and 2 FFTs.  Even when accounting for this difference, the back-and-forth method is still vastly superior in both iteration count and computation time.  

\begin{table}
\center{\caption{BFM v.s. gradient ascent (two discs) \label{table:vs_disc_1024}}}

\begin{tabular}{ c | c | c | c | c | c| c|}
\cline{2-7}
 & \multicolumn{2}{|c|}{Error $10^{-4}$} & \multicolumn{2}{|c|}{Error $10^{-5}$}
 & \multicolumn{2}{|c|}{Error $10^{-6}$}
 \\ \hline 
\multicolumn{1}{|c|}{Method} &  Iterations & Time (s) & Iterations & Time (s)  & Iterations & Time (s)\\ \hline
\multicolumn{1}{|c|}{BFM} & 3 & 4.23 & 5 & 5.21 & 5 & 5.21\\ \hline
\multicolumn{1}{|c|}{gradient ascent} & 50 & 19.0 & 149 & 47.9 & 321 & 94.5  \\ \hline
\end{tabular}
\end{table}

\begin{table}
\center{\caption{BFM v.s. gradient ascent (1 square to 4 squares) \label{table:vs_square_1024}}}

\begin{tabular}{ c | c | c | c | c | c| c|}
\cline{2-7}
 & \multicolumn{2}{|c|}{Error $10^{-4}$} & \multicolumn{2}{|c|}{Error $10^{-5}$}
 & \multicolumn{2}{|c|}{Error $10^{-6}$}
 \\ \hline 
\multicolumn{1}{|c|}{Method} &  Iterations & Time (s) & Iterations & Time (s)  & Iterations & Time (s)\\ \hline
\multicolumn{1}{|c|}{BFM} & 3 & 4.98 & 5 & 5.82 & 14 & 12.2\\ \hline
\multicolumn{1}{|c|}{gradient ascent} & 34 & 13.6 & 106 & 34.0 & 744 & 207  \\ \hline
\end{tabular}
\end{table}

\subsubsection*{Optimal transport of images}

In our next set of experiments, we consider optimal transport between two arbitrary black and white images. We shall assume that the image pixels take values in $[0,1]$, where $0$ represents black and $1$ represents white.  The images can then be converted into probability densities by renormalizing pixel values so that each image has total mass $1$.  

In the context of image processing, one typically wishes to compute optimal transport maps to obtain realistic looking interpolations between images.  We will focus on the quadratic cost, which is the most natural for this task.
Given two images with associated probability densities $\mu, \nu$ and the optimal map $T_*$ between them, one can define for each time $t\in [0,1]$ the so called \emph{displacement interpolant} \cite{fry, McCann_interpolation}
\[ \rho(t,x)=S_{t\, \#}\mu,
\]
where $S_t(x)=tT_*(x)+(1-t)x$.  At time zero, $S_0$ is the identity map, so $\rho(0,x)=\mu(x)$, and at time 1, $S_1=T_*$, so $\rho(1,x)=\nu(x)$. The intermediate times $t\in (0,1)$ essentially give a ``video'' which shows how the optimal map deforms one probability density into the other.

We compute three different examples of image interpolation, Figures~\ref{fig:jack_to_horse},  \ref{fig:balls_to_man} and~\ref{fig:monge_to_kantorovich}, using the images in Figure~\ref{fig:images} as our initial and final densities. In Figure~\ref{fig:jack_to_horse}, the starting image is the silhouette of a jack-o'-lantern and the final image is a silhouette of a winged horse.  In Figure~\ref{fig:balls_to_man}, the starting configuration is the union of 4 differently sized discs and the ending configuration is the silhouette of a man holding a sword.   In Figure~\ref{fig:monge_to_kantorovich}, the starting configuration is a portrait of Monge and the final configuration is a photo of Kantorovich.  Note that in each of these examples, the intermediate time interpolations remain sharp and there is almost no diffusion.  In addition, near times 0 and 1, the discontinuities of the optimal map are well localized and resolved.  Figure~\ref{fig:jack_to_horse} is particularly interesting, as one can see how the delicate features of the legs, wings, and head of the horse form over time. 

Even though the examples are large ($1024\times 1024$ pixels) and computed to a high degree of accuracy, the total computation time to find the optimal map and construct the interpolation does not exceed 20 seconds in any case.  To the best of our knowledge,  the size of the examples and the accuracy that we achieved were out of reach for all previous computational optimal transport methods.

\begin{figure}
\includegraphics[width=0.3 \textwidth]{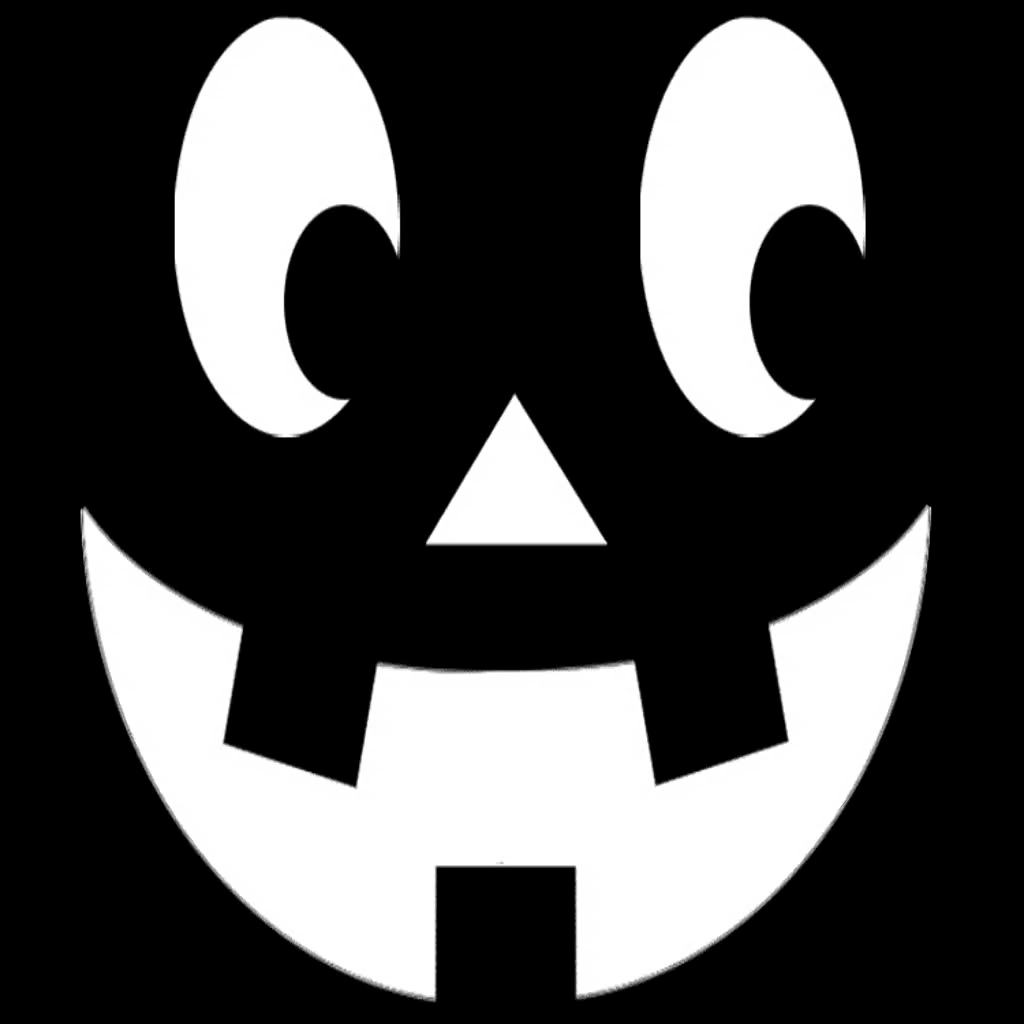}
\includegraphics[width=0.3 \textwidth]{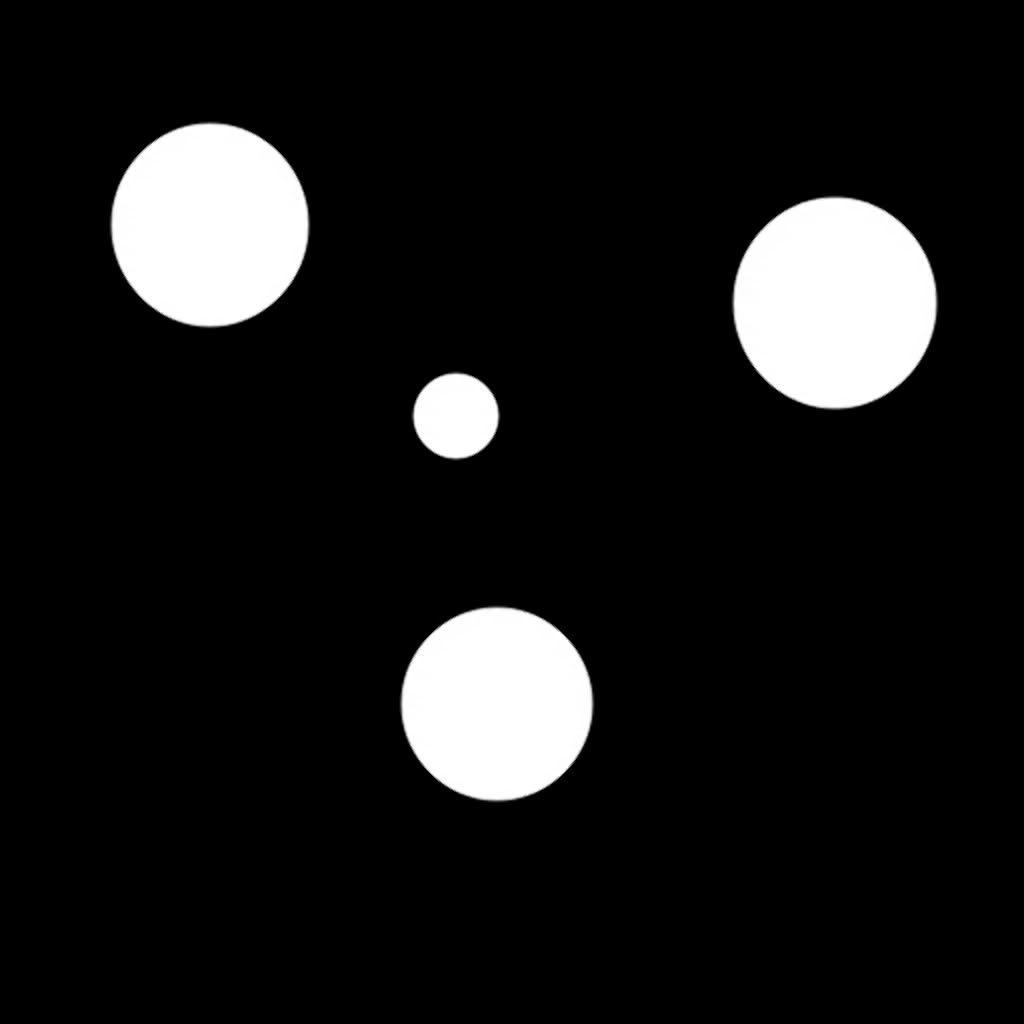}
\includegraphics[width=0.3 \textwidth]{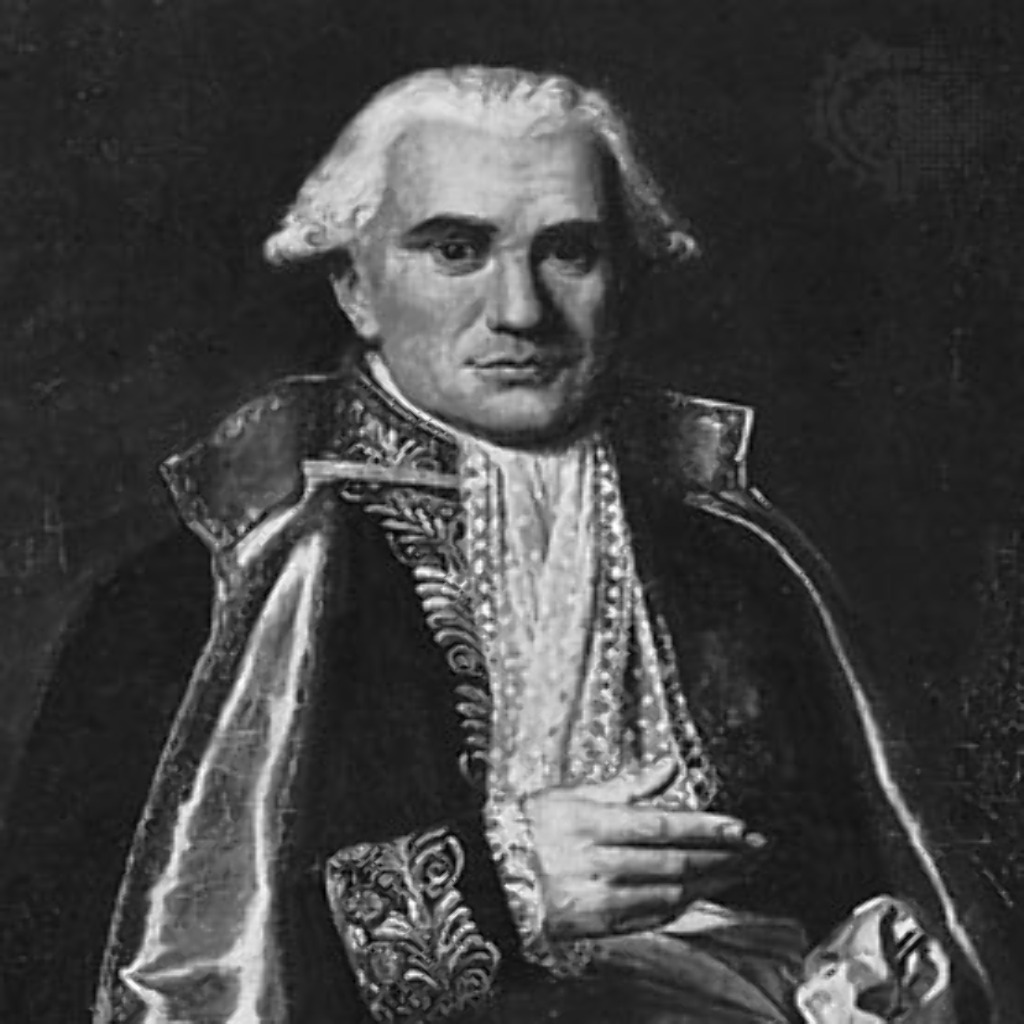}
\includegraphics[width=0.3 \textwidth]{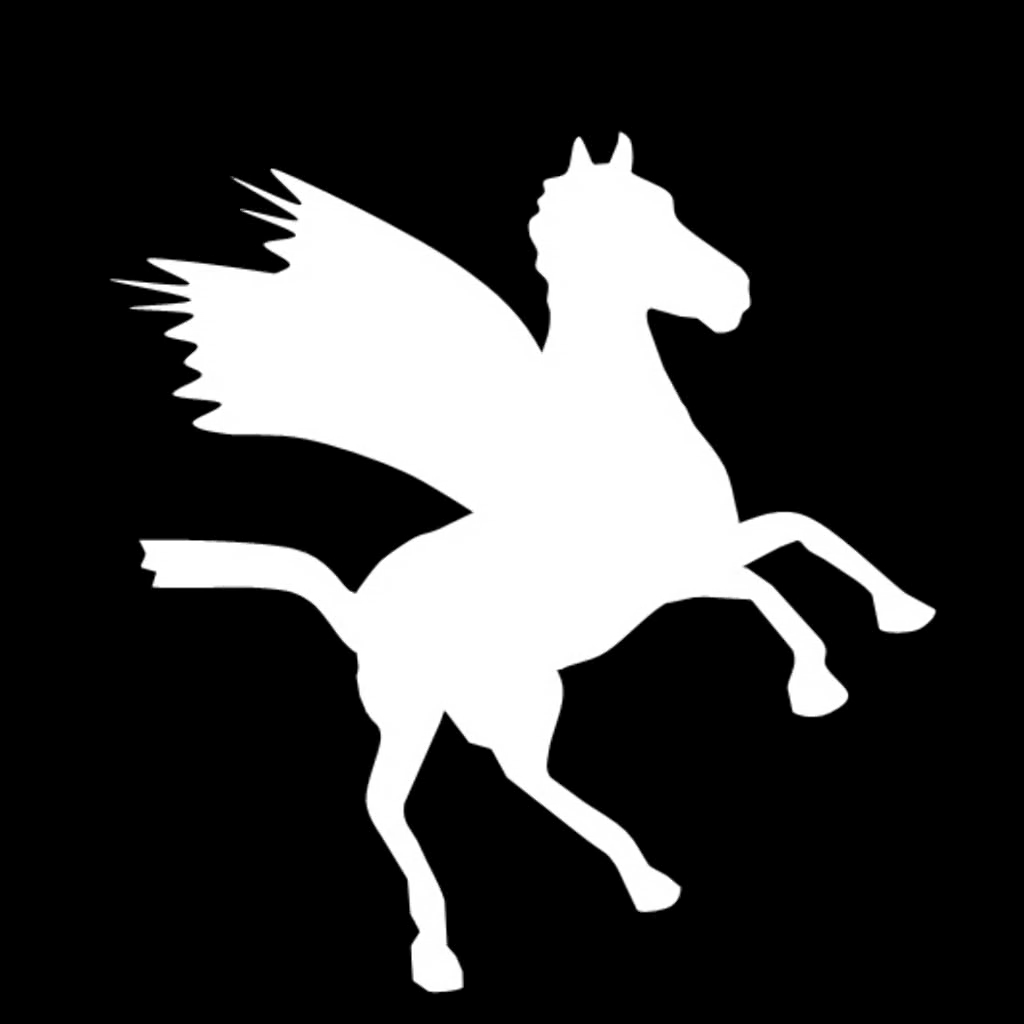}
\includegraphics[width=0.3 \textwidth]{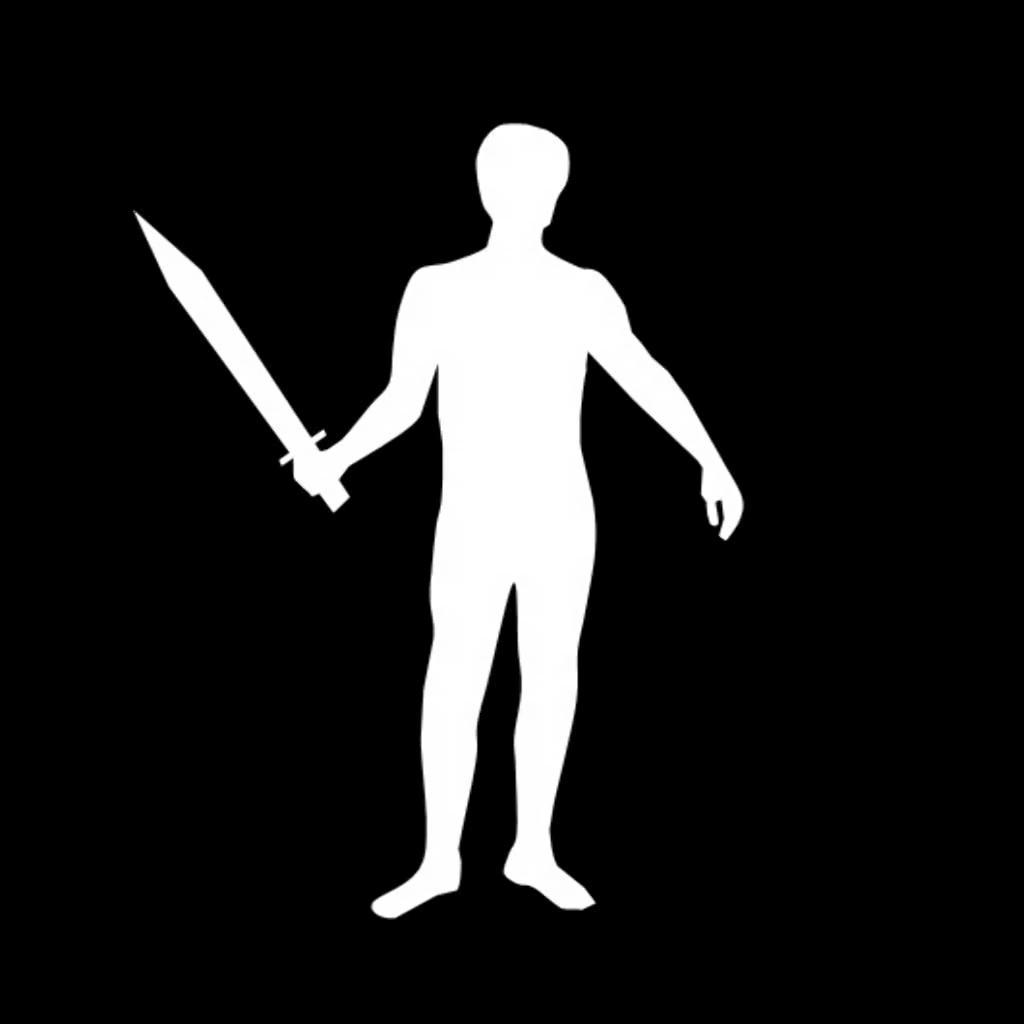}
\includegraphics[width=0.3 \textwidth]{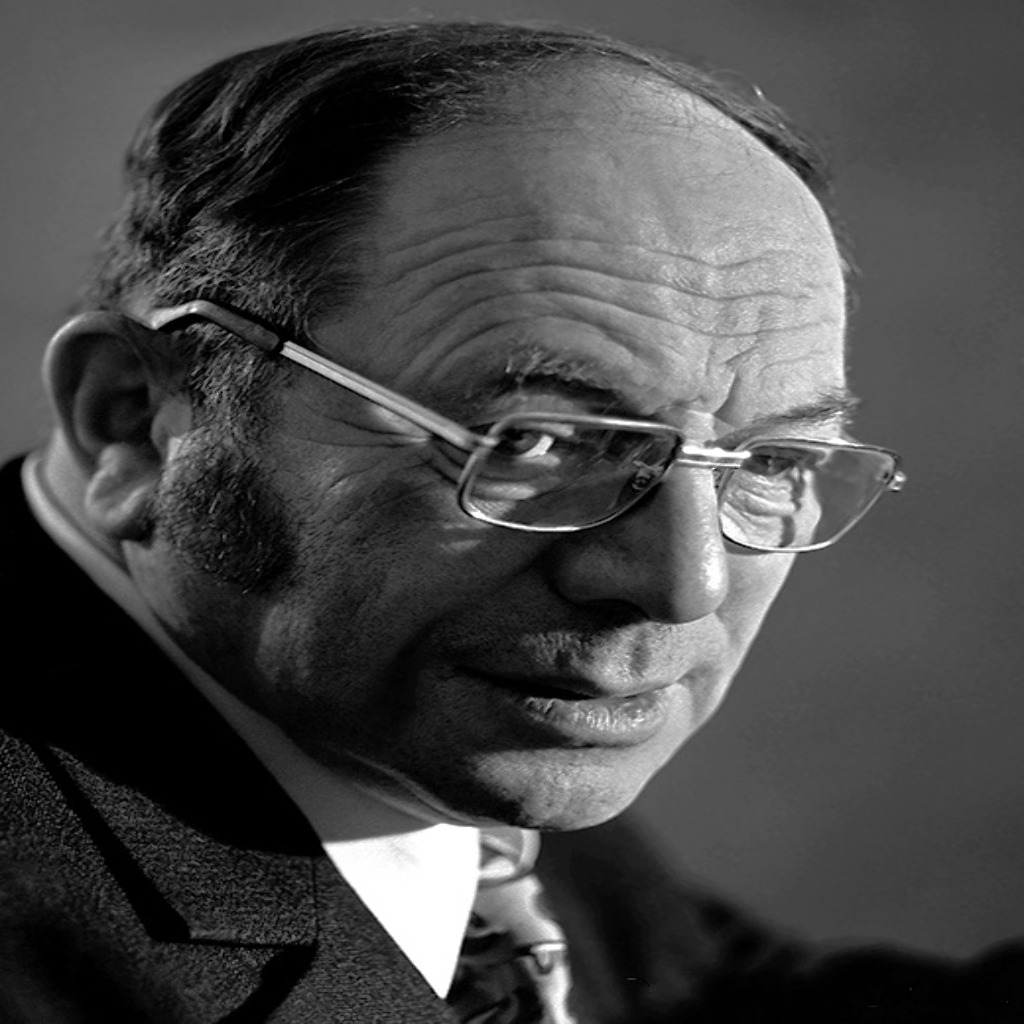}

\caption{Several image examples used to construct interpolations.   \label{fig:images}}
\end{figure}

\begin{figure}
\includegraphics[width=0.3 \textwidth]{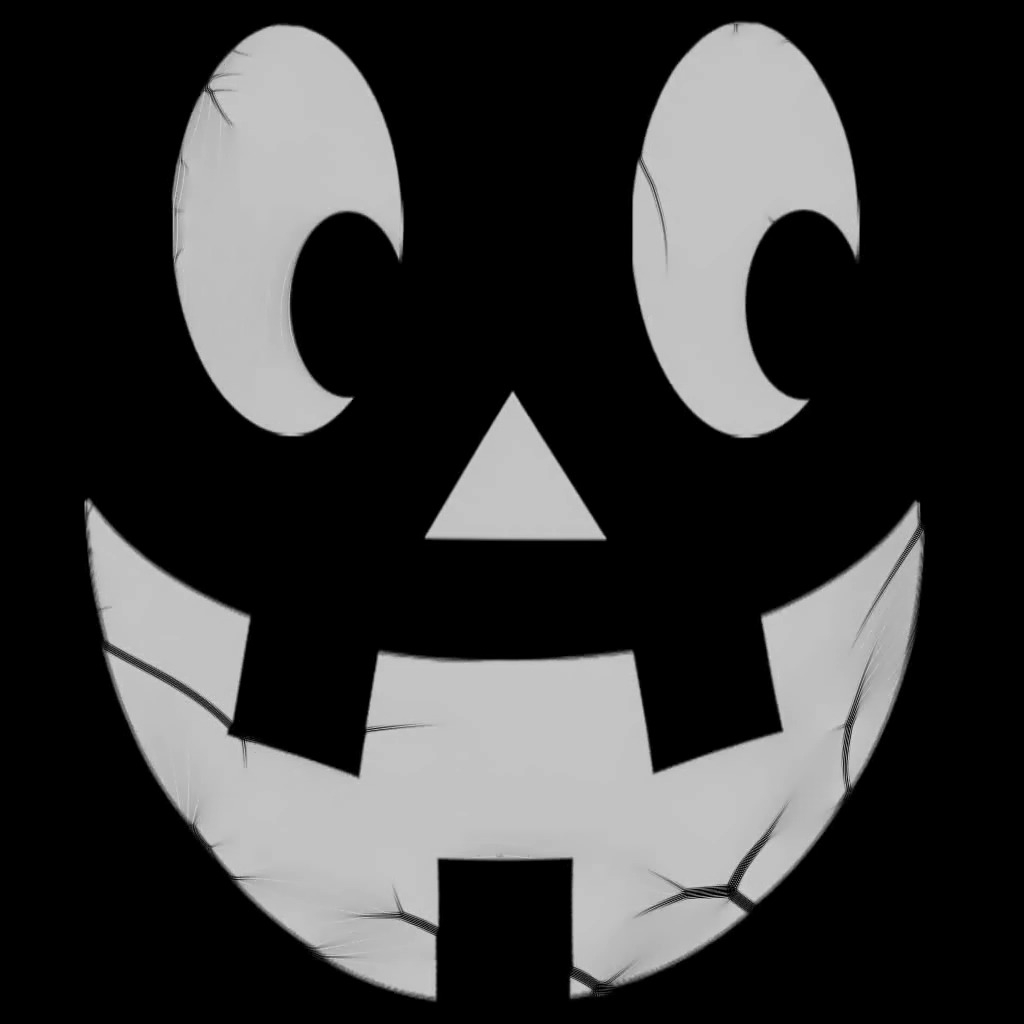}
\includegraphics[width=0.3 \textwidth]{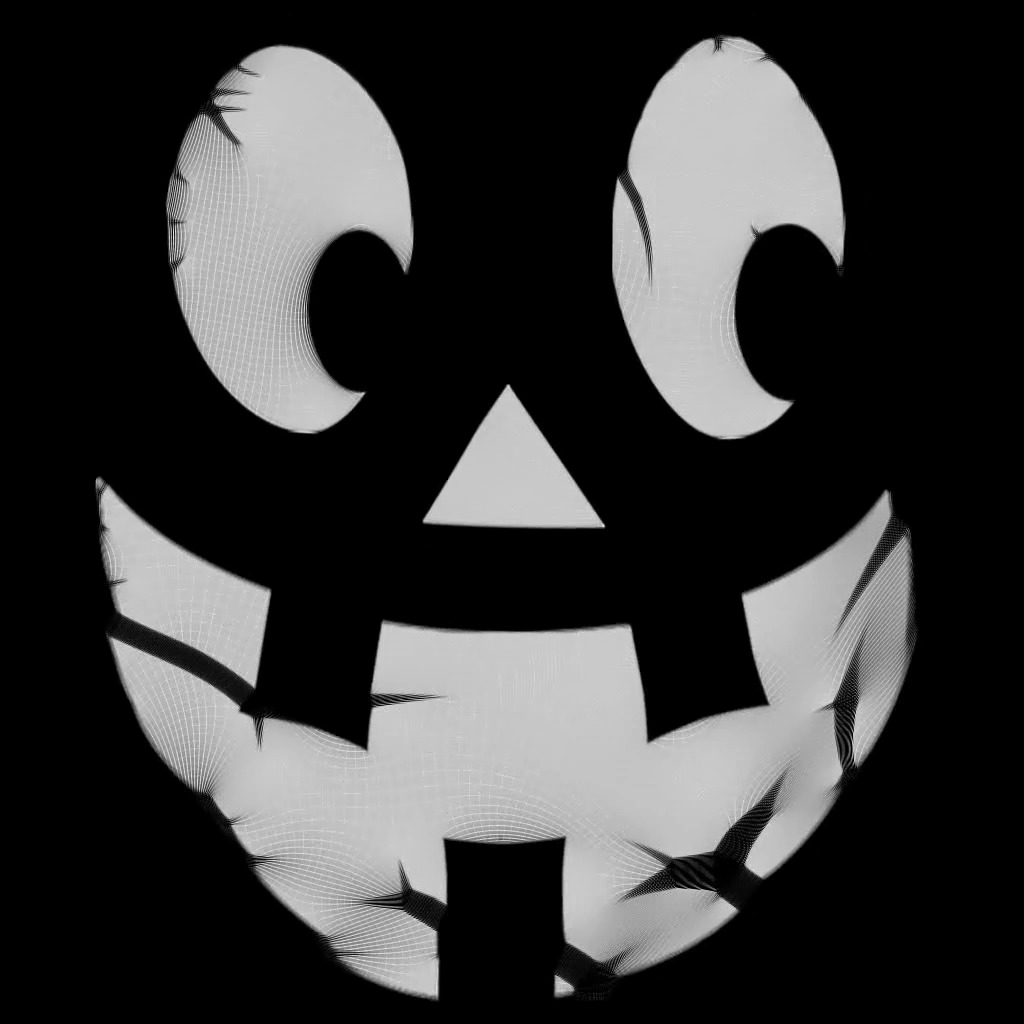}
\includegraphics[width=0.3 \textwidth]{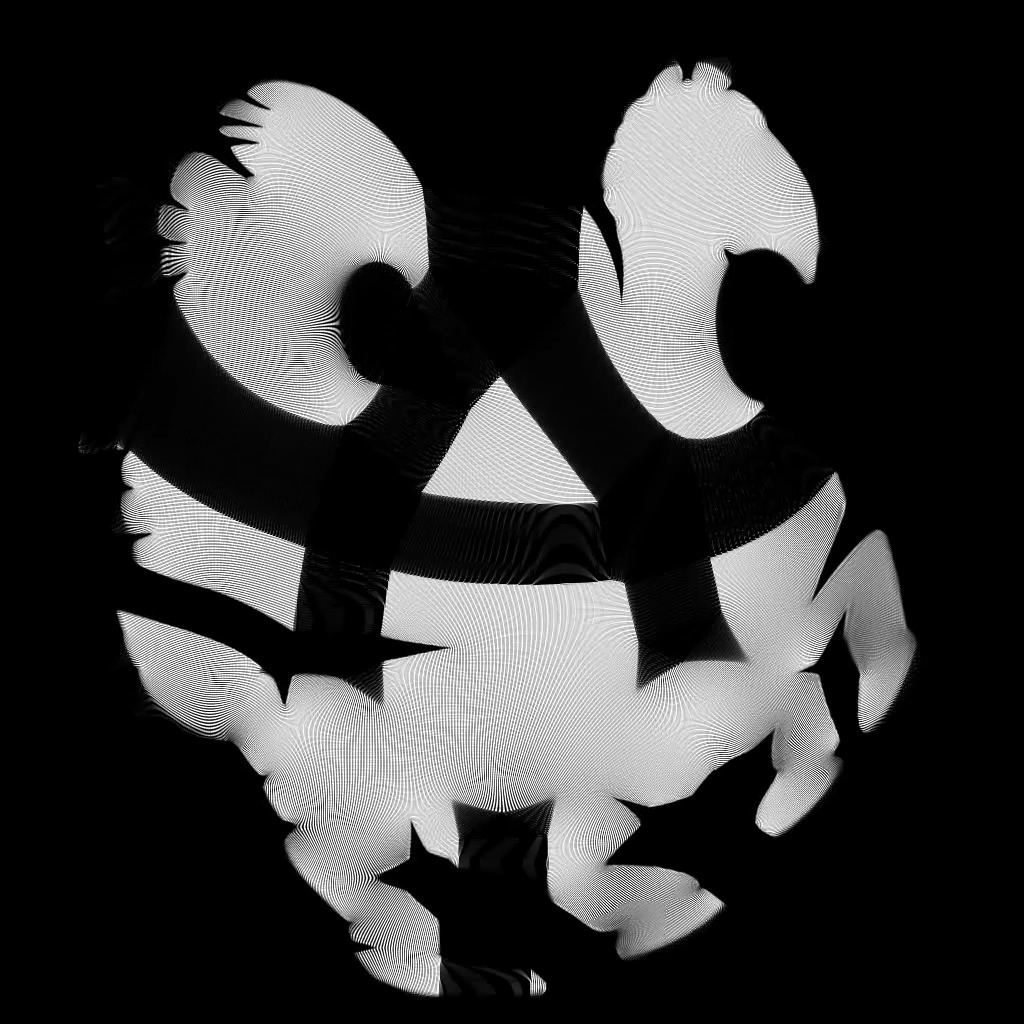}
\includegraphics[width=0.3 \textwidth]{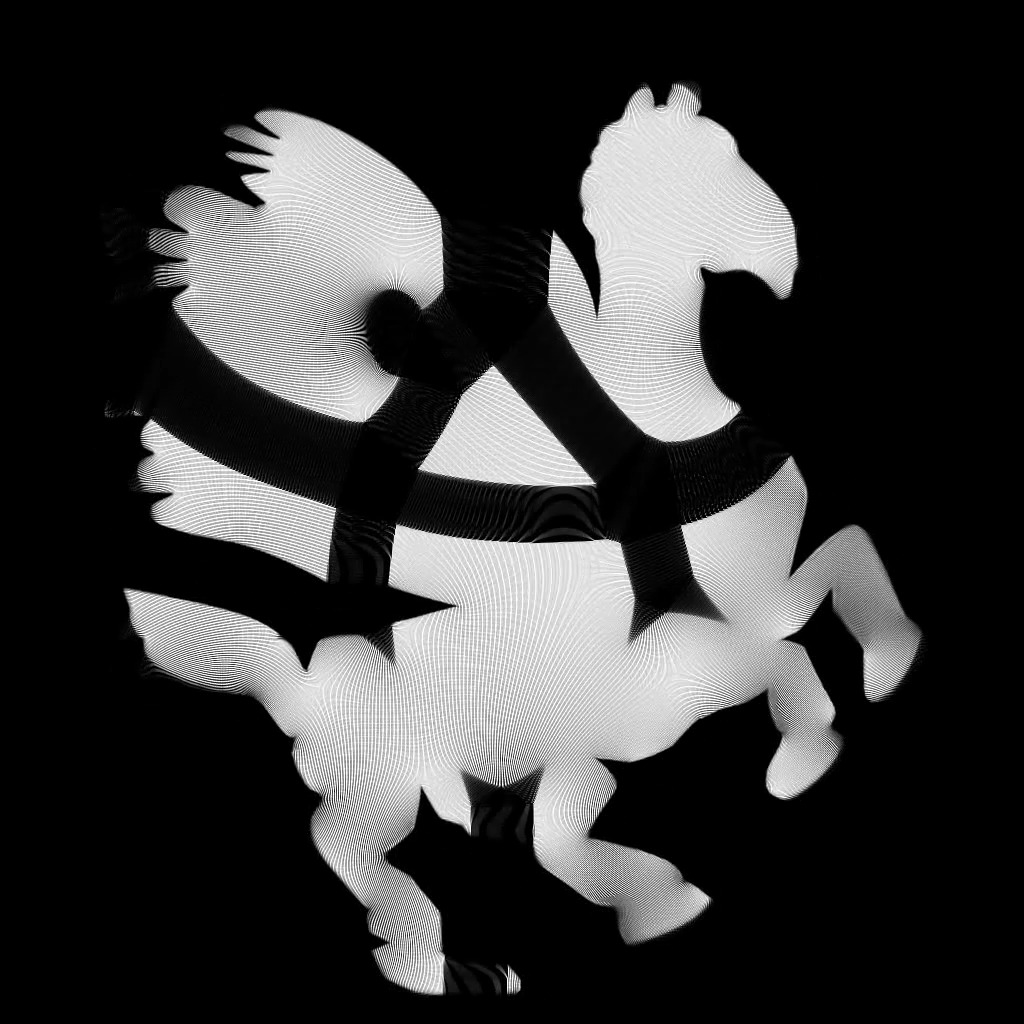}
\includegraphics[width=0.3 \textwidth]{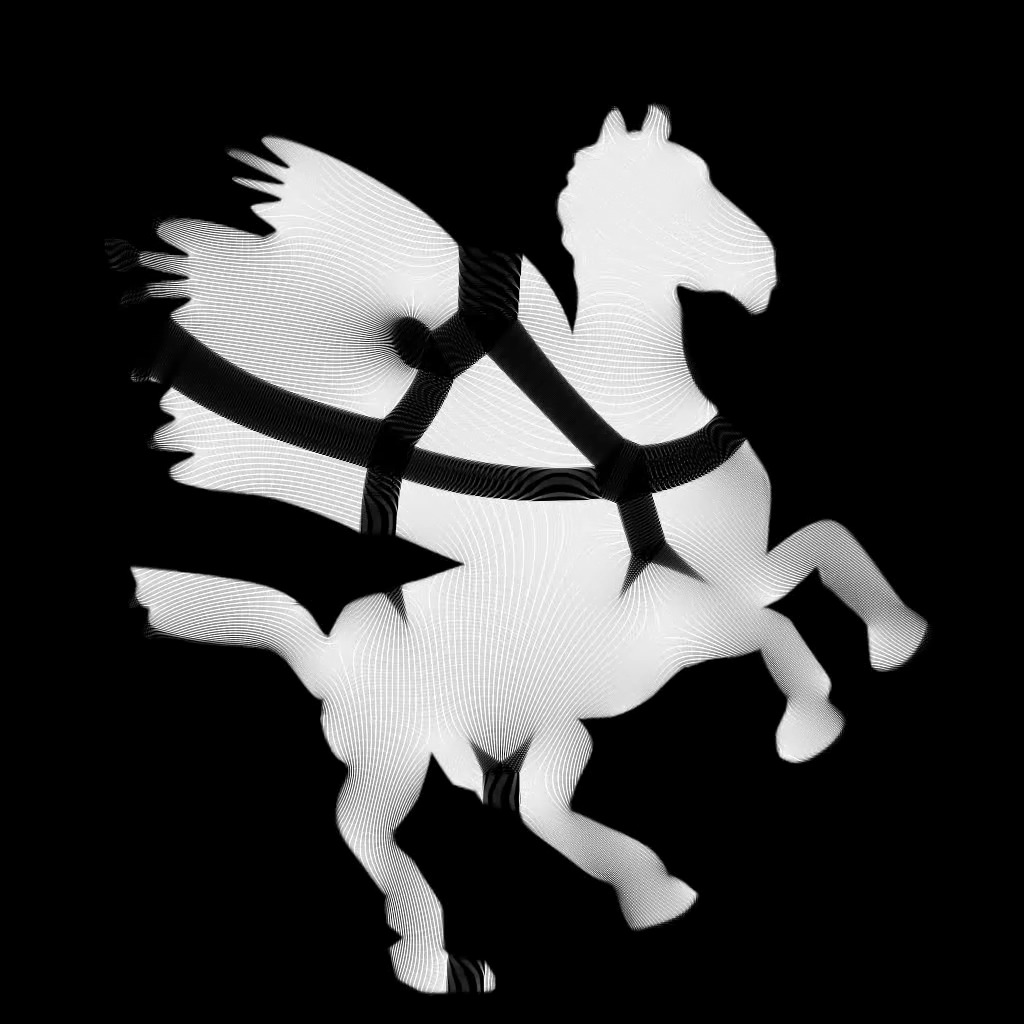}
\includegraphics[width=0.3 \textwidth]{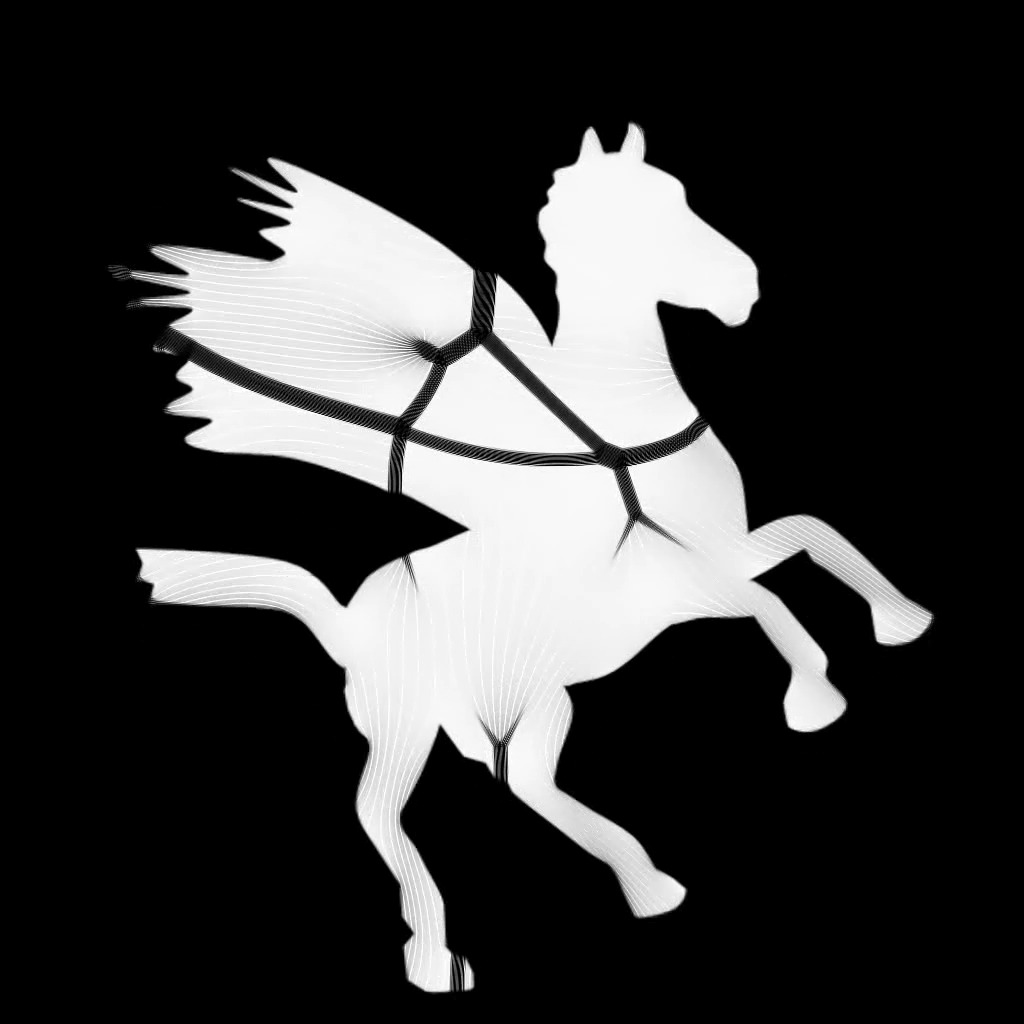}
\caption{Quadratic cost optimal transport between two $1024\times 1024$ pixel images.  The total computation time to compute the optimal map and then compute the interpolation was 20 seconds. \label{fig:jack_to_horse}}
\end{figure}

\begin{figure}
\includegraphics[width=0.3 \textwidth]{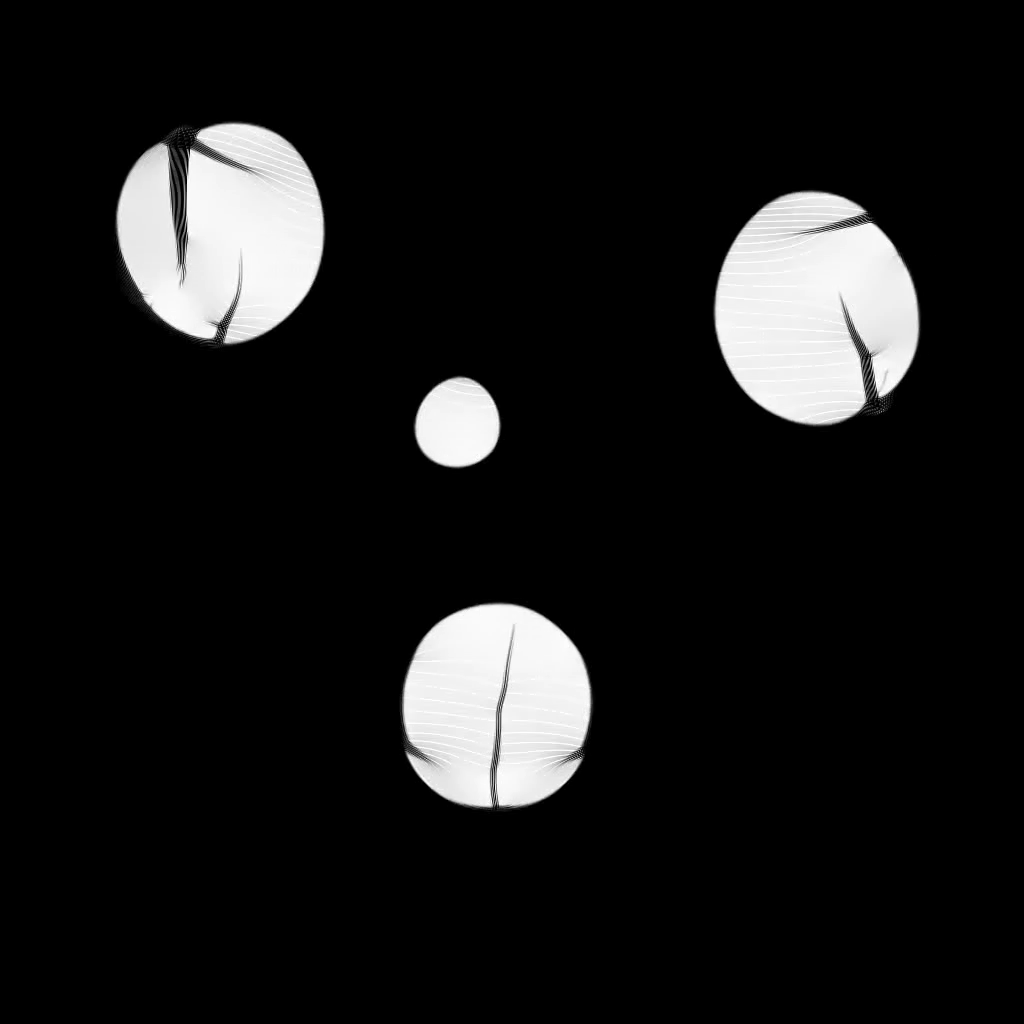}
\includegraphics[width=0.3 \textwidth]{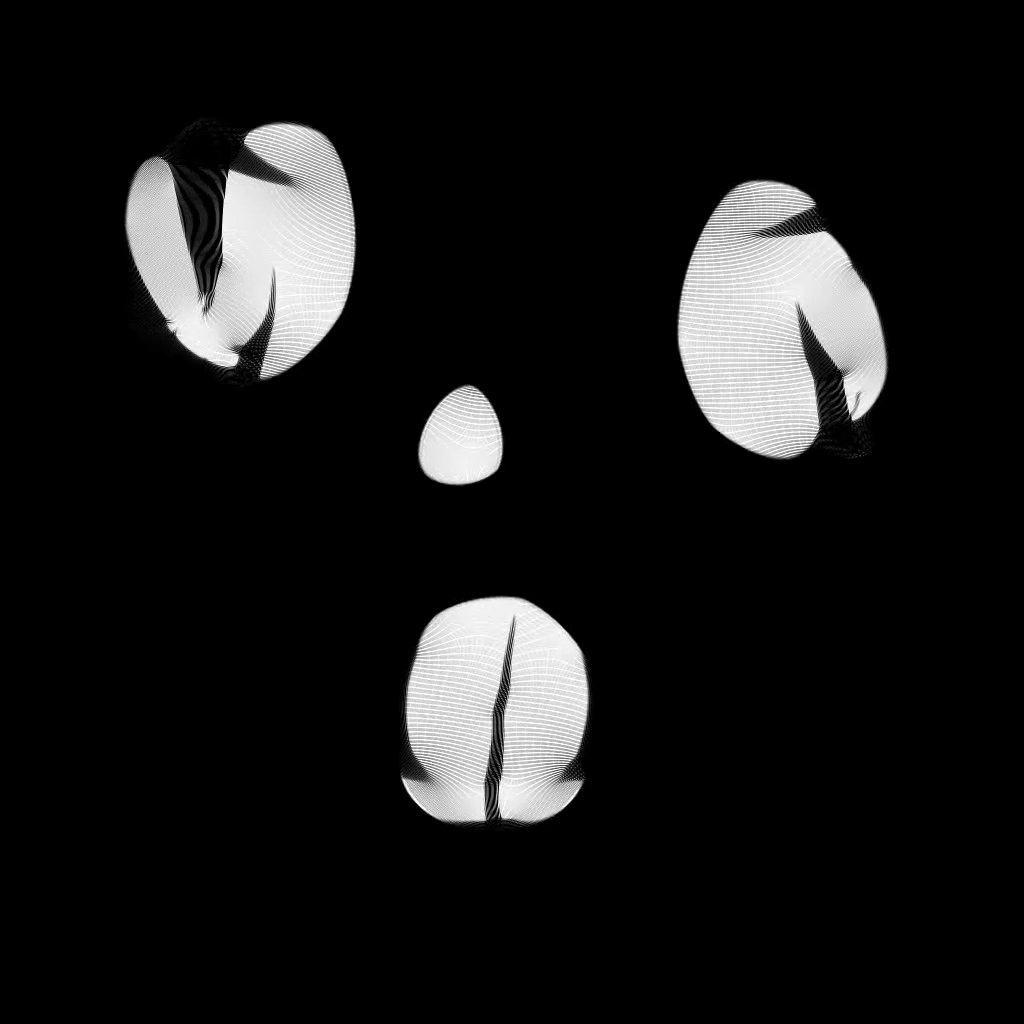}
\includegraphics[width=0.3 \textwidth]{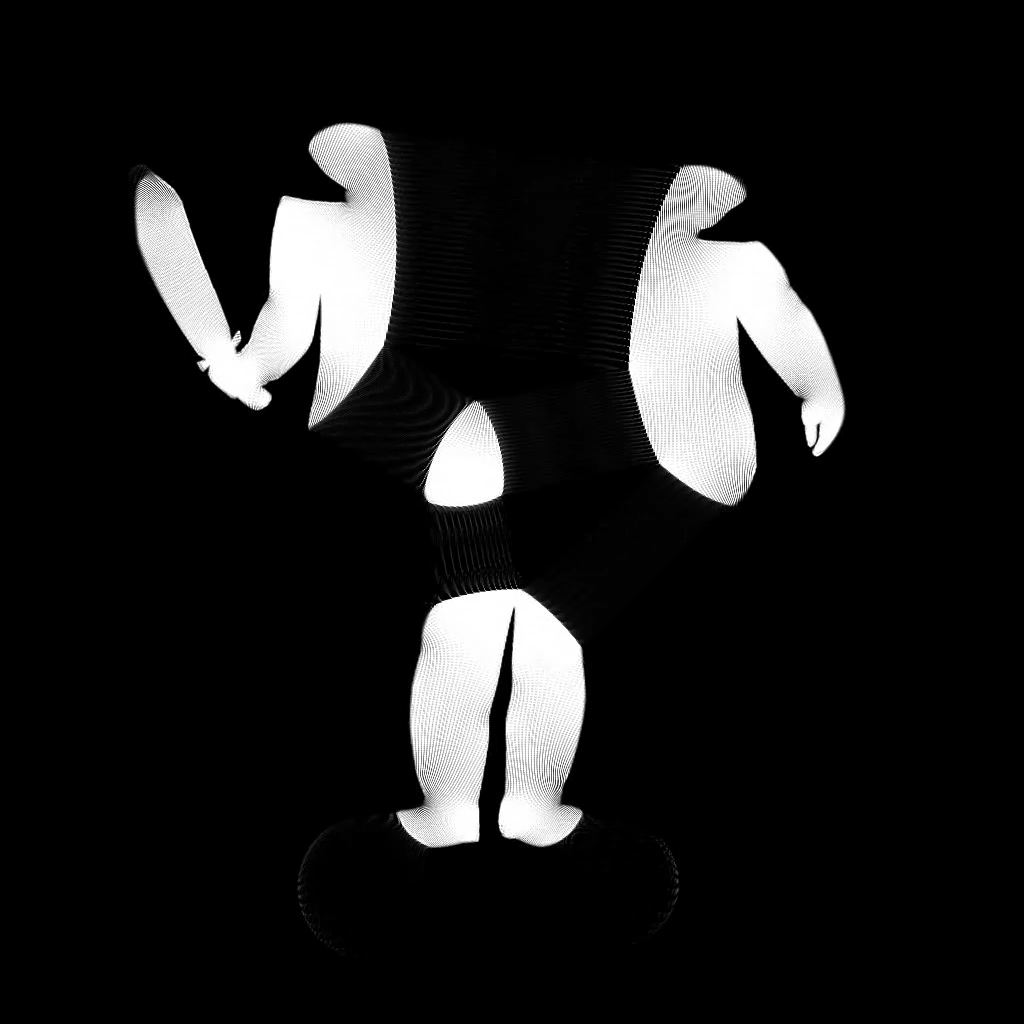}
\includegraphics[width=0.3 \textwidth]{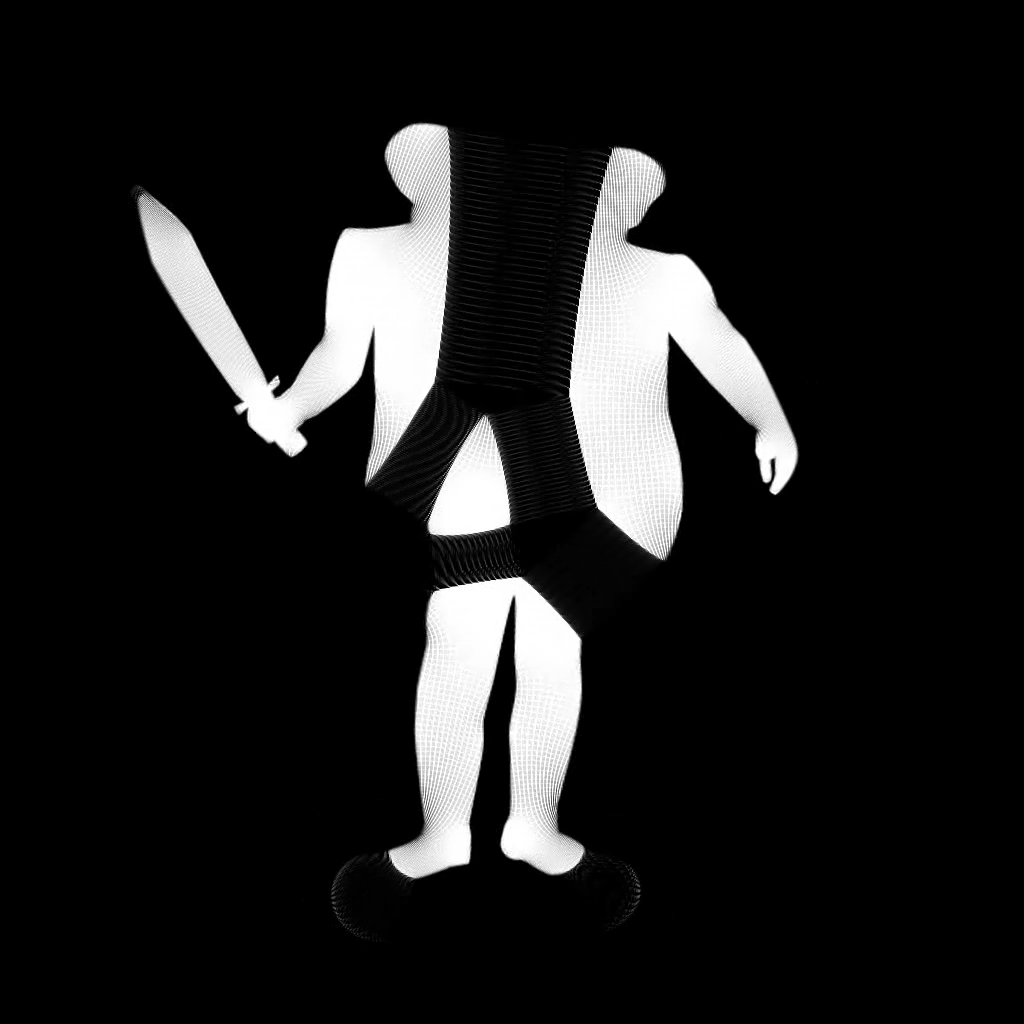}
\includegraphics[width=0.3 \textwidth]{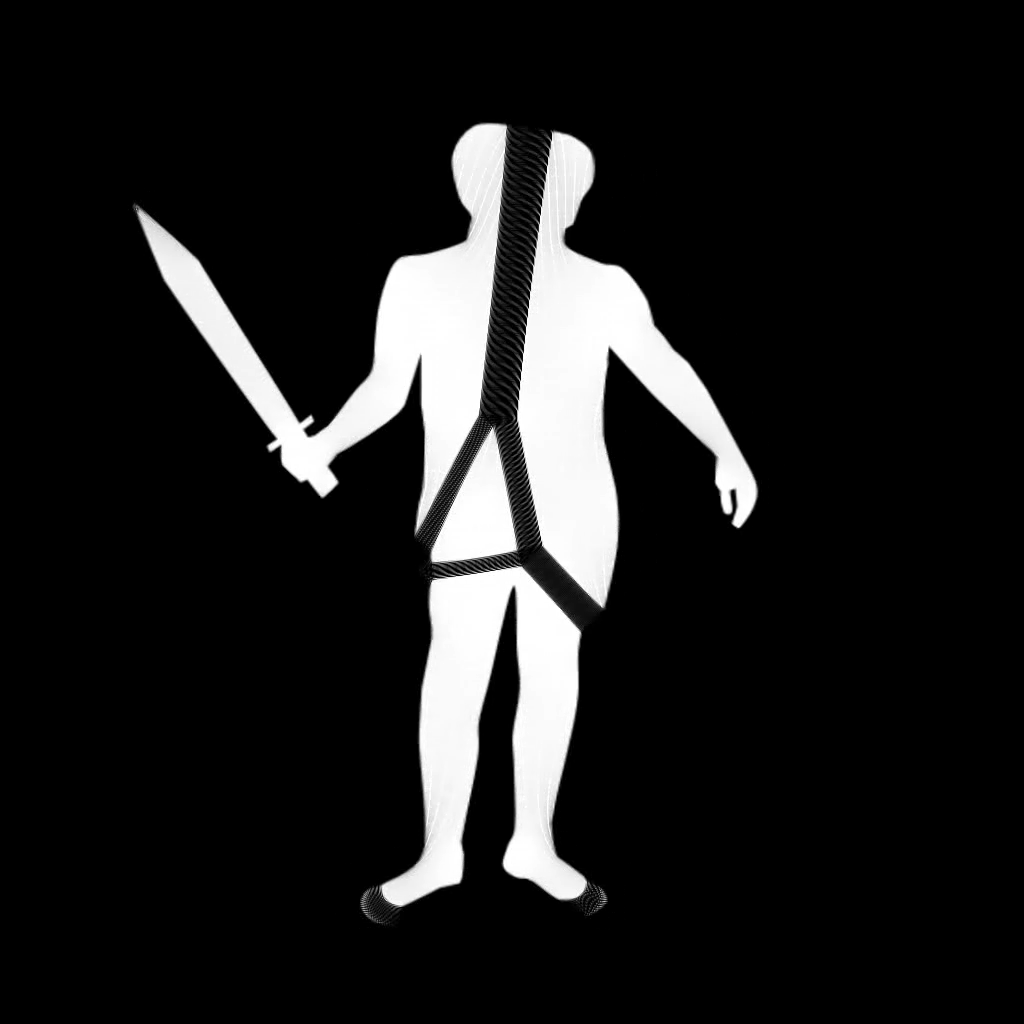}
\includegraphics[width=0.3 \textwidth]{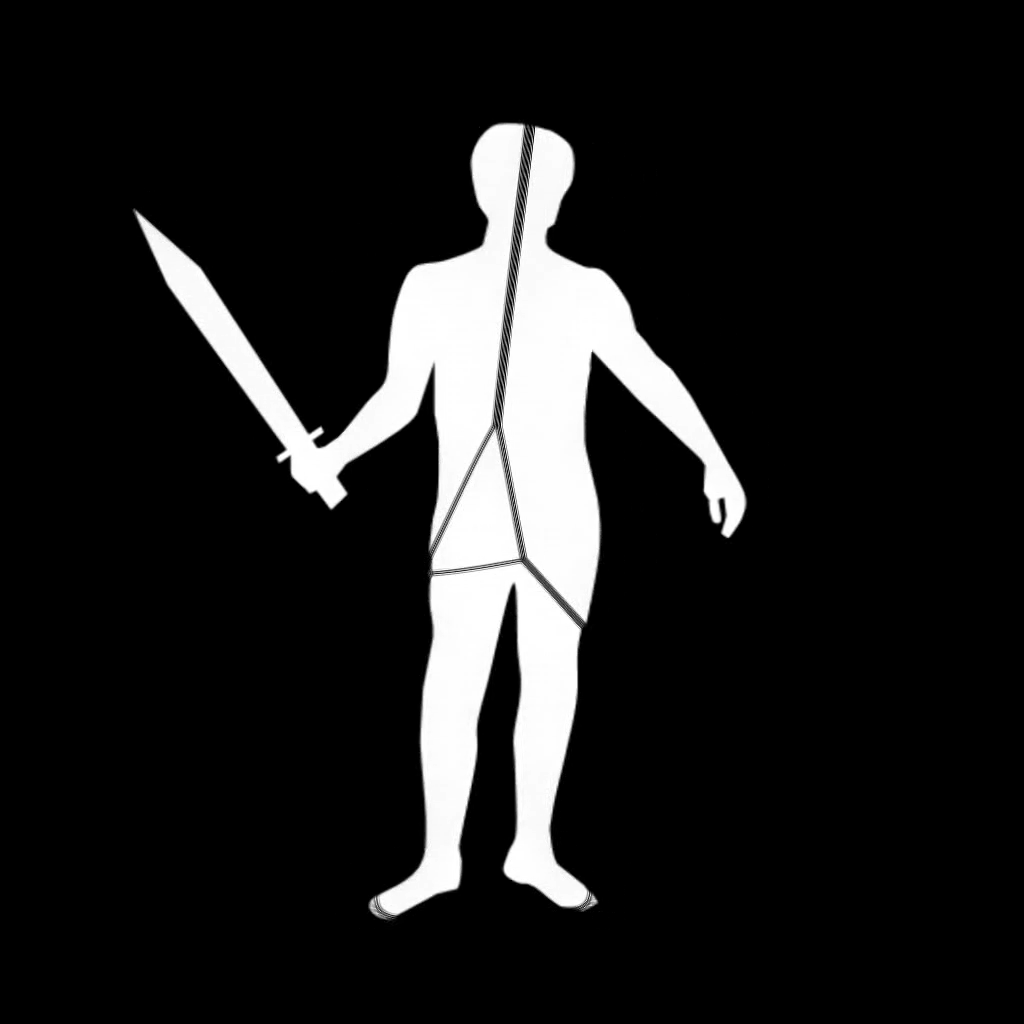}
\caption{Quadratic cost optimal transport between two $1024\times 1024$ pixel images.  The total computation time to compute the optimal map and then compute the interpolation was 18 seconds. \label{fig:balls_to_man} }
\end{figure}

\begin{figure}
\includegraphics[width=0.3 \textwidth]{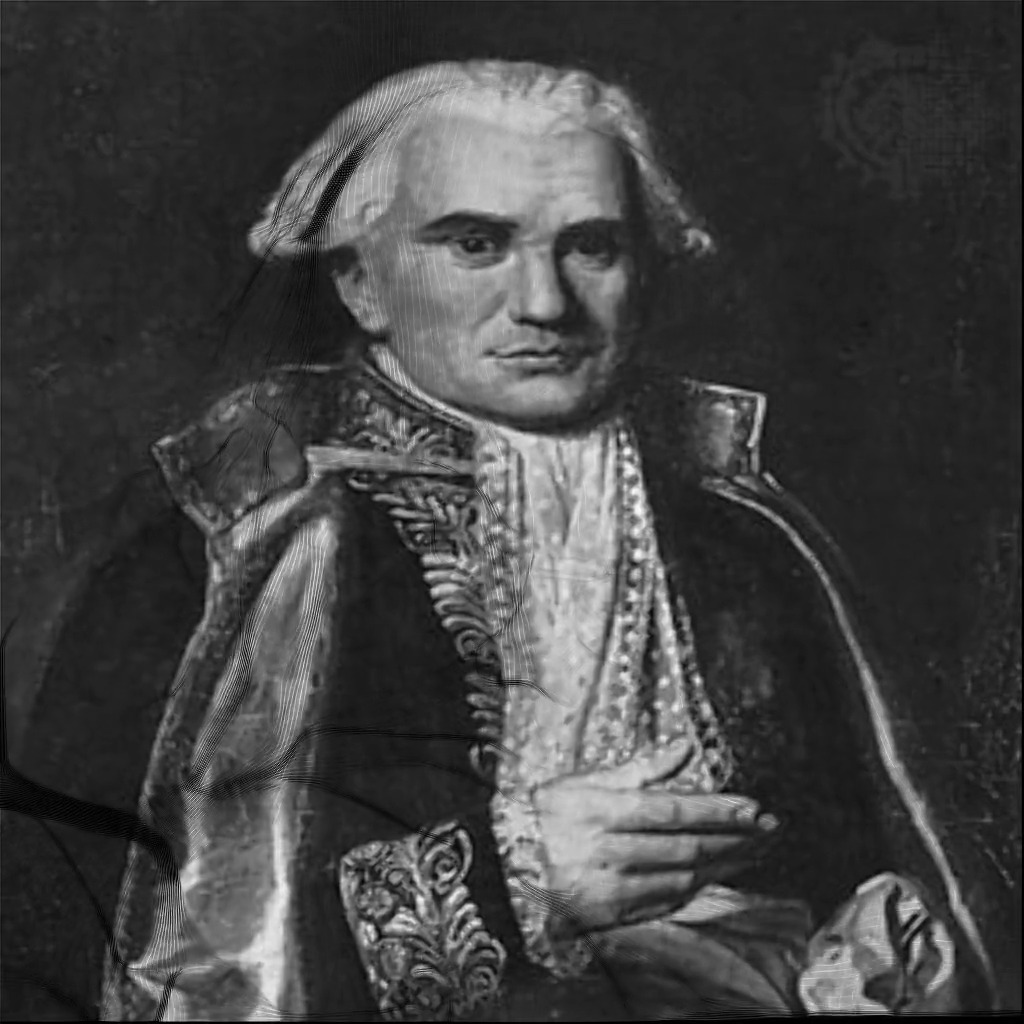}
\includegraphics[width=0.3 \textwidth]{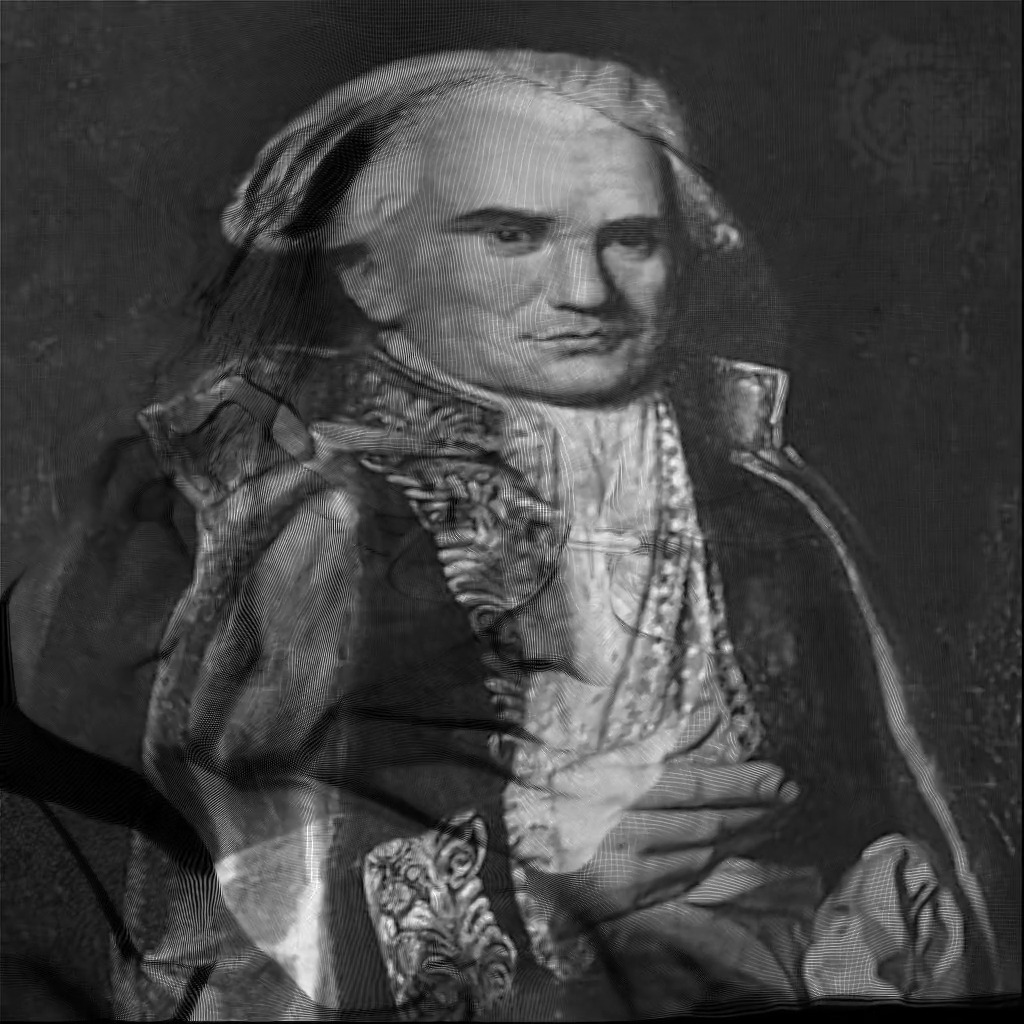}
\includegraphics[width=0.3 \textwidth]{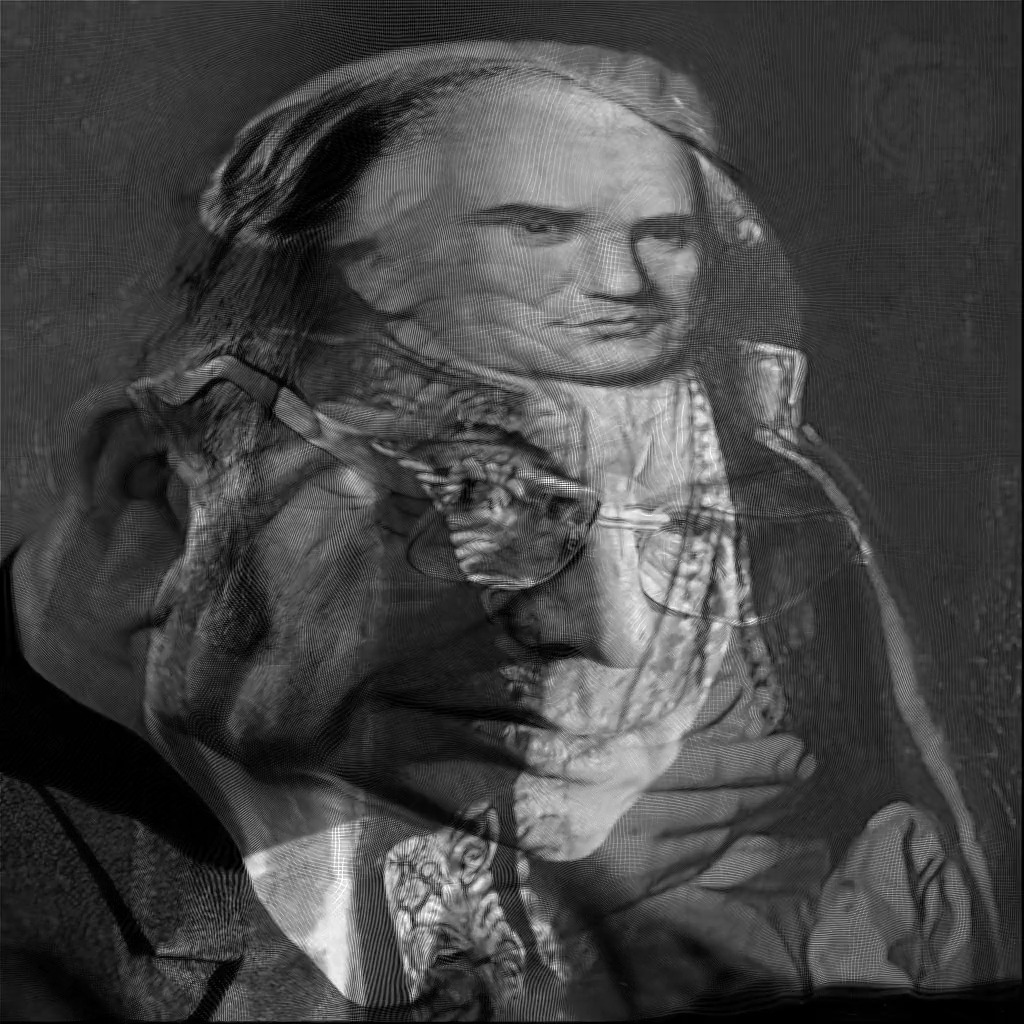}
\includegraphics[width=0.3 \textwidth]{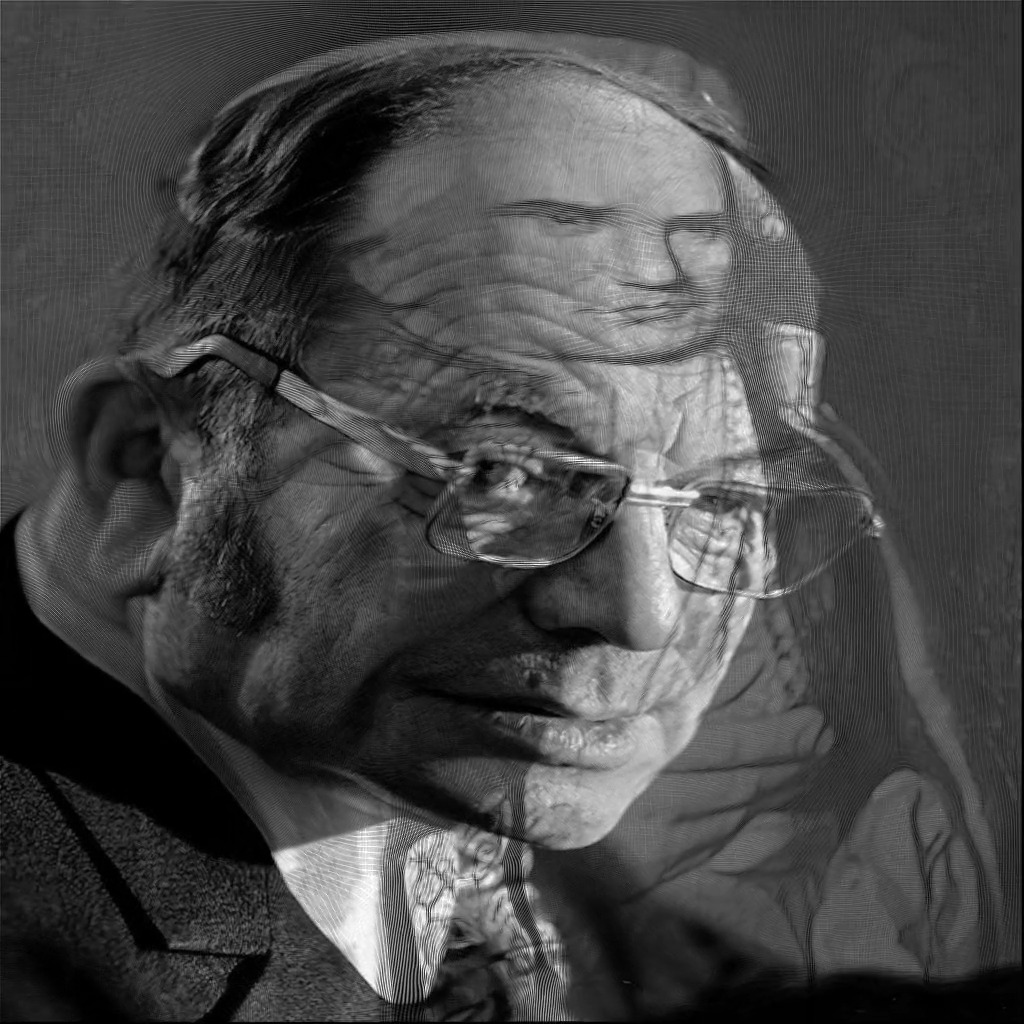}
\includegraphics[width=0.3 \textwidth]{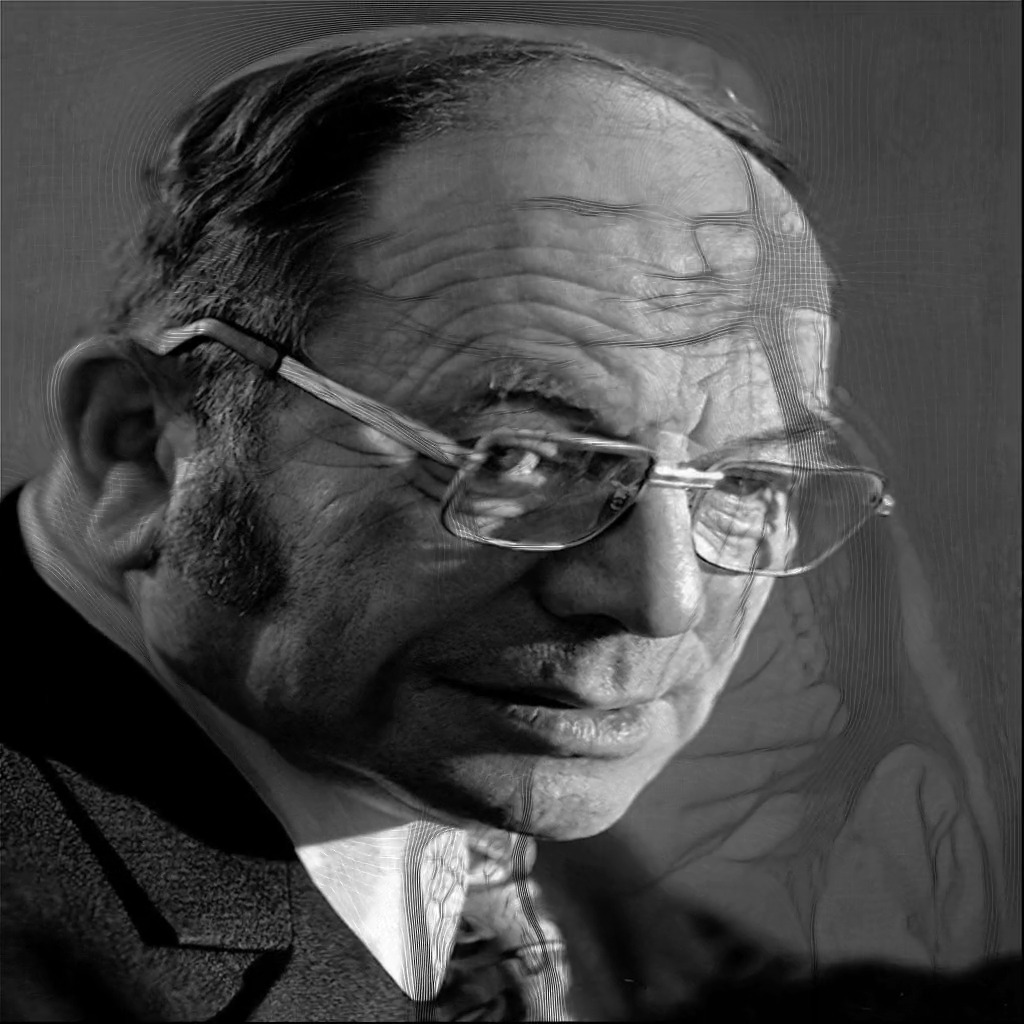}
\includegraphics[width=0.3 \textwidth]{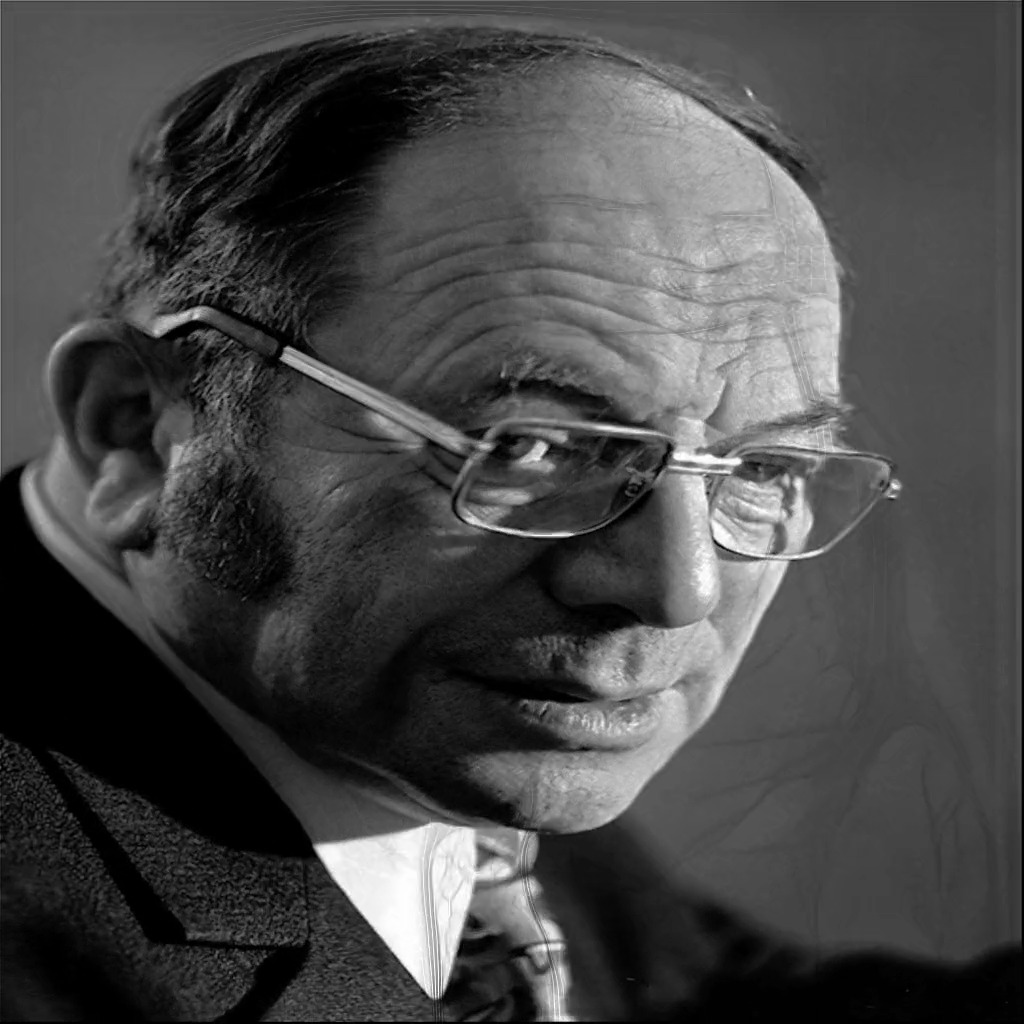}
\caption{Quadratic cost optimal transport between two $1024\times 1024$ pixel images.  The total computation time to compute the optimal map and then compute the interpolation was 20 seconds. \label{fig:monge_to_kantorovich} }
\end{figure}

\subsubsection*{Optimal transport with general costs}

In the final set of experiments, we consider optimal transport with more general cost functions.  As we noted earlier, the $c$-transform corresponding to any cost of the form $c(x,y)=\sum_{i=1}^d h_i(y_i-x_i)$ for $h_i$ strictly convex can be computed efficiently using the fast $c$-transform. The most interesting costs of this form are sums of $p^{th}$ powers, i.e.
\begin{equation}\label{eq:p_costs}c(x,y)=\sum_{i=1}^d \frac{1}{p_i}|y_i-x_i|^{p_i},\end{equation}
where $p_i>1$ for each coordinate index $i$.   These costs are particularly interesting when the powers $p_i$ are different for each $i$.  Indeed, in this case the optimal map can change considerably as the values of $p_i$ are varied.  

We illustrate this phenomenon in two dimensions by choosing $\mu$ to be the union of two discs of radius $\frac{1}{8}$ centered at $(\frac{1}{4}, \frac{1}{4})$ and $(\frac{3}{4}, \frac{3}{4})$ respectively, and $\nu$ to be the union of two discs of radius $\frac{1}{8}$ centered at $(\frac{1}{4}, \frac{3}{4})$ and $(\frac{3}{4}, \frac{1}{4})$ respectively:
\begin{center}
\begin{tikzpicture}[scale=1.5]
\fill[black] (0,0) rectangle (1,1);
\fill[white] (1/4,1/4) circle (1/8);
\fill[white] (3/4,3/4) circle (1/8);
\draw (1/2,-0.2) node {$\mu$};
\tikzset{shift={(1.2,0)}}
\fill[black] (0,0) rectangle (1,1);
\fill[white] (1/4,3/4) circle (1/8);
\fill[white] (3/4,1/4) circle (1/8);
\draw (1/2,-0.2) node {$\nu$};
\end{tikzpicture}
\end{center}
We then consider the optimal transport trajectory between $\mu$ and $\nu$ for several different choices of $p_1$ and $p_2$ (c.f. Figure \ref{fig:p_costs}). 

\begin{figure}
\includegraphics[width=0.225 \textwidth]{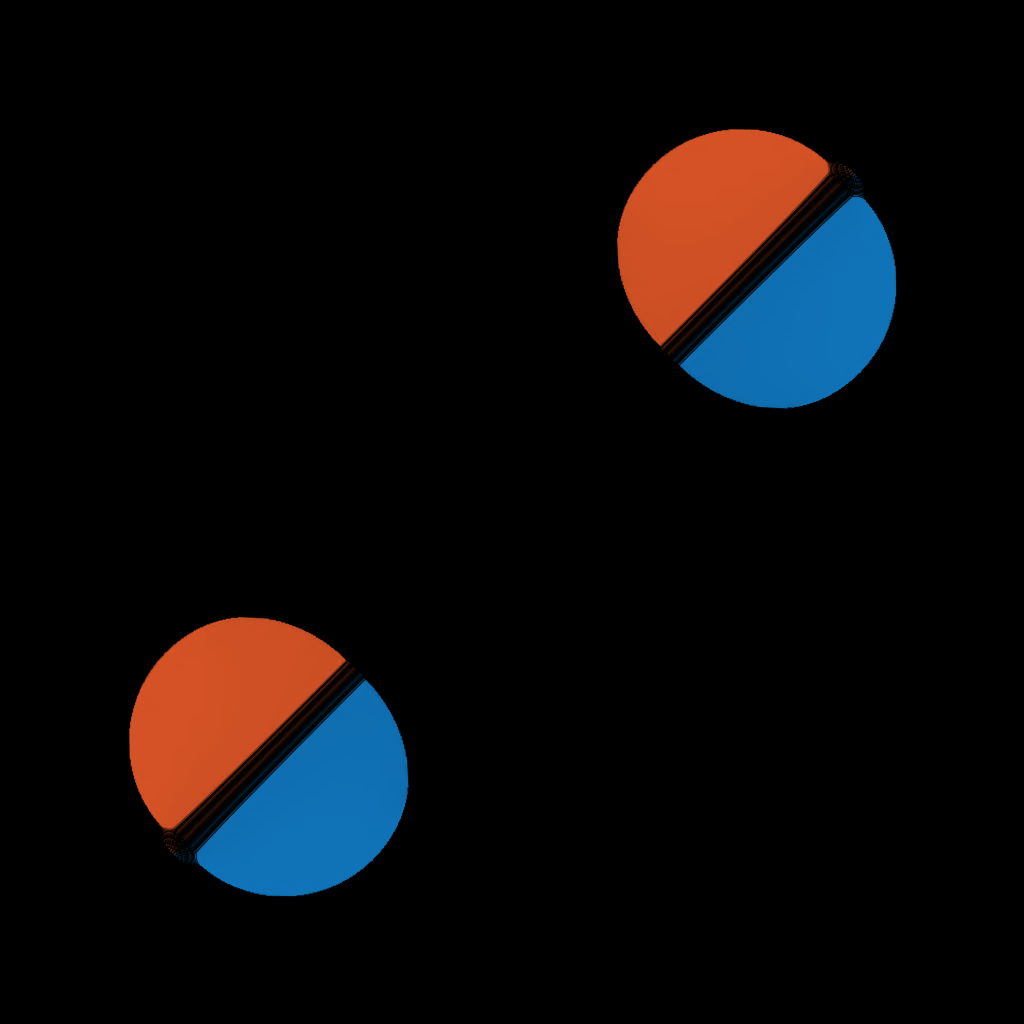}
\includegraphics[width=0.225 \textwidth]{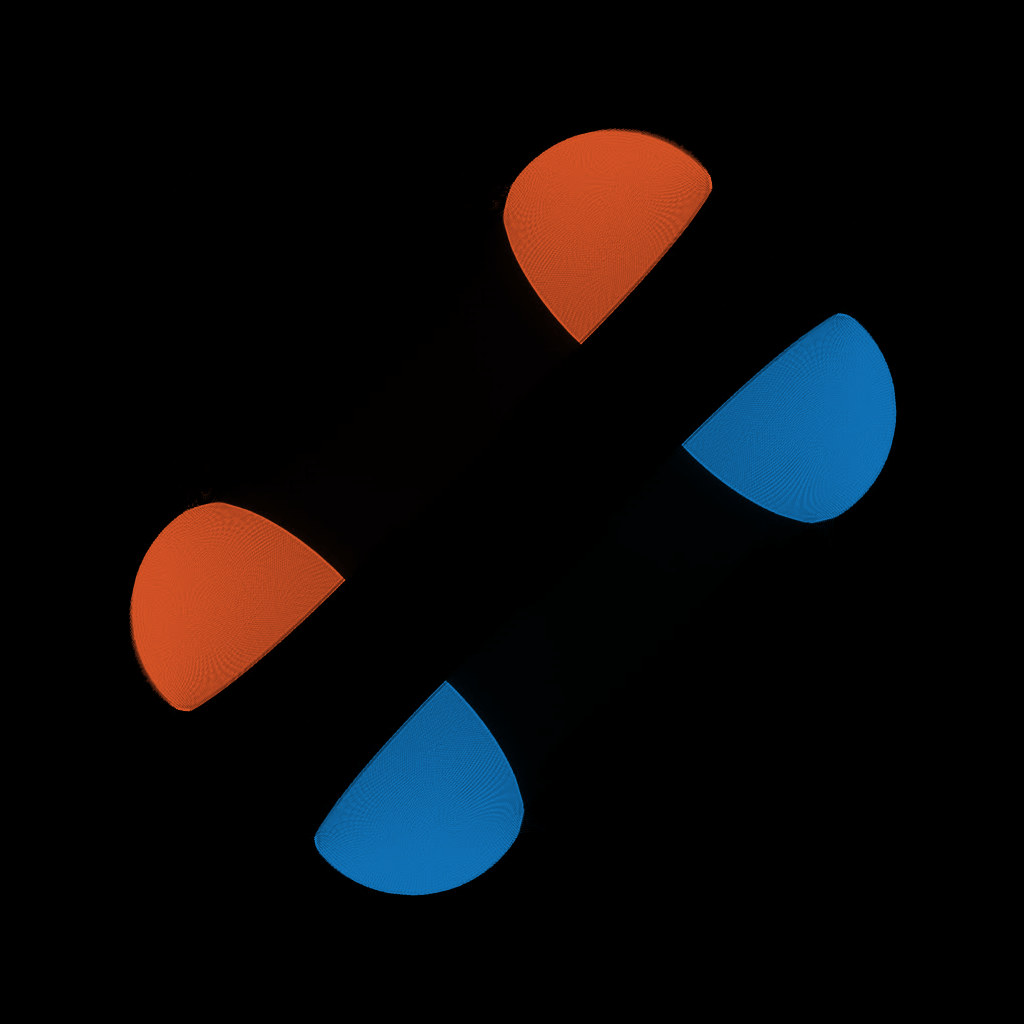}
\includegraphics[width=0.225 \textwidth]{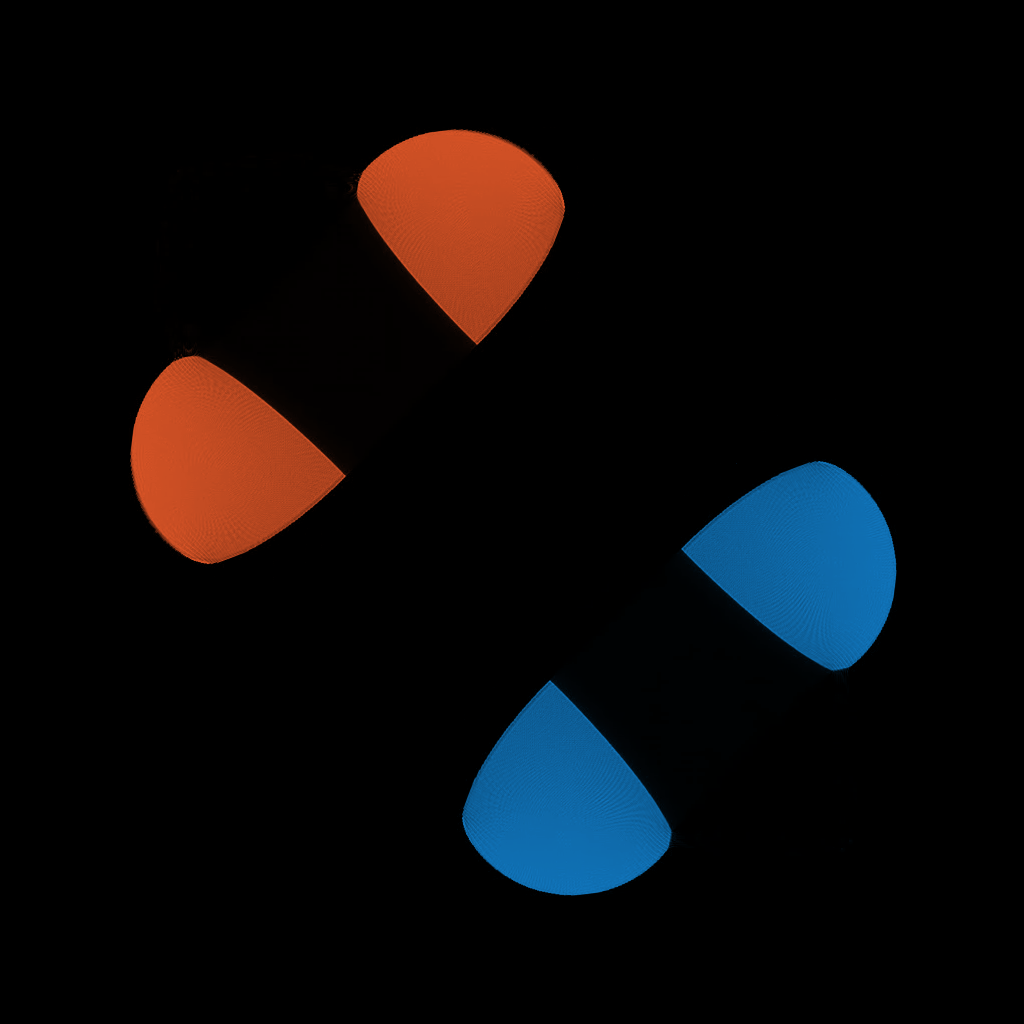}
\includegraphics[width=0.225 \textwidth]{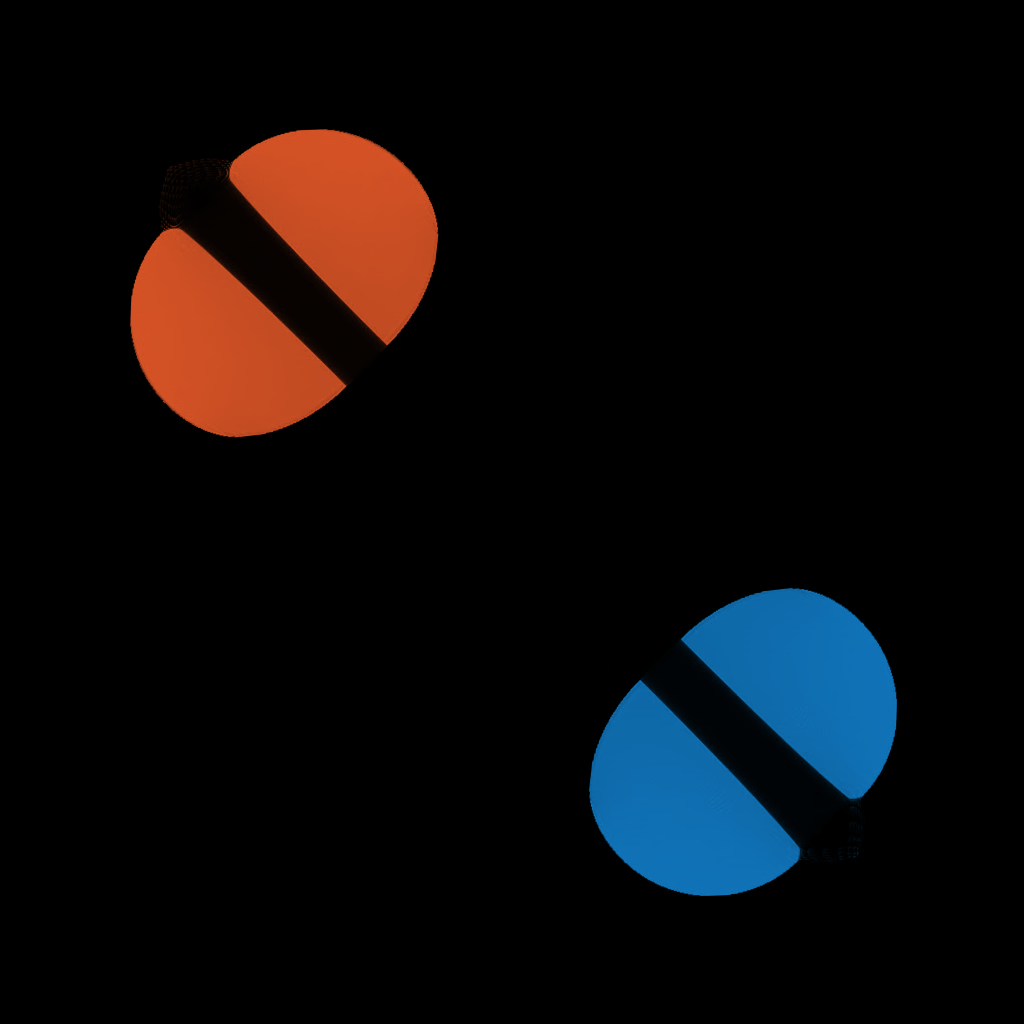}
\includegraphics[width=0.225 \textwidth]{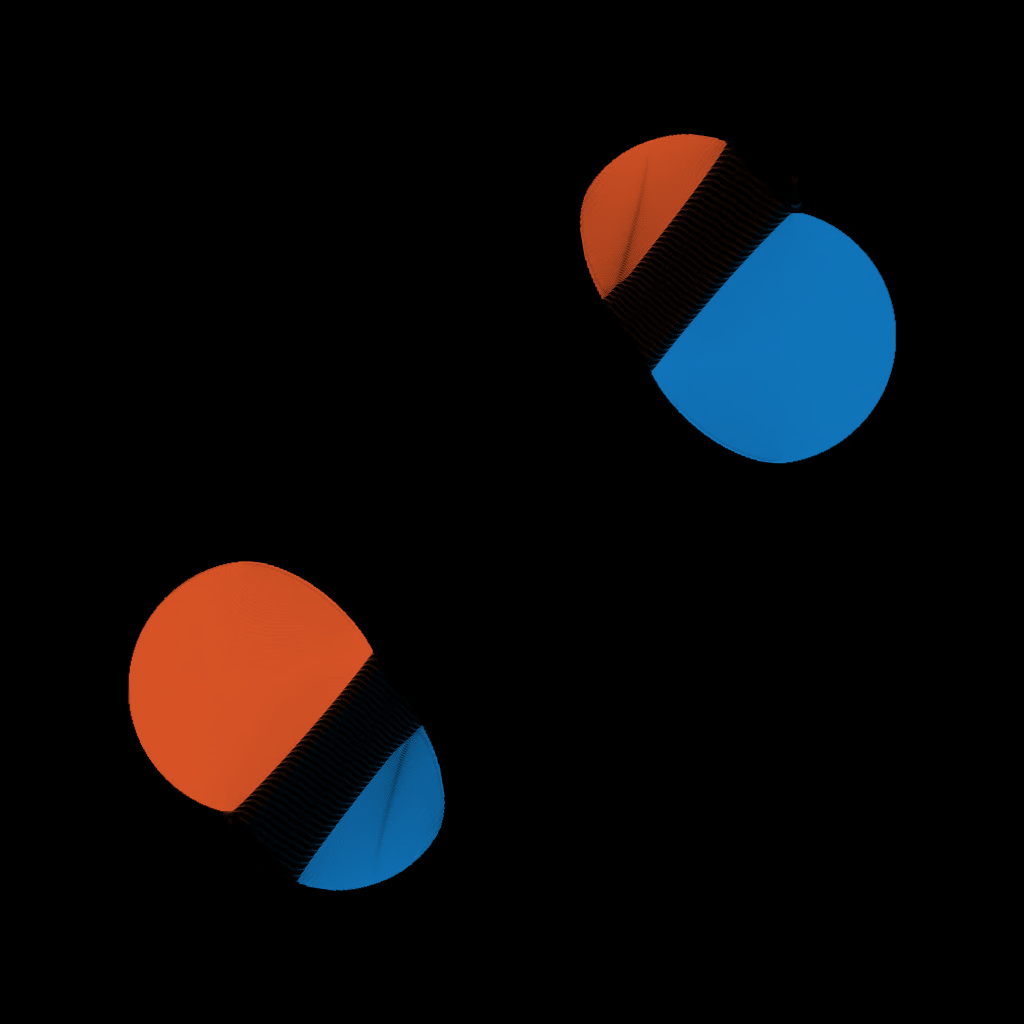}
\includegraphics[width=0.225 \textwidth]{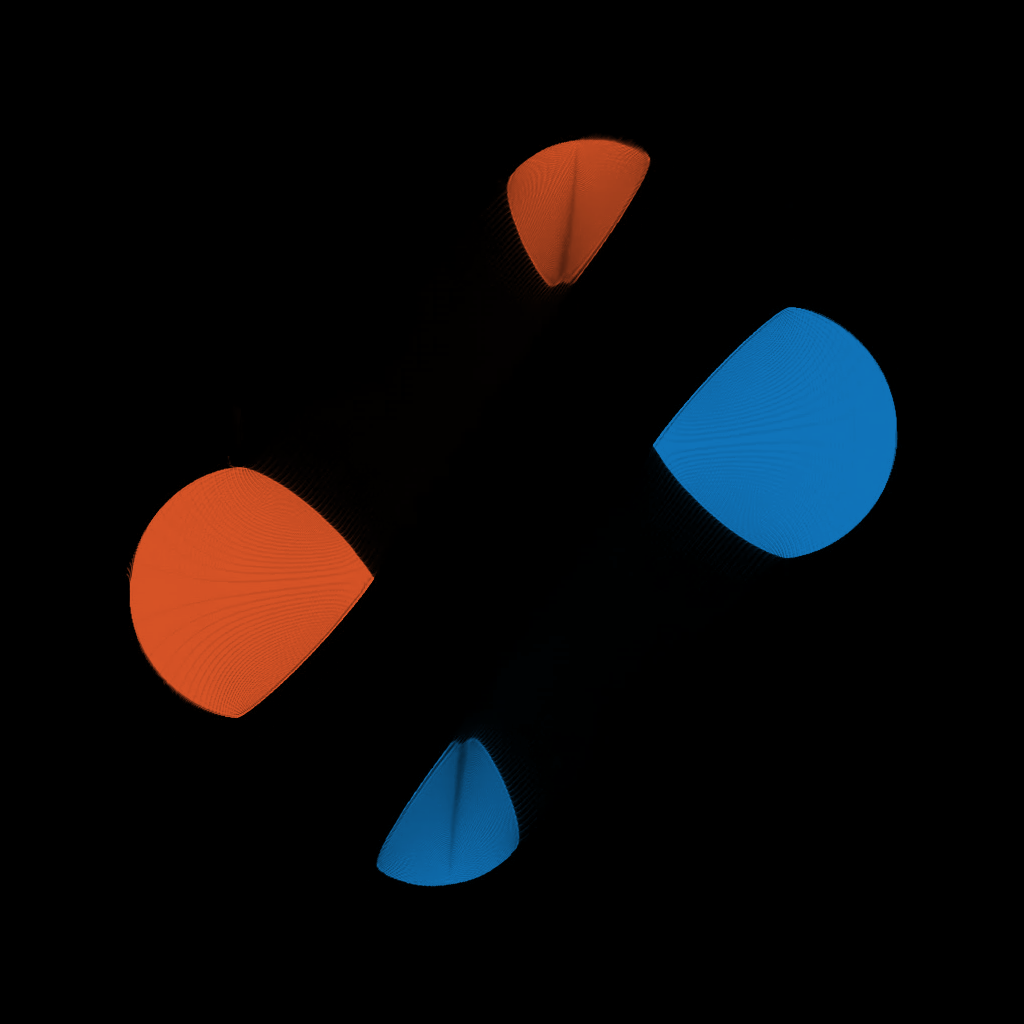}
\includegraphics[width=0.225 \textwidth]{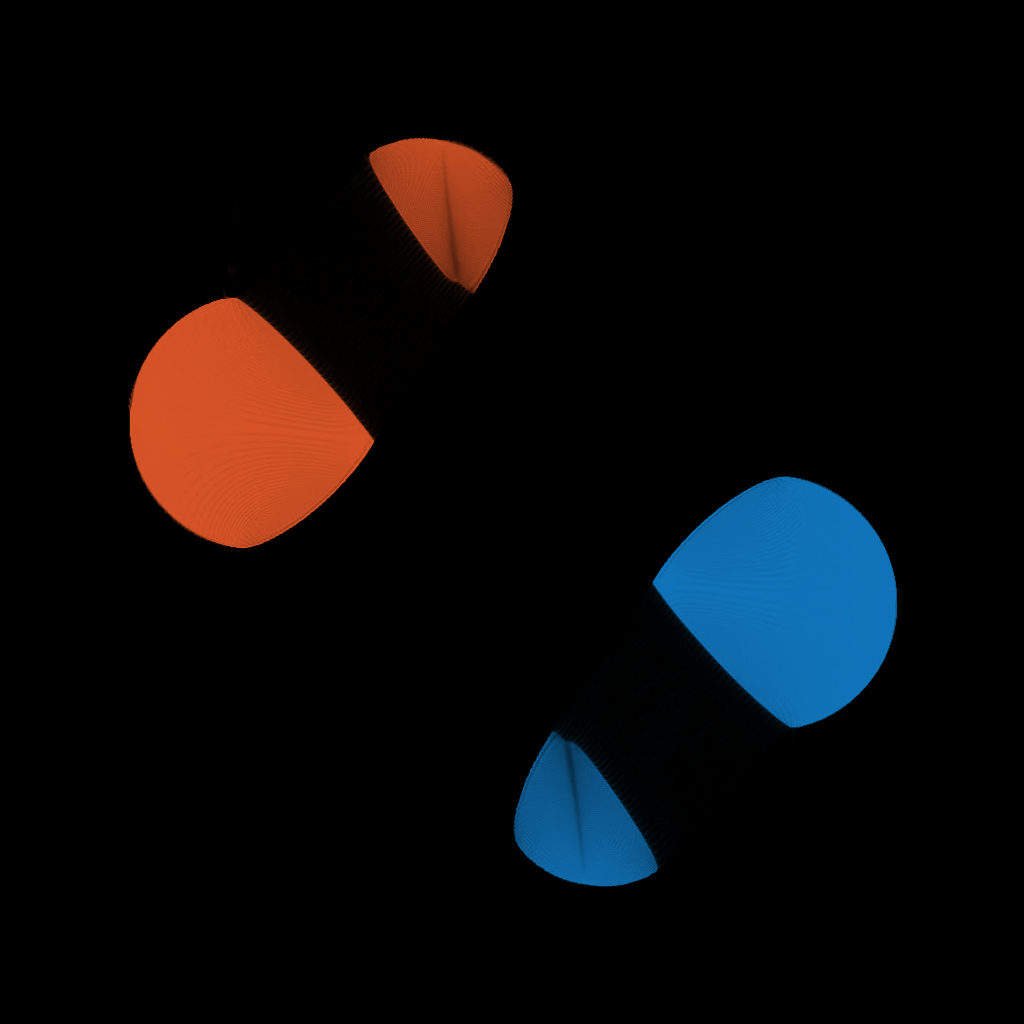}
\includegraphics[width=0.225 \textwidth]{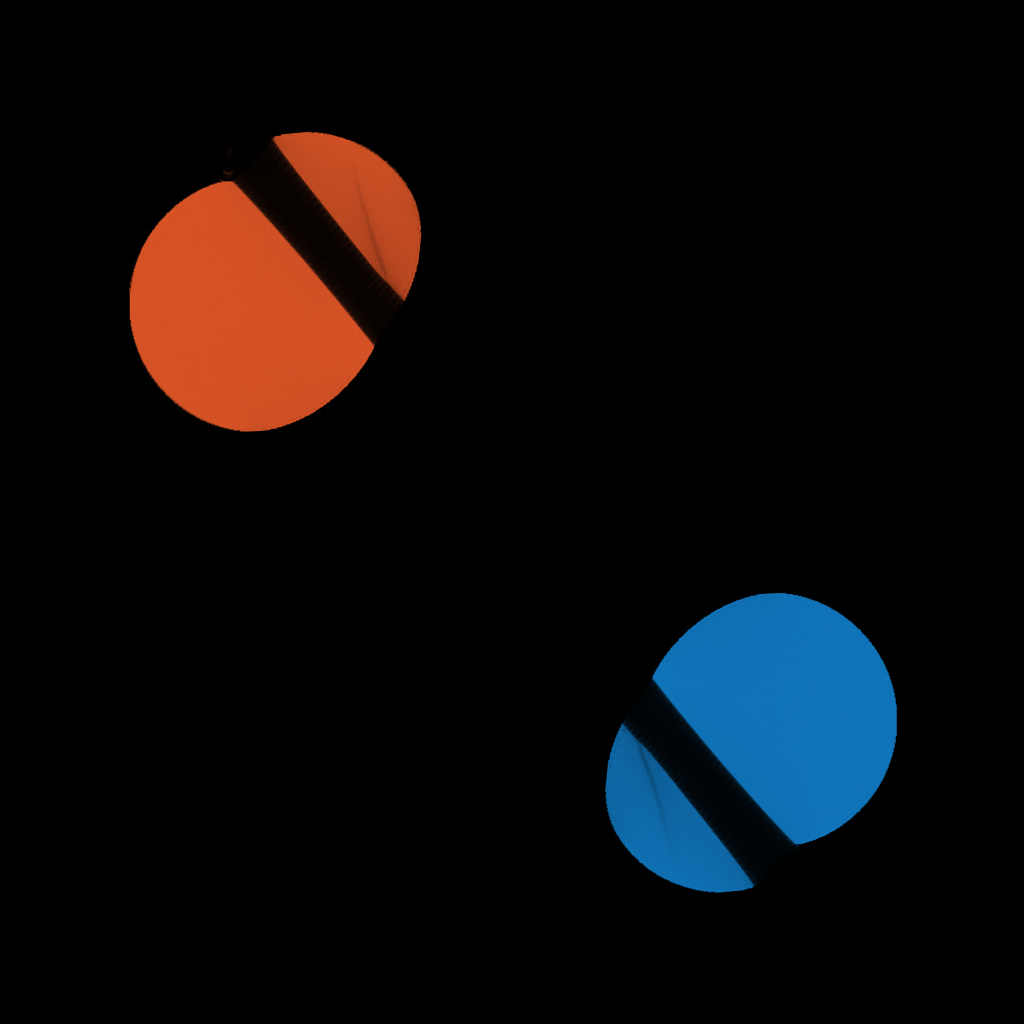}
\includegraphics[width=0.225 \textwidth]{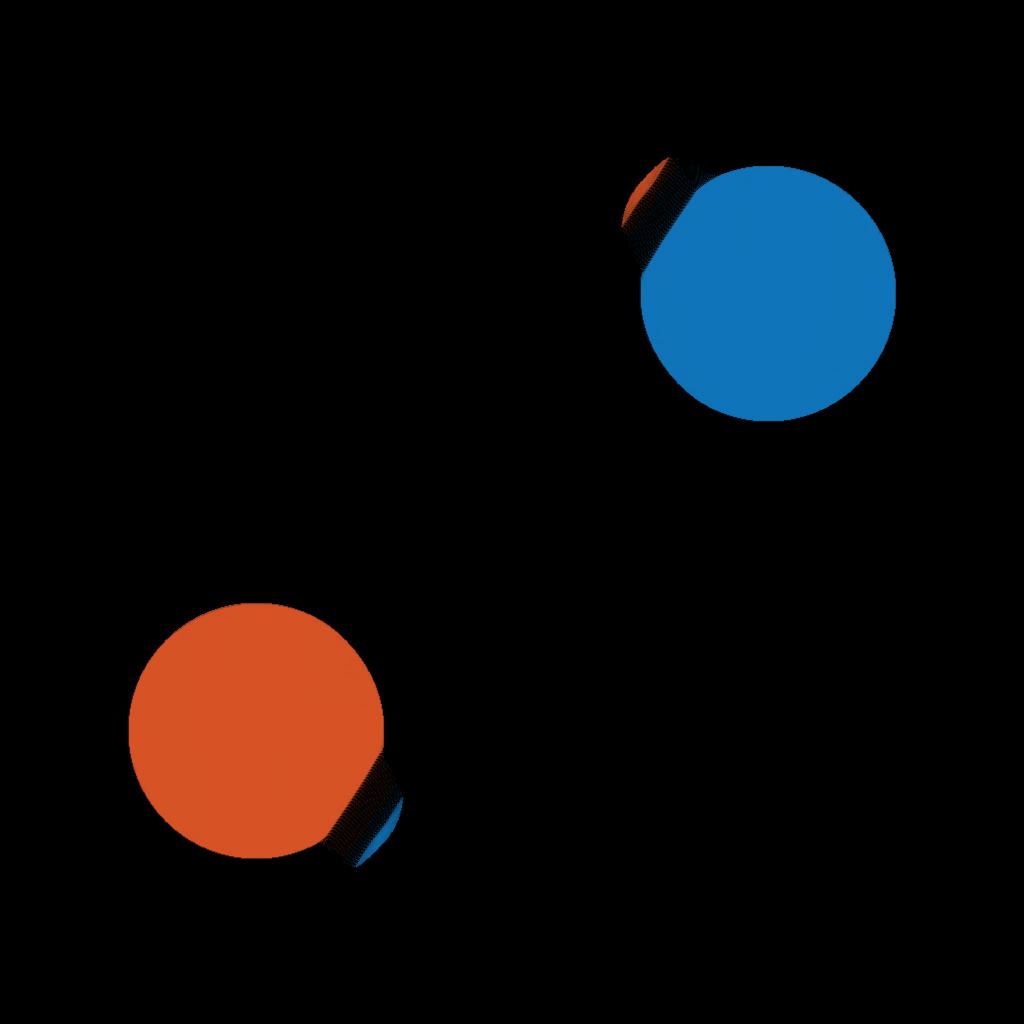}
\includegraphics[width=0.225 \textwidth]{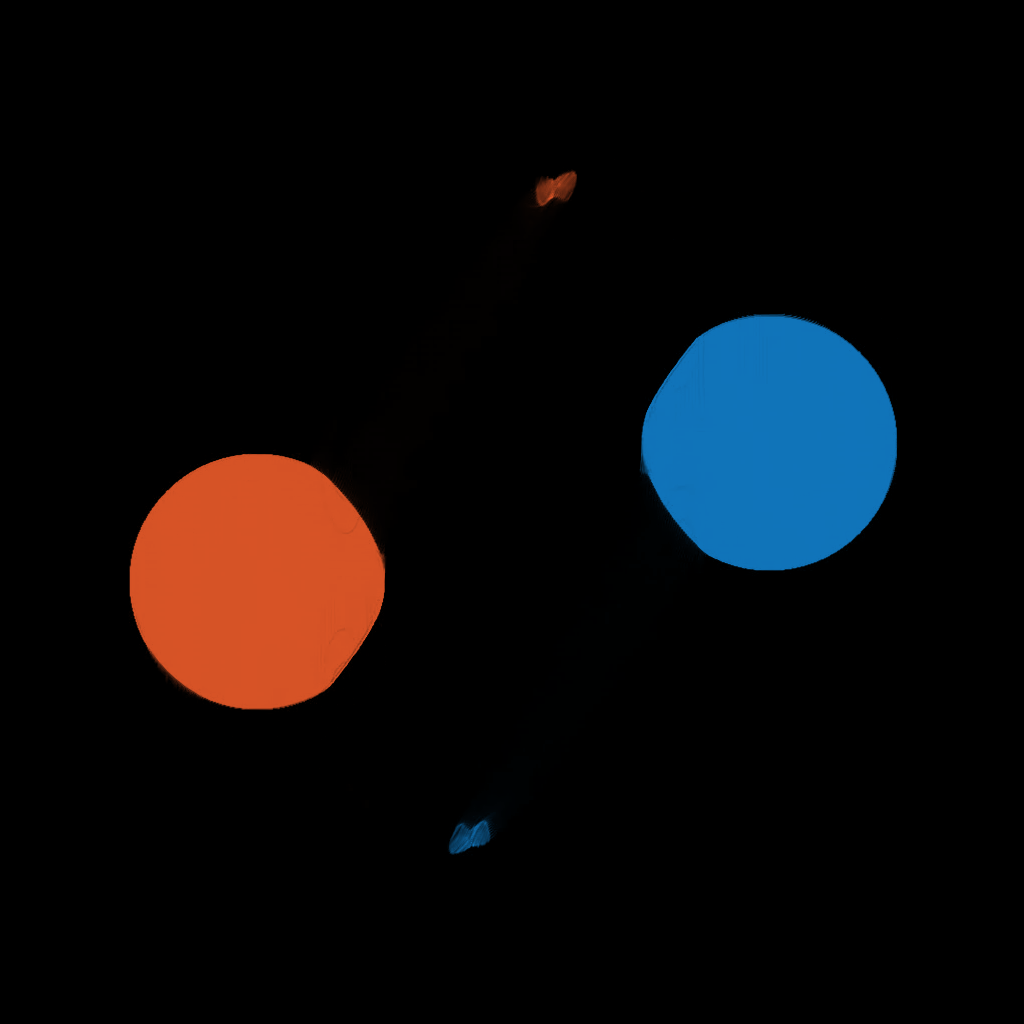}
\includegraphics[width=0.225 \textwidth]{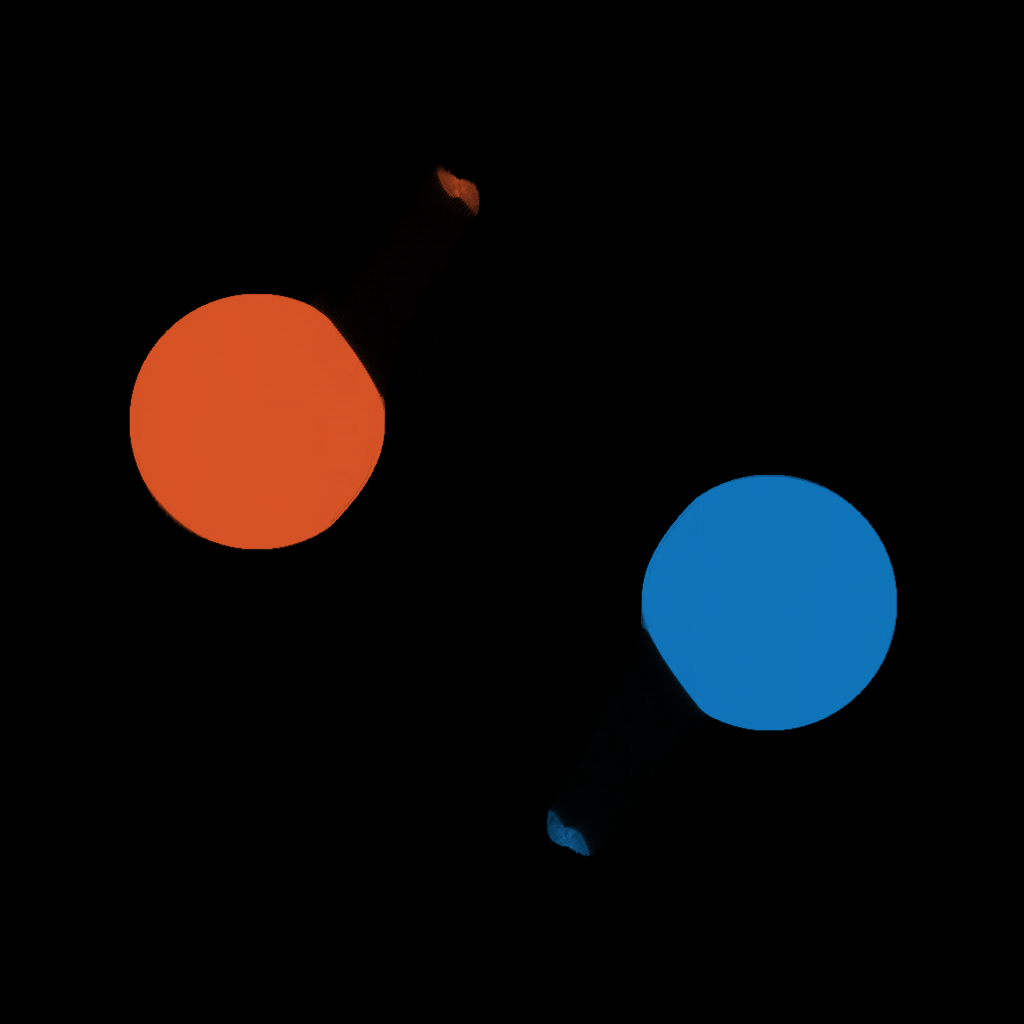}
\includegraphics[width=0.225 \textwidth]{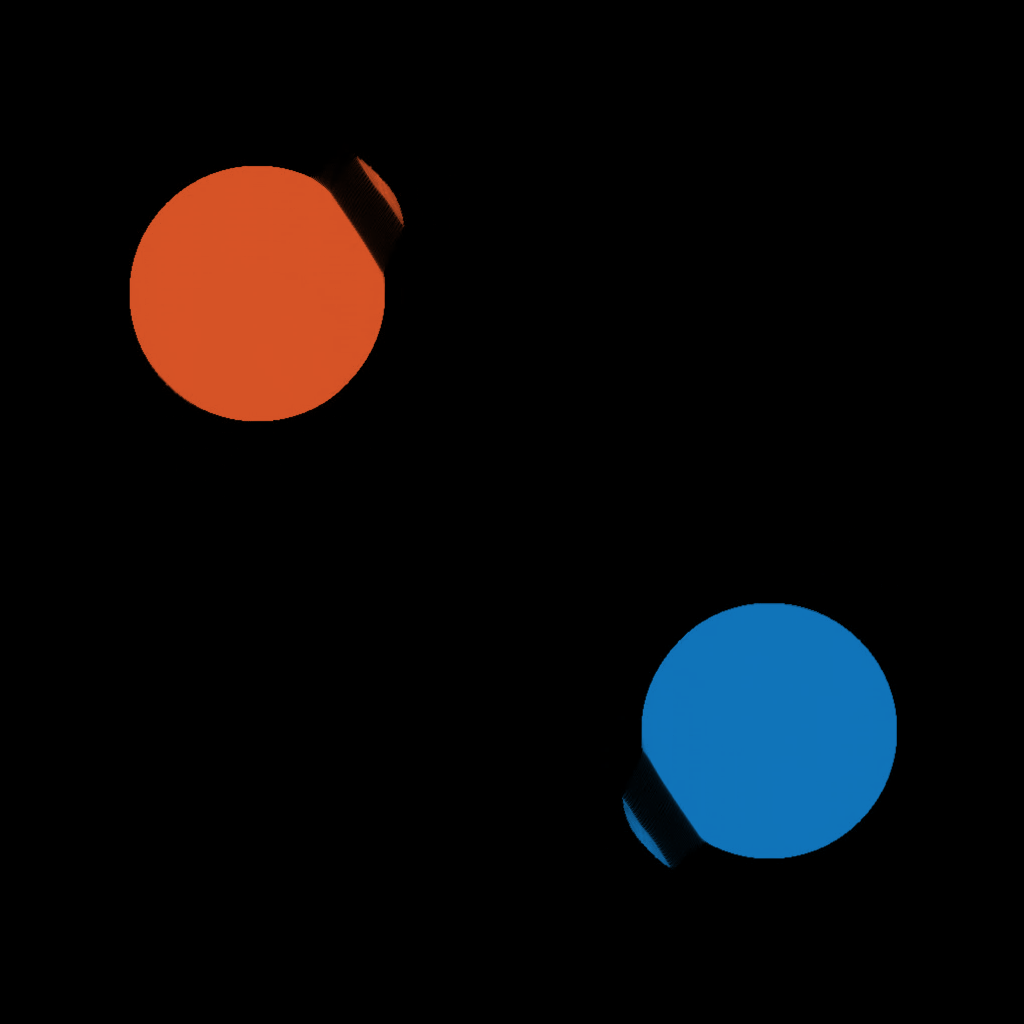}
\includegraphics[width=0.225 \textwidth]{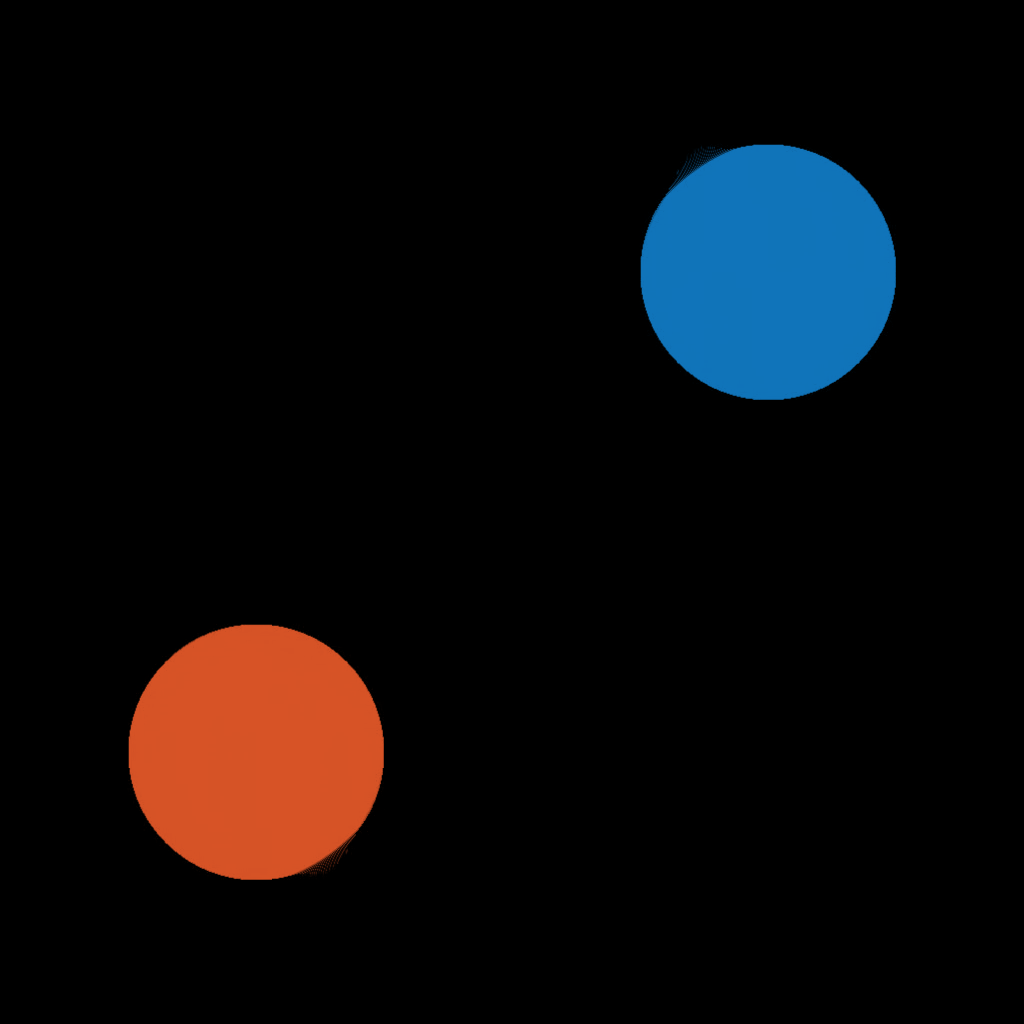}
\includegraphics[width=0.225 \textwidth]{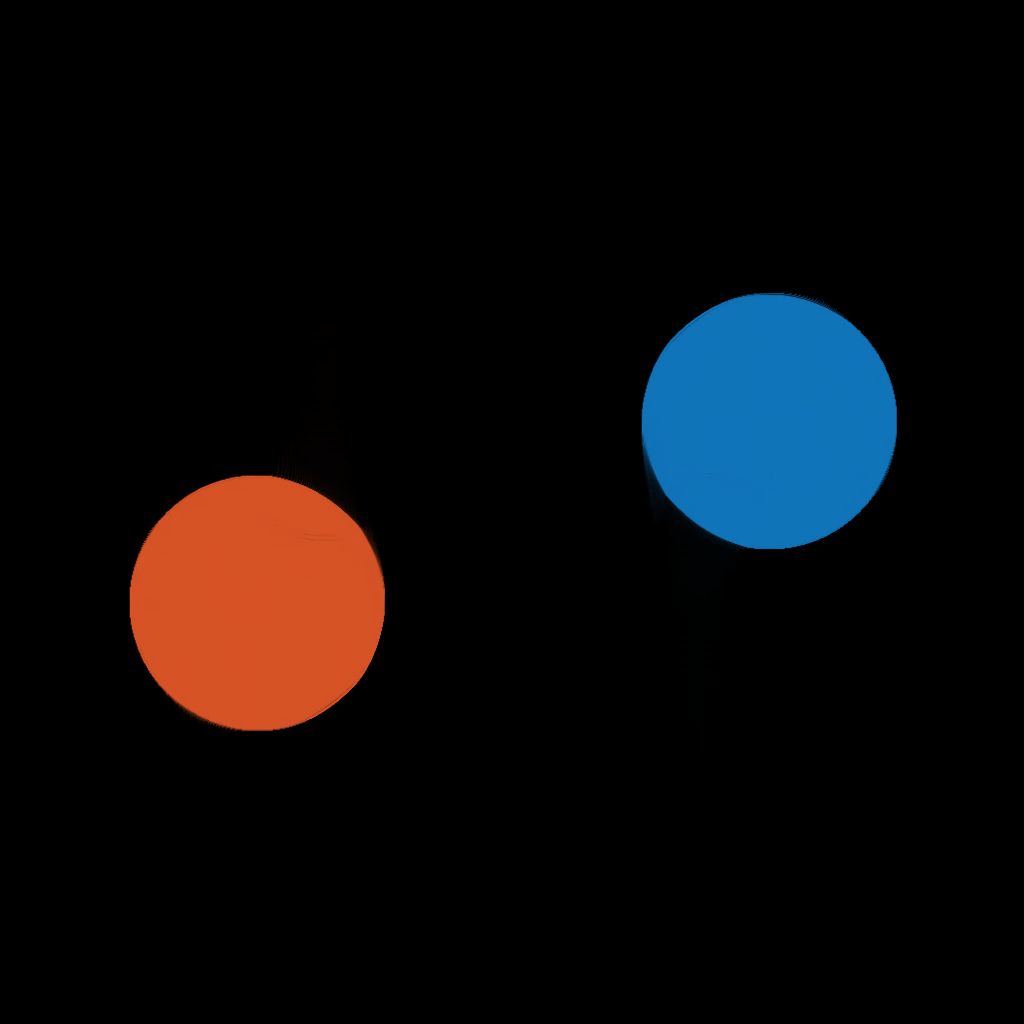}
\includegraphics[width=0.225 \textwidth]{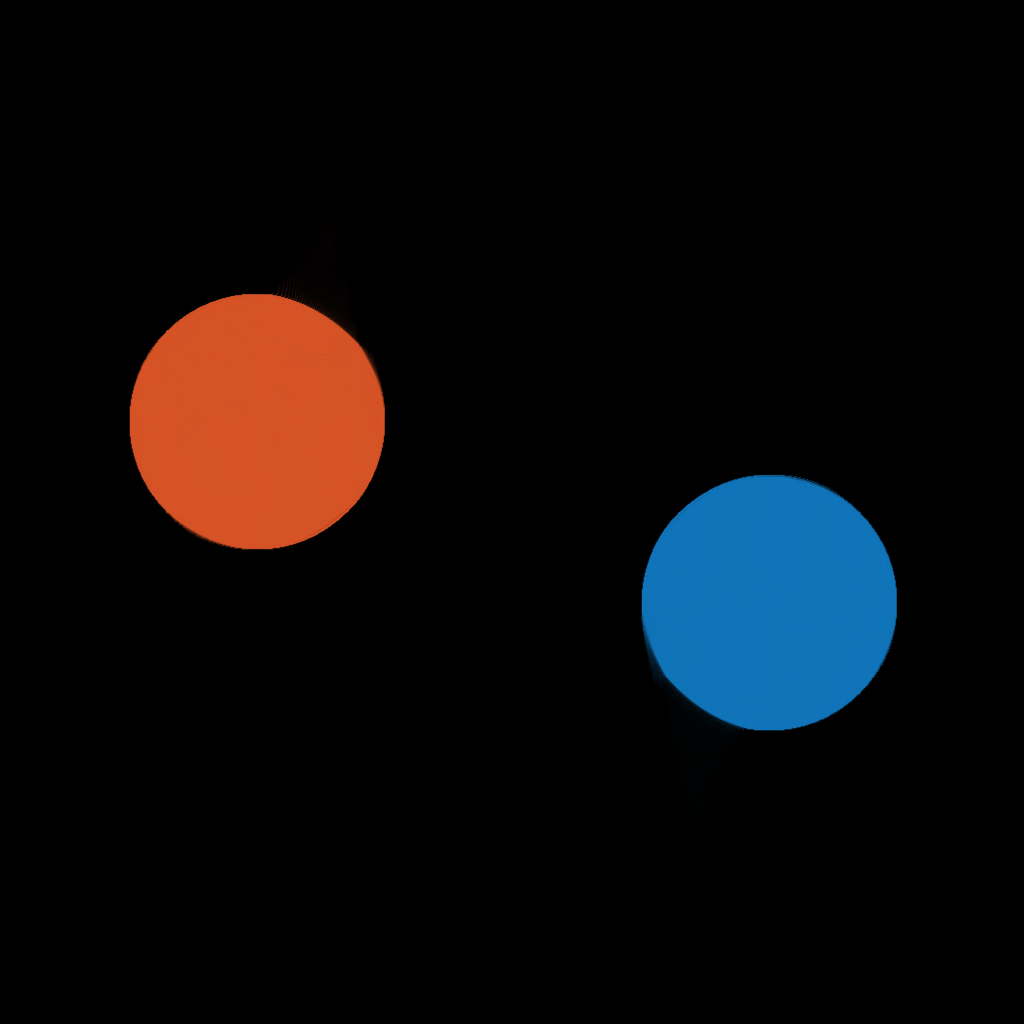}
\includegraphics[width=0.225 \textwidth]{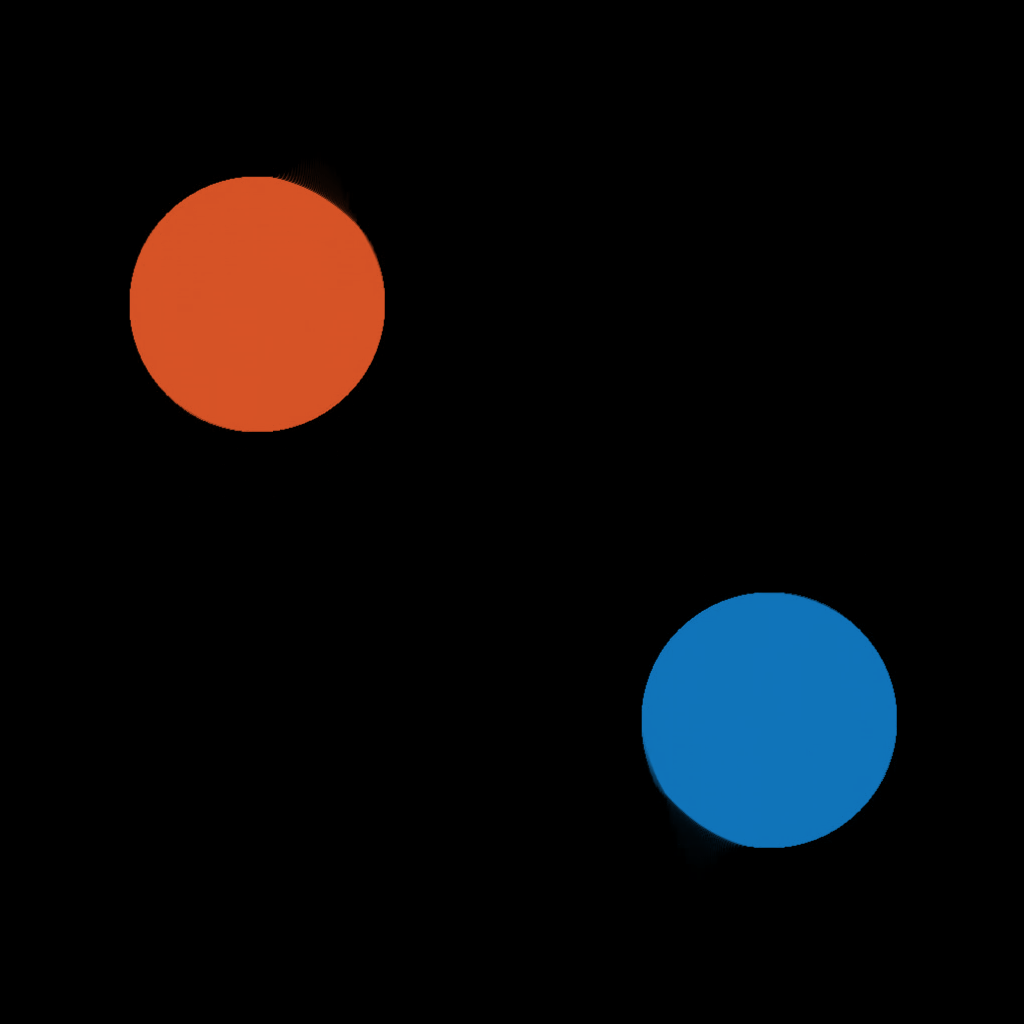}
\caption{Optimal transport with different cost functions.  Each row depicts the optimal trajectory for a cost of the form~\eqref{eq:p_costs}.  Row 1: $p_1=2$, $p_2=2$; Row 2: $p_1=1.75$, $p_2=2.25$; Row 3: $p_1=1.5$, $p_2=2.5$; Row 4: $p_1=1.1$, $p_2=3$. Colors represent an assignment: orange pixels move to the upper left disk while blue pixels move to the lower right disk. \label{fig:p_costs} }
\end{figure}

We start off with $p_1=p_2=2$.  In this case, each of the discs in $\mu$ splits evenly along the line $y=x$.  The portion of the discs that lie above the line are transported to the disc in $\nu$ centered at $(\frac{1}{4}, \frac{3}{4})$ and the parts below the line are transported to the disc in $\nu$ centered at $(\frac{3}{4}, \frac{1}{4})$. When $p_1<p_2$, it is cheaper to travel in the vertical direction as compared to the horizontal direction (for distances smaller than 1). As a result, when we decrease $p_1$ and increase $p_2$ the optimal map changes to transport more mass in the vertical direction and less mass in the horizontal direction.  When the values become sufficiently extreme, (for example $p_1=1.1$ and $p_2=3$), horizontal transport becomes so expensive that the optimal cost map moves mass in the vertical direction only.  See Figure \ref{fig:p_costs} for a visualization of the optimal trajectories for various choices of $p_1$ and $p_2$.

\appendix

\section{Proofs}

\begin{proof}[Proof of Theorem~\ref{thm:gradient_ascent}]
The ascent property follows from plugging in the choices $\phi=\phi_{n+1}$ and $\hat{\phi}=\phi_n$ into the given inequality. 

Now we suppose that $F$ has a unique maximizer $\phi_*$ and that the sequence $\{\phi_n\}_{n=0}^{\infty}$ lies in a bounded subset of $\mathcal{H}$. We can then extract a weak limit point $\tilde{\phi}$ of the sequence, and by weak upper semi-continuity $F(\tilde{\phi})\geq \lim_{n\to\infty} F(\phi_n)$ (note weak upper semi-continuity is automatic since we assume that $F$ is concave and real-valued).  The values $F(\phi_n)$ are bounded above by $F(\phi_*)$, therefore $\lim_{n\to\infty} \normH{\nabla_{\calH} F(\phi_n)}=0$.  
We can then establish the string of inequalities
$$F(\phi_*)\leq F(\phi_n)+\bracket{\nabla_{\calH} F(\phi_n), \phi_*-\phi_n}\leq F(\phi_n)+2R\normH{\nabla_{\calH} F(\phi_n)},$$
where the first inequality follows from the concavity of $F$ and the second follows from choosing $R=\max(\normH{\phi_*}, \sup_n \normH{\phi_n})$. 
Taking a limit on the right hand side we can conclude that $F(\tilde{\phi})\geq F(\phi_*)$, and thus, $\tilde{\phi}=\phi_*$ since the maximizer is unique.

\end{proof}

\begin{proof}[Proof of Fact~\ref{fact:derivative-J}]
Here we prove the result on $J$ only, the argument for $I$ being identical by symmetry.
Given a continuous function $u\colon\Omega\to \R$
we consider the perturbation
$$J(\phi+tu)=\int_{\Omega} (\phi+tu)^c(x)d\mu(x)+(\phi(x)+tu(x)) d\nu(x).$$
Using property~\ref{lem:c_transform_properties:derivative} of Lemma~\ref{lem:c_transform_properties} we see that
$$\lim_{t \to 0} \frac{J(\phi+tu)-J(\phi)}{t}=\int_{\Omega} u(x)\nu(x)-u(T_{\phi}(x))\mu(x).$$ 
Thus, 
$$\delta J_{\phi}(u)=\int_{\Omega} u\,(\nu-T_{\phi\, \#}\mu)$$
and the result then follows.  
\end{proof}

We now turn our attention to Proposition~\ref{prop:ascent-descent}. The notion of Bregman divergence will play an important role in its proof.

\begin{definition}
\label{def:Bregman-divergence}
Let $\calH$ be a Hilbert space and consider a Fréchet-differentiable function $F\colon\calH\to\Real$. The \emph{Bregman divergence} associated to $F$ is defined by
\[
F(\phi_2|\phi_1) = F(\phi_2) - F(\phi_1) - \delta F_{\phi_1}(\phi_2 - \phi_1),
\]
for all $\phi_1, \phi_2\in\calH$. Similarly, if $\Om$ is a closed and convex subset of $\Rd$ and $\xi\colon\Om\to\Real$ is a differentiable function, we define
\[
\xi(x_2|x_1) = \xi(x_2) - \xi(x_1) - \nabla \xi(x_1)\cdot (x_2 - x_1),
\]
for all $x_1, x_2\in\Om$. 
\end{definition}

As hinted by the above definition, we will make use of Bregman divergences both for Kantorovich potentials (defined on $\Om$) and dual functionals (defined on the Hilbert space $\dot{H}^1(\Om)$). The next results gathers properties which are well-known in the theory of Bregman divergences.

\begin{lemma}
\label{lemma:Bregman-divergences}
Let $\xi$ be a convex and differentiable function defined on $\Om$, and denote by $\xi^*$ its convex conjugate $\xi^*(p) = \sup_{x} p\cdot x - \xi(x)$. Then

\begin{enumerate}[(i)]
\item For any $\lambda \ge 0$, $\xi$ is $\lambda$-strongly convex $\iif \xi(x_2|x_1) \ge \frac{\lambda}{2} \abs{x_2-x_1}^2$.

\item $\xi(x_2|x_1) = \xi^*(p_1|p_2)$ for all $x_1, x_2\in\Om$, where we set $p_i = \nabla\xi(x_i)$. 

\item \label{lemma:Bregman-divergences:lower-to-upper} For any $\lambda > 0$, if $\xi$ is $\lambda$-strongly convex then $\xi(x_1|x_2) \le \frac{\lambda^{-1}}{2} \abs{p_2 - p_1}^2$, for all $x_1, x_2\in\Om$, where $p_i = \nabla\xi(x_i)$.

\item \label{lemma:Bregman-divergences:dual} Fix $x_0\in\Om$. Let $f(u)=\xi(x_0+u|x_0)$ and $g(v)=\xi^*(p_0+v|p_0)$, where $p_0=\nabla\xi(x_0)$. Then $f^*(v)=g(v)$. 

\end{enumerate}
\end{lemma}

Now, recall that that the dual functional $J$ is defined by $J(\phi) = \int \phi \,\nu + \int \phi^c\,\mu$. It is a concave functional, and note that the first term $\int\phi\,\nu$ is linear in $\phi$ and therefore plays no role in the convexity of $J$. We therefore set
\begin{equation}
\label{eq:def-F}
F(\phi) = -\int \phi^c\,\mu,
\end{equation}
a convex functional which essentially has the same convexity properties as $J$ since they only differ by a linear term. For instance one can check directly that for any potentials $\phi_1$ and $\phi_2$ we have 
\begin{equation}
\label{eq:Bregman-divergence-J-F}
J(\phi_2 | \phi_1) = -F(\phi_2|\phi_1) .
\end{equation}
Finally, the convex conjugate of $F$ is defined here by
\[
F^*(\rho) = \int_{\Om} \phi\,\rho - F(\phi). 
\]
We are at this point ready to state the next lemma which is at the core of obtaining stability estimates in the gradient schemes.

\begin{lemma}
\label{lemma:main}
Let $\phi_1$ and $\phi_2$ be two twice-differentiable real-valued functions defined on $\Om$ such that
\[
(1-\lambda^{-1})I \le D^2\phi_i(x) \le (1-\lambda)I,
\]
for $i=1,2$, for some $\lambda>0$. Moreover, let $\rho_i = T_{\phi_i\#}\mu$ the associated mass densities, for $i=1,2$. Then
\begin{equation}
\label{eq:upper-bound-divergence-F}
F(\phi_2 | \phi_1) \le \frac{1}{2}\normlinf{\mu} \,\lambda^{-(d+1)} \norm{\phi_2 - \phi_1}_{\dot{H}^1}^2.
\end{equation}
Additionally,
\begin{equation}
\label{eq:lower-bound-divergence-F}
F(\phi_2 | \phi_1) \ge \frac{1}{2}\normlinf{\mu}^{-1} \,\lambda^{d+1} \norm{\rho_2 - \rho_1}_{H^{-1}}^2.
\end{equation}
\end{lemma}

\begin{proof}
\textbf{Step 1.} We start the proof by establishing the equality
\begin{equation}
\label{eq:expression-Bregman-divergence-F}
F(\phi_2|\phi_1) = \int_{\Om} \Big[ \frac{1}{2}\babs{T_{\phi_1}(x) - T_{\phi_2}(x)}^2 - \phi_2\big(T_{\phi_1}(x)\big|T_{\phi_2}(x)\big) \Big] \, \mu(x) \, dx .
\end{equation}
To do so, first note that a simple variant of Fact~\ref{fact:derivative-J} tells us that the derivative of $F$ is precisely 
\begin{equation}
\label{eq:derivative-F}
\delta F_{\phi} = T_{\phi\#}\mu.
\end{equation}
Therefore the Bregman divergence associated to $F$ can be written 
\[
F(\phi_2|\phi_1) = \int_{\Om}\Big[-\phi_2^c(x) + \phi_1^c(x) - \phi_2\big(T_{\phi_1}(x)\big) + \phi_1\big(T_{\phi_1}(x)\big) \Big] \, \mu(x) \, dx .
\]
The above expression can be simplified using the identity
\[
-\phi^c(x) = \phi\big(T_{\phi}(x)\big) - \frac{1}{2}\babs{T_{\phi}(x)-x}^2,
\]
which follows immediately from the definitions of the $c$-transform and of $T_{\phi}$. Consequently,
\[
F(\phi_2|\phi_1) = \int_{\Om} \Big[ \frac{1}{2}\babs{T_{\phi_1}(x) - x}^2 - \frac{1}{2}\babs{T_{\phi_2}(x) - x}^2 + \phi_2\big(T_{\phi_2}(x)\big) - \phi_2\big(T_{\phi_1}(x)\big) \Big]\,\mu(x)\,dx.
\]
To proceed further we make use of another equality,
\begin{equation}
\label{eq:expression-T-gradient-phi}
\nabla\phi(T_{\phi}(x)) = T_{\phi}(x)-x,
\end{equation}
which is again a direct consequence of the $c$-transform definition. Adding the term $-\big(T_{\phi_2}(x) - x\big)\cdot\big(T_{\phi_1}(x) - T_{\phi_2}(x)\big) + \nabla\phi_2(T_{\phi_2}(x))\cdot \big(T_{\phi_1}(x) - T_{\phi_2}(x)\big) $ (which is thus always $0$) to the expression of $F(\phi_2|\phi_1)$ is enough to obtain the desired equality~\eqref{eq:expression-Bregman-divergence-F}.

\textbf{Step 2.} We now establish the inequality
\begin{equation}
\label{eq:inequality-Bregman-divergence-F}
F(\phi_2|\phi_1) \le \frac{\lambda^{-1}}{2} \int_{\Om} \abs{\nabla \phi_2(y) - \nabla\phi_1(y)}^2\, \rho_1(y)\,dy,
\end{equation}
where $\rho_1 = T_{\phi_1\#}\mu$. Set 
\[
\xi_2(x) = \frac{1}{2}\abs{x}^2 - \phi_2(x),
\]
which is a $\lambda$-strongly convex function because of the assumption made on $\phi_2$. Then, expression~\eqref{eq:expression-Bregman-divergence-F} can be written
\[
F(\phi_2|\phi_1) = \int_{\Om} \xi_2\big(T_{\phi_1}(x)\big|T_{\phi_2}(x)\big)\, \mu(x) \, dx .
\]
By Lemma~\ref{lemma:Bregman-divergences}\ref{lemma:Bregman-divergences:lower-to-upper} we derive the upper bound
\[
\xi_2(t_1|t_2) \le \frac{\lambda^{-1}}{2}\abs{\nabla \xi_2(t_1) - \nabla \xi_2(t_2)}^2,
\]
for any $t_1, t_2\in\Om$. Employing again identity~\eqref{eq:expression-T-gradient-phi} we see that on the one hand
\[
\nabla\xi_2\big(T_{\phi_1}(x)\big) = x + \nabla\phi_1\big(T_{\phi_1}(x)\big) - \nabla\phi_2\big(T_{\phi_1}(x)\big),
\]
and on the other hand $\nabla\xi_2\big(T_{\phi_2}(x)\big) = x$. We deduce the upper bound
\[
\xi_2\big(T_{\phi_1}(x)\big|T_{\phi_2}(x)\big) \le \frac{\lambda^{-1}}{2}\abs{\nabla \phi_1\big(T_{\phi_1}(x)\big) - \nabla\phi_2\big(T_{\phi_1}(x)\big)}^2,
\]
which after a change of measure implies the desired inequality~\eqref{eq:inequality-Bregman-divergence-F}.

\textbf{Step 3.} The next part of the proof consists of estimating the $L^{\infty}$ norm of $\rho_1$, 
\begin{equation}
\label{eq:Linf-estimate}
\normlinf{\rho_1} \le \lambda^{-d} \normlinf{\mu}.
\end{equation}
The strong convexity of $\phi_1$ implies that the map $T_{\phi_1}$ is injective.  
Since $\rho_1$ is the (density of) the pushforward measure $T_{\phi_1\#}\mu$, we have the change of variable formula
\[
\babs{\det DT_{\phi_1}(x)} \,\rho_1\big(T_{\phi_1}(x)\big) = \mu(x) ,
\]
which is valid for injective and differentiable maps. 
We now show that $DT_{\phi_1}$ is a symmetric positive-definite matrix whose eigenvalues are bounded below by $\lambda$. To do so, set
\[
\xi_1(x) = \frac{1}{2}\abs{x}^2 - \phi_1(x).
\]
By the assumption made on $\phi_1$ we have $D^2\xi_1(x) \le \lambda^{-1}$. Making use of identity~\eqref{eq:expression-T-gradient-phi}, we can write $x = \nabla\xi_1\big(T_{\phi_1}(x)\big)$, which can be inverted into $\nabla\xi_1^*(x) = T_{\phi_1}(x)$. Since the Hessian of $\xi_1$ has eigenvalues bounded above by $\lambda^{-1}$, its convex conjugate is $\lambda$-strongly convex. Therefore \[
\det DT_{\phi_1}(x) \ge \lambda^{d},
\]
for any $x\in\Om$. The bound~\eqref{eq:Linf-estimate} directly follows, and combining it with~\eqref{eq:inequality-Bregman-divergence-F} concludes the proof of~\eqref{eq:upper-bound-divergence-F}. 

\textbf{Step 4.} Now that the upper bound~\eqref{eq:upper-bound-divergence-F} is proven, it directly implies the lower bound~\eqref{eq:lower-bound-divergence-F}. Indeed, rewrite~\eqref{eq:upper-bound-divergence-F} as
\[
F(\phi_1 + h|\phi_1) \le \frac{C}{2}\normhone{h}^2,
\]
with $h=\phi_2-\phi_1$ and $C = \normlinf{\mu} \lambda^{-(d+1)}$. Taking convex conjugation on both sides, which reverses the sign of inequalities, yields
\[
F(\phi_1 + \cdot\,|\phi_1)^*(u) \ge \frac{C^{-1}}{2}\normhmone{u}^2.
\]
Indeed it is easy to see that $\frac{1}{2}\normhone{\cdot}^2$ and $\frac{1}{2}\normhmone{\cdot}^2$ are conjugate to each other.
We then apply the Bregman divergence property~\ref{lemma:Bregman-divergences:dual} from Lemma~\ref{lemma:Bregman-divergences} to the functional $F$ and obtain
\[
F^*(\rho_1 + u|\rho_1) \ge \frac{C^{-1}}{2}\normhmone{u}^2,
\]
where $\rho_1 = \delta F_{\phi_1} = T_{\phi_1\#}\mu$. This is precisely the desired inequality when $u=\rho_2-\rho_1$. 
\end{proof}

We have now at our disposition all the necessary tools to prove Proposition~\ref{prop:ascent-descent}. 

\begin{proof}[Proof of Proposition~\ref{prop:ascent-descent}]
\textbf{Ascent property.} Consider two consecutive iterates $\phi_n$ and $\phi_{n+1}$. By Fact~\ref{fact:derivative-J} the derivative of $J$ takes the form $\delta J_{\phi} = \nu - T_{\phi\#}\mu$, which implies that
\[
J(\phi_{n+1}) - J(\phi_n) = J(\phi_{n+1}|\phi_n) + \int_{\Om} (\phi_{n+1} - \phi_n)(\nu - T_{\phi_n\#}\mu).
\]
Here we introduced a \emph{Bregman divergence} $J(\cdot|\cdot)$ (see Definition~\ref{def:Bregman-divergence}  in this appendix). Let $F$ be defined by~\eqref{eq:def-F}, then as previously noted in~\eqref{eq:Bregman-divergence-J-F} we have $J(\phi_{n+1}|\phi_n) = -F(\phi_{n+1}|\phi_n)$. Additionally the gradient step $\phi_{n+1} - \phi_n = \sigma (-\laplacian)^{-1}(\nu - T_{\phi_n\#}\mu)$ implies that
\[
\int_{\Om} (\phi_{n+1} - \phi_n)(\nu - T_{\phi_n\#}\mu) = \sigma^{-1}\normhone{\phi_{n+1} - \phi_n}^2.
\]
Combining these two expressions we arrive at
\[
J(\phi_{n+1}) - J(\phi_n) = -F(\phi_{n+1}|\phi_n) + \sigma^{-1}\normhone{\phi_{n+1} - \phi_n}^2,
\]
and the choice $\sigma = \frac{\lambda^{d+1}}{\normlinf{\mu}}$ together with Lemma~\ref{lemma:main} stated and proven right above is enough to obtain the ascent property 
\[
J(\phi_{n+1}) - J(\phi_n) \ge \frac{1}{2}\normlinf{\mu}\lambda^{-(d+1)} \normhone{\phi_{n+1}-\phi_n}^2 .
\]

\textbf{Decrease in $H^{-1}$ norm.} A simple computation, expanding and rearranging quadratic terms, shows that
\[
\norm{\rho_{n+1} - \nu}_{H^{-1}}^2 - \norm{\rho_{n} - \nu}_{H^{-1}}^2 = \norm{\rho_{n+1} - \rho_n}_{H^{-1}}^2 - 2\int_{\Om} (-\laplacian)^{-1}(\nu - \rho_n)(\rho_{n+1} - \rho_n).
\]
Since the iterate $\phi_{n+1}$ is defined by $\phi_{n+1} = \phi_n + \sigma (-\laplacian)^{-1}(\nu - \rho_n)$,we obtain
\[
\norm{\rho_{n+1} - \nu}_{H^{-1}}^2 - \norm{\rho_{n} - \nu}_{H^{-1}}^2 = \norm{\rho_{n+1} - \rho_n}_{H^{-1}}^2 - 2\sigma^{-1} \int_{\Om} (\phi_{n+1} - \phi_n)(\rho_{n+1} - \rho_n) .
\]
Next, as was explained by~\eqref{eq:derivative-F} the derivative of the functional $F$ defined by~\eqref{eq:def-F} is $\delta F_{\phi} = \rho$, with $\rho = T_{\phi\#}\mu$, and since $F$ is a convex functional this relation can be inverted into
\[
\phi = \delta F^*_{\!\rho},
\]
where $F^*$ denotes the convex conjugate of $F$ defined by $F^*(\rho) = \sup_{\phi} \int \phi\,\rho - F(\phi)$. Thus we can write
\begin{align*}
 & \norm{\rho_{n+1} - \nu}_{H^{-1}}^2 - \norm{\rho_{n} - \nu}_{H^{-1}}^2 \\ 
 =& \norm{\rho_{n+1} - \rho_n}_{H^{-1}}^2 - 2\sigma^{-1} \big(\delta F^*_{\rho_{n+1}} - \delta F^*_{\rho_n}\big)(\rho_{n+1} - \rho_n) \\ 
 =& \norm{\rho_{n+1} - \rho_n}_{H^{-1}}^2 - 2\sigma^{-1}  \big(F^*(\rho_{n+1}|\rho_n) + F^*(\rho_n|\rho_{n+1})\big) \\
 =& -2\sigma^{-1}\Big[ F^*(\rho_{n+1}|\rho_n) + F^*(\rho_n|\rho_{n+1}) - \frac{\sigma}{2}\norm{\rho_{n+1} - \rho_n}_{H^{-1}}^2 \Big] .
\end{align*}
By Lemma~\ref{lemma:main} we have $F^*(\rho_{n+1}|\rho_n) + F^*(\rho_n|\rho_{n+1}) \ge \sigma \norm{\rho_{n+1} - \rho_n}_{H^{-1}}^2$, which is enough to conclude the proof.
\end{proof}

\begin{proof}[Proof of Proposition~\ref{prop:ascent-back-and-forth}]
Fix an iterate $\phi_n$. Then $J(\phi_n) = D(\phi_n,\phi_n^c)$. By Proposition~\ref{prop:ascent-descent} we have
\[
J(\phi_{n+\frac{1}{2}}) \ge J(\phi_n),
\]
since the step $\phi_n\to\phi_{n+\frac{1}{2}}$ is a gradient ascent step of $J$; also note that $J(\phi_{n+\frac{1}{2}}) = D(\phi_{n+\frac{1}{2}}, (\phi_{n+\frac{1}{2}})^c) = D(\phi_{n+\frac{1}{2}},\psi_{n+\frac{1}{2}})$. Next, it is easy to see that the value of $D$ always increases (or plateaus) when taking $c$-transforms, thus
\[
D\big((\psi_{n+\frac{1}{2}})^c,\psi_{n+\frac{1}{2}}\big) \ge D(\phi_{n+\frac{1}{2}},\psi_{n+\frac{1}{2}}).
\]
Therefore, we have shown that $I(\psi_{n+\frac{1}{2}}) \ge J(\phi_{n+\frac{1}{2}})$; the remaining inequalities can be treated in a similar fashion.
\end{proof}

\bibliographystyle{amsalphaurl}
\bibliography{fast-ot}

\end{document}